\documentclass[10pt,a4paper]{article}

\usepackage{amsmath,amssymb,latexsym,color,dsfont,mleftright}
 \usepackage{authblk,mathtools, nccmath,subcaption}
 \usepackage{graphicx,float}
 \usepackage[square,numbers,comma]{natbib}
 \usepackage{tikz-cd}
  \usepackage{bbm}
  \usepackage{cleveref}
\usepackage{ntheorem} 
\usepackage[colorinlistoftodos]{todonotes}
\usepackage{CJKutf8}
\usepackage{float}
\usepackage[export]{adjustbox}
\usepackage[T1]{fontenc}
\usepackage{esint}
\usepackage{mathtools}

\newtheorem{theorem}{Theorem}
 
\newtheorem{conjecture}{Conjecture}
\newtheorem{lemma}{Lemma}
\newtheorem{proposition}{Proposition}
\newtheorem{definition}{Definition}
\newtheorem{coro}{Corollary}
\newtheorem{nota}{Notation}

\newtheorem{example}{Example}
\newtheorem{problem}{Problem}
\newtheorem{remark}{Remark}

\newtheorem{THEO}{Theorem}

\newcommand{\lqqd}{\hfill{$\Box$}\bigskip}
\newcommand{\CC}{\mathbb{C}}

\newcommand{\NN}{\mathbb{N}}

\newcommand{\RR}{\mathbb{R}}
\newcommand{\ZZ}{\mathbb{Z}}
\newcommand{\QQ}{\mathbb{Q}}
\newcommand{\LL}{\mathcal{L}}
\newcommand{\MM}{\mathcal{M}}

\newcommand{\supp}[1]{\mbox{{ \textrm{supp}\/}}[#1]}

\newcommand{\dsty}{\displaystyle}
\newcommand{\proof}{\bigskip\noindent {\sc Proof.  }}

\def\bbbuildrel#1_#2{\mathrel{
\mathop{\kern 0pt#1}\limits_{#2}}}

\title{Semiclassical expansion for  exactly solvable differential operators}

\author[1]{Jorge A. Borrego-Morell}
\author[2]{Boris Shapiro}
\affil[1]{Department of Mathematics, Universidade Federal do Rio de Janeiro, Campus Duque de Caxias, Brazil.  ORCID: 0000-0001-8871-2739}
\affil[2]{ Department of Mathematics, Stockholm University, SE-106 91
Stockholm,  Sweden. ORCID: 0000-0002-8438-3971}
\begin{document}

%%% ----------------------------------------------------------------------
\maketitle
%%% ----------------------------------------------------------------------

\begin{abstract}

We investigate the linear differential equation $\mathcal{M}(v(z,\eta))=\eta^M{v(z,\eta)}$, where $\eta>0$ is a large spectral parameter. The operator $\mathcal{M}=\sum_{k=1}^{M}\rho_{k}(z)\frac{d^k}{dz^k}$ ($M\ge 2$) has polynomial coefficients where the leading coefficient $\rho_M(z)$ is a monic complex-valued polynomial of degree $M$, and $\deg \rho_k \leq k$ for $k < M$. We prove the Borel summability of its WKB solutions in the Stokes regions. For $M=3$, assuming $\rho_M$ has simple zeros, we provide a complete description of the Stokes complex (the union of all Stokes curves). Finally, we demonstrate that for Euler-Cauchy equations, the WKB solutions converge in the standard sense.

\end{abstract}

\noindent {\it MSC2020 classes:}   Primary 34M60,  34E20, 34M40, 34M30, 34M35.\\


\section{Introduction}

In mathematical physics, a linear differential operator 
\begin{equation}\label{eq:ES} 
 \mathcal{M}=\sum_{k=1}^{M}\rho_{k}(z)\frac{d^k}{dz^k}
 \end{equation}
with complex polynomial coefficients is termed \textcolor{blue}{\emph{exactly solvable}} if:   
\begin{itemize}
    \item[(i)] $\deg \rho_k \leq k$ for $1 \leq k \leq M$;
    \item[(ii)] there exists at least one $1 \leq \ell \leq M$ such that $\deg \rho_\ell = \ell$.
\end{itemize}

This terminology arises because such operators preserve the infinite flag of polynomial subspaces of degree at most $n$. Consequently, one can explicitly determine the sequences of eigenvalues and corresponding eigenfunctions—referred to here as \textcolor{blue}{\emph{eigenpolynomials}} $\{Q_n^\mathcal{M}(z)\}_{n=0}^\infty$—using linear algebra. We call an exactly solvable operator \textcolor{blue}{\emph{non-degenerate}} if $\deg \rho_M = M$, in which case the monic eigenpolynomials are unique for all sufficiently large $n$.

Exactly solvable operators first gained prominence in the 1930s regarding the \emph{Bochner-Krall problem}, which seeks to identify operators whose eigenpolynomials are orthogonal with respect to a linear functional. While this problem remains largely open, various results on the asymptotic behavior of these sequences are available in, e.g., \cite{BeRuSh04, Berg07, BerRull02, masshap01}, and specifically on the Bochner-Krall problem in \cite{Bo29, Ev01, HST, JuKwLee97, Kr40, Turb92b}.

\begin{nota}
For a polynomial $P(z)$ of degree $n$, let $\mu_P = \frac{1}{n}\sum_{i=1}^n \delta(z-u_i)$ denote its root-counting measure, where $u_i$ are the roots of $P(z)$ and $\delta$ is the Dirac delta. For an open set $U \subset \mathbb{C}$, $\mathcal{H}(U)$ denotes the space of analytic functions on $U$.
\end{nota}

  \medskip
  It has been conjectured in \cite{masshap01}  and shown in \cite  {BerRull02} that for any non-degenerate exactly solvable operator $\MM$,  the sequence $\{\mu_n^\MM\}$ of the root-counting measures of its sequence $\{Q_n^\MM(z)\}$ of eigenpolynomials converges in the weak sense to a probability measure $\mu^{\MM}$  depending only on the leading coefficient $\rho_M(z)$. Moreover $\mu^{\MM}$ is supported on an embedded graph in $\CC$ which is topologically a tree whose leaves (i.e. vertices of valency $1$) are exactly  all roots of $\rho_M(z)$, see \cite[Th.3]{BerRull02}. Further, the support of $\mu^{\MM}$ lies inside the convex hull of these roots and can be straightened out in a certain local canonical coordinate which is very natural from the point of view of semiclassical asymptotic for solutions of a linear ODE. More information about $\mu^{\MM}$ can be found in \cite{masshap01, BerRull02}.  
  
  In particular,   for $M=3$, $\supp {\mu^\MM}$ is a tree  with leaves given by the zeros of $\rho_3(z)$.  Hence, if all these zeros are simple,    $\supp {\mu^\MM}$   is the union of three smooth Jordan arcs $\{\mathfrak{r}_1,\mathfrak{r}_2,\mathfrak{r}_3\}$ connecting each zero of $\rho_3(z)$  to a common point $v$ contained in the convex hull of the zeros of $\rho_3(z)$, see Figure \ref{supp}. (The angle between any pair of these arcs at $v$ is $120^o$, see \cite{BerRull02} and Lemmas~\ref{char} and ~\ref{angles}). 
  \begin{figure}[h!]
\centering
        \includegraphics[width=0.2\textwidth]{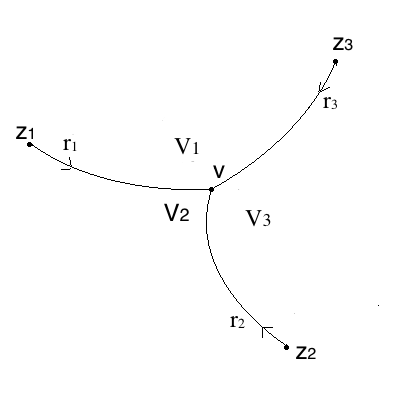}
         \caption{ Typical structure of  $\supp {\mu^\MM}$ for $M=3$}\label{supp}
\end{figure}

%{\color{red} Can we substitute the above paragraph and figure by Lemmas \ref{char} and \ref{angles}?}
  
  \medskip
  \begin{nota} {\rm 
  Given a non-degenerate exactly solvable operator $\MM$,  set   $\Omega:=\CC\setminus \supp {\mu^\MM}$. For a given Jordan arc    $\tau$   connecting  $\infty$ and  an arbitrary  point $z_0\in \supp {\mu^\MM}$,  we  denote by $\Omega_{\tau}$ the open connected set $\Omega\setminus \tau$ and by $\mathcal{Z}$ the set of zeros of $\rho_M$. Given  an open set $U\subset\CC$, let  $\mathcal{H}(U)$ stand for  the space of analytic functions in   $U$.   For a set $A$, let $\mathring{A}$ denote the interior of $A$.     If  $\gamma$ is a closed Jordan curve in $\CC$, we denote by $int(\gamma)$ and $ext(\gamma)$  the  bounded and unbounded connected components of $\CC\setminus \gamma$ respectively.   For any oriented Jordan arc $\tau \subset \CC$,  we denote by  $\tau^+$   the (local) side to the left  of $\tau$.  
  
  Denote by $ w_1$ the branch of $\frac{1}{\sqrt[M]{\rho_M(z)}}$ in $\Omega$ which has asymptotic  $\dsty \frac{1}{z}$ near $\infty$.   Introduce   the other branches of $\frac{1}{\sqrt[M]{\rho_M(z)}}$ as
\begin{equation}\label{wj}
w_j(z):=e^{\frac{2 \pi i (j-1) }{M}}w_1(z),\; z\in \Omega,\; j=2,\ldots,M.
\end{equation}
Further, define $\Phi_0$ as the primitive of $w_1(z)$ in $\Omega_\tau$ such that
$$\lim_{z\to \infty}  \Phi_0(z)-\ln z=0$$ and define $\Phi_1$  as the primitive  of the function 
$\frac{(M -1)\rho^\prime_M(z)} {2M\rho_M (z)} - \frac{\rho_{M-1}(z)}{M\rho_M (z)}$ in $\Omega_\tau$  
such that $$\lim_{z\to \infty} \Phi_1(z) - \left(\frac{M -1}{2} -\frac{\rho_{M-1,M-1}} {M}\right)\ln z = 0,$$
where $\rho_{M-1,M-1}$ is the coefficient of $z^{M-1}$ in  $\rho_{M-1}(z)$. 
}
\end{nota}

  \medskip
  In a recent publication \cite{BoMo22} the first author has established the following WKB-expansion for the sequence $\{Q_n^\MM(z)\}$ of monic eigenpolynomials of a given non-degenerate exactly solvable operator $\MM$, originally conjectured in \cite{masshap01}. 

\begin{THEO}[see Theorem 1 of  \cite{BoMo22}]\label{th:BM}\label{th1:BM} For a  non–degenerate exactly solvable operator $\MM$ of order $M \ge 2$ and the sequence $\{Q_n^\MM(z)\}$ of its monic eigenpolynomials, when $n\to \infty$ one has  the asymptotic expansion  in the sense of Poincar\'e
$$Q_n^\MM(z) \sim \exp \left( n \Phi_0(z)-\left(\frac{M-1}{2}-\frac{\rho_{M-1,M-1}}{M} \right) \Phi_0(z) + \Phi_1(z)\right) \left(1+\frac{C_1(z)}{n}+\frac{C_2(z)}{n^2}+\dots\right),
$$
uniformly on compacts subsets $K\subset \Omega$, where  for $ j \ge 1$, $C_j$ are analytic in $\Omega$. 
\end{THEO}
(The definition of the above asymptotic expansion can be found  in e.g. (7.03), p.16 of \cite{Olv74}).  

\medskip

Notice that the  expansion given in Theorem \ref{th1:BM} implies  that the polynomial eigenfunctions can be obtained of a dominant  WKB-solution   of the differential equation $\MM (v(z,\eta))=\eta^M{v(z,\eta)}$, ($\eta$ a large spectral parameter), multiplied by appropriate exponential small terms which can be defined from the whole set of WKB-solutions, see \cite[Lem.7 \& Lem.9]{BoMo22}.  Motivated by this fact, in the present article we attempt  to extend   the existing results of the exact WKB-analysis to the case of  exactly solvable operators. In particular,  we establish  the Borel summability  of the WKB-solutions of the differential equation $\MM (v(z,\eta))=\eta^M{v(z,\eta)}$, when  $\eta>0$ is a large spectral parameter, in the regions bounded by the  Stokes curves (see exact statements  below).   The description of the global geometry of Stokes curves is a challenging open  problem  for very many types of operators. In  this direction,   we give the full description of the Stokes  curves in  case  of $\rho_3$ having simple roots. Typically the study of the WKB-solutions of higher order differential operators is carried out  by using factorization in lower order operators \cite{AokiKawTak94,AKT95,AKKT04,KKT10,ka11}. In what follows,   we provide an example of an exactly solvable operator not admitting such a factorization. Finally, it is well-known that,  in general,  the WKB-series diverges  in the usual sense. However,  as we show below, in the case of  the Euler-Cauchy differential operators,   their WKB-series converges. 

\medskip

\subsection{Short historical account} 
Analysis of the Borel resummed WKB-solutions  has been a topic of interest for at least  half a century.  Bender and Wu \cite{BendWu} were the first to notice the relevance of the Borel summability to the  analysis of the WKB-solutions.  In \cite{Vor}  Voros  studied the special case of the second order Schr\"odinger equation with a quartic potential establishing  the connection formulae for its WKB-solutions, and    in \cite{Silv85}  Silverstone  discussed the connection problem further. At the same time, Ecalle was  developing the theory of resurgent functions \cite{E81,Ec84}, which is also based on the Borel sums. Extending these contributions,  several researchers \cite{Ph88,CNP93,DDP93,delpham99} introduced  what is now known  as the {\it exact}  or the  {\it complex WKB-analysis}, see   \cite{kawai2005algebraic} for some historical remarks.

\medskip
  The study  of the regions in which the  WKB-solutions are  Borel summable is a fundamental problem of the  exact WKB-analysis. For the second order Schr\"odinger type linear ordinary differential equations  
$$ \frac{d^2y}{dz^2}-\eta^2 Q(z)y=0,$$
when $Q$ is a rational function and $\eta$ is a large positive parameter, this problem  has been solved in  \cite{TAK17}.

\medskip
For the ordinary differential equations of the form 
$$\frac{d^2y}{dz^2}-\eta^2\left(f_0(z)+\frac{f_1(z)}{\eta}+\frac{f_2(z)}{\eta^2}\right)y=0, $$
where $f_0, f_1, f_2$ are analytic in some domain and  large $\Re[\eta]>0$, the Borel summability was established in  \cite{N21},   see also  references in \cite{N23}.  

\medskip
In \cite{MT17} the authors studied  the  Borel  summability of the WKB-solutions  for higher order linear differential equations of the form 
$$\sum_{j=0}^na_j(z)\left(\eta^{-1}\frac{\partial}{\partial z}\right)^jy(z,\eta)=0,$$
 with polynomial coefficients  and large positive parameter $\eta$  
by using  reduction to the linear second order differential equations via middle convolutions.

  In a recent publication \cite{N23}, the author considers  differential equations of the form 
$$\left(-\frac{\partial^n}{\partial z^n}+\sum_{k=0}^{n-2}\eta^{n-k}f_k(z,\eta)\frac{\partial^k}{\partial z^k}\right)y(z,\eta)=0,$$
where $n\geq 2$ and  $f_k$ are analytic functions of $z$ varying in some  domain $\mathbb{D}$, bounded or otherwise, of a Riemann surface,  possessing as $\eta\rightarrow \infty$  and $z$ varying uniformly in $\mathbb{D}$ asymptotic expansions 
$$f_0(z,\eta)\sim \sum_{m=0}^{\infty} \frac{f_{0,m}(z)}{\eta^m}, \quad f_k(z,\eta)\sim \sum_{m=1}^{\infty} \frac{f_{k,m}(z)}{\eta^m}, \quad 1\leq k\leq n-2. $$
The author proves  the  Borel  summability of the WKB-solutions with respect to the parameter $\eta$,  when $z$ varies in $\mathbb{D}$.

\medskip
The present manuscript is organized as follows. In Section~\ref{sec:2} we discuss some basic aspects of the general theory of the Borel summability and  formulate our main results. Section~\ref{PR} is devoted to a number of technical results required for the main proofs which we carry out  in  Section~\ref{Proo}. In Section~\ref{CoordChang} we give  an example of an exactly solvable differential operator which can not be factorized into the WKB-type linear differential operators   of lower order.  Finally, in Section~\ref{SpOp} we show the existence of differential equations which admit convergent WKB-solutions.

%Worth to say that  exactly solvable operators depending on the parameter $\eta$ as  in \eqref{OPer}  do not  fall into this  class.  In the case of the Borel summability of  WKB-solutions

%and with poles in their coefficients has been considered in the context of factorization of these operators into operators of lower order, \cite{KKT10,KKT11}.

\medskip\noindent
\emph {Acknowledgements.}  The authors wish to express sincere gratitude to Professor Yoshitsugu Takei for helpful  discussions. The second author wants to acknowledge the financial support of his research provided by the Swedish Research Council grant 2021-04900.

\section{Preliminaries, basic notions, and formulation of main results} \label{sec:2}
 
\subsection{Preliminaries}

Below we  often use of the general notion of an operator of the WKB-type,  introduced in  \cite{AKKT04}, see also \cite{AKT95}. Let $U\subset\CC$ be an open subset, $(z,y)\in U\times \CC$ and $(z,y;\zeta,\eta)\in T^*(U\times\CC)$, where $T^*$ denotes  the cotangent bundle.  

\begin{definition}
A \textcolor{blue}{{differential operator $P$ of the WKB-type}}  on an open set $U\subset \CC_z$ is a microdifferential operator of order $0$ defined on  $(z,y;\zeta,\eta)\in T^*(U\times\CC), \eta\neq 0$ commuting with the differentiation with respect to $y$, i.e.  $[P,\partial_y]:=P\partial_y-\partial_y P=0$. Thus, its  total symbol $\sigma_0(P)$   is a formal power series of the  form:
$$\sigma_0(P)=\sum_{j\geq0}\eta^{-j}P_j\left(z,\frac{\zeta}{\eta}\right),$$
where $(P_j(z,\zeta))_{j\geq 0}$ are  holomorphic functions in $z\in U$ and entire functions in $\zeta$ (in our current context they are actually polynomials in $\zeta$), and they satisfy the following growth  condition:  

-- there exists a constant $c_0>0$ such that for each compact set $K$ in $U\times \CC_{\zeta}$, we can find another constant $c_K$ for which 
$$\sup_{K}|P_j(z,\zeta)|\leq c_Kj!c_0^j.$$

Following the traditional terminology of the microlocal analysis, we call $\sigma(P):=P_0(z,\zeta)$ the \textcolor{blue}{{principal symbol of the operator $P$}}.
\end{definition}

\begin{example} {\rm In \cite{MT17} the authors  studied  the Stokes geometry of the WKB-solutions of  linear $n$-th order differential operators of the form 
\begin{eqnarray}\label{S}
\mathcal{S}\phi(z,\eta)=\sum_{j=0}^na_j(z)\left(\eta^{-1}\frac{\partial}{\partial z}\right)^j\phi(z,\eta),
\end{eqnarray}
where $a_j(z)$ are polynomials, $a_n$ is a  non--zero complex constant, and $\eta>0$ is a large parameter. By definition,  $\sigma_0(\mathcal{S})$ can be expressed as  $\dsty \sigma_0(\mathcal{S})=P_0\left(z,\frac{\zeta}{\eta}\right)=\sum_{j=0}^na_j(z)\left(\frac{\zeta}{\eta}\right)^j$. Hence, $\mathcal{S}$ is of the WKB-type.
}
\end{example}

Consider the linear differential  operator
\begin{eqnarray}\label{OPer}
\dsty \LL=\MM- \eta^{M}
\end{eqnarray}
 where $\MM$  given by  \eqref{eq:ES} is an  exactly solvable operator.  Let   $v(z,\eta)$ be a family of solutions to the equation 
$$\LL (v(z,\eta))=0,$$ 
where  (unless otherwise specified) $\eta$ is assumed  to be  a large positive  parameter.  In other words,  $v(z,\eta)$  satisfies the relation 
\begin{eqnarray}\label{Eq1} 
v^{(M)}(z,\eta)+\sum_{k=1}^{M-1} \frac{\rho_k(z)}{\rho_M(z)} v^{(k)}(z,\eta)-\eta^{M} \frac{v(z,\eta)}{ \rho_M(z) }=0. 
\end{eqnarray}

\begin{example}{\rm 
Let  $\LL$ be  defined as  in \eqref{OPer}.  Then $\eta^{-M}\LL$ is of the WKB-type. Indeed, 
\begin{eqnarray}
\sigma_0(\eta^{-M}\LL)&=&\sum_{k=1}^{M}\eta^{-M+k}\rho_{k}(z)\frac{\zeta^k}{\eta^k}- 1 \\
   &=&-1+\rho_M(z)\left(\frac{\zeta}{\eta}\right)^M+\sum_{j=1}^{M-1}\eta^{-j}\rho_{M-j}(z)\left(\frac{\zeta}{\eta}\right)^{M-j}.
\end{eqnarray}
}
\end{example}

\medskip
In case of a linear differential operator $P$ given by  \eqref{OPer} or \eqref{S}, the  fundamental role in its Stokes geometry  is played by the branch points of its \textcolor{blue}{{\it symbol curve}}  $\Gamma_P\subset\CC^2\simeq \CC_z\times \CC_{\zeta}$ given by the \textcolor{blue}{{\it symbol equation}}
\begin{eqnarray}\label{Alg}
\begin{aligned}
\sigma(P)(z,\zeta)=0.
\end{aligned}
\end{eqnarray}

Observe that  for \eqref{OPer}, its symbol equation reduces to 
\begin{equation}\label{chareq} 
\sigma(\LL)(z,\zeta)= \rho_M(z)\zeta^M-1= 0.
\end{equation}
This equation  plays a fundamental role in the asymptotic analysis of the operator $\LL$ and is referred to as  the \textcolor{blue}{ \emph{characteristic equation}} of \eqref{OPer}. 

\medskip
 Projecting  $\Gamma_P$  on the first coordinate $\CC_z$ we obtain the finite subset of  $\Gamma_P$ consisting of the \textcolor{blue}{ {\it branching points}} of this projection, i.e. points near which the projection is not a local diffeomorphism. To globalize the situation,  one usually considers the compactification $\widehat \Gamma_P \subset \CC P^1_z\times \CC P^1_{\zeta}$ of $\Gamma_P \subset \CC_z\times \CC_{\zeta}$, its projection on $\CC P^1_z$, and its branching points. 

\smallskip
When \eqref{Alg} reduces to an algebraic equation $b_n(z)\zeta^n+b_{n-1}(z)\zeta^{n-1}+\ldots +b_0(z)=0$  (see  \cite[p.185]{jones87})  the \textcolor{blue}{ \emph{ set of critical values}} of $\widehat \Gamma_P$  is a finite subset of  $\CC P^1_z$ whose points satisfy at least one of the following 3 conditions: 
\begin{itemize}
\item $z=\infty$;
\item $b_n(z)=0$;
\item $z\in \CC_z$ is a point at which  the characteristic equation \eqref{Alg} has a multiple root in the variable $w$. 
\end{itemize} 
It is known that the projection of any branch point of $\widehat \Gamma_P$ to $\CC P^1_z$  lies among its  critical values, cf. \cite[Th 4.14.3 p.186]{jones87}.  

\smallskip

\begin{definition}\label{WKBS}
A  \textcolor{blue}{WKB-solution } of a linear differential operator  $P$ of the WKB-type is a formal solution of the form: 
\begin{eqnarray*}
\dsty \psi(z,\eta,z^*)&=&\eta^{-\frac{1}{2}}\exp\left(\dsty \int_{z^*}^z S(\zeta,\eta)d\zeta \right)\\
                  &=&\eta^{-\frac{1}{2}}\exp\left({\dsty {\eta}\int_{z^*}^z\sum_{k=0}^{\infty}\frac{S_k(\zeta)d\zeta}{\eta^{k}}}\right)\\
                  &=&\exp\left( \eta \int_{z^*}^zS_0(\zeta)d\zeta\right) \sum_{n=0}^{\infty}\frac{\phi_n(z)}{\eta^{n-\frac{1}{2}}}, 
\end{eqnarray*} 
where   $S_k$ are (locally) holomorphic functions, and $z^*$ is some reference point. 
\end{definition}

\begin{remark}{\rm 
Some authors instead of  $\eta^{-\frac{1}{2}}$ in the above definition  consider the more general normalizing factor  $\eta^{-\alpha}, \alpha\geq 0$,   cf. \cite{TAK17}.
}
\end{remark}

A WKB-solution for \eqref{OPer}  can be constructed by substituting the expression $\dsty  e^{\int^z S(\xi,\eta)d\xi}$ in \eqref{OPer} and solving it for $S(\xi,\eta)$. In Lemma \ref{Fj} we show that   $S$ should satisfy  a generalized Riccati equation of order $M-1$. Expanding $\dsty S(z,\eta)=\sum_{k=0}^{\infty}S_k(z)\eta^{1-k}$, 
%\Boris{Are these $y_k$ the same as above?}  \Jorge{Ok, fixed. We still will have to make a normalization in the definition   to follow the general theory}
 one can easily check that $S_0(z)$ is a solution of  the simple algebraic equation 
\begin{eqnarray*}
 \rho_M(z)S_0^M(z)-1=0 \quad \leftrightarrow \quad S_0(z)=\frac{1}{\root M \of {\rho_M(z)}},
\end{eqnarray*} 
and the remaining $S_n, n>0$ can be obtained recursively from $S_{n-1}, \ldots,S_0$, cf. \cite{Si60}. To obtain a WKB-solution,  (considering  $z^*$ as the  reference point) the indefinite integral in the expression is  substituted by a definite integral $\int_{z^*}^z S(\xi,\eta)d\xi$. 

In our case, the functions $S_n$  may  have  singularities at each zero $z_k$ of $\rho_M(z)$ in such a way that  the above integral of $S(\xi,\eta)$ can not be defined in the usual sense   when $z^*=z_k$. Indeed,   from \cite[Prop.1 b)]{BoMo22}, the function $S_1$ is given by 
$$S_1(z)=\frac{(M -1)\rho^\prime_M(z)} {2M\rho_M (z)} - \frac{\rho_{M-1}(z)}{M\rho_M (z)}.$$
In particular, if $\rho_{M-1}\neq \frac{(M-1)}{2}\rho_M^{\prime}$,   we have  $S_1(z)\sim  a_k(z-z_k)^{-m_k}$  in a neighborhood of $z_k$, where $m_k$ is the multiplicity of the root $z_k$ and $a_k\in \CC$.  Notice also that  for $\rho_M(z)=(z-z_k)^M$,  the function $S_0$ reduces to $\dsty \frac{\sqrt[M]{1}}{z-z_k}$.

Consequently, when the reference point is taken at $z_k$, we interpret the integral $\int_{z_k}^z$ in the sense of the \textcolor{blue}{\emph{Hadamard finite part}}. We recall its definition, originally introduced by Hadamard to treat specific classes of divergent integrals \cite[Ch.1]{MiLa86}. For simplicity, consider a function $f(x)$ for $x \in \mathbb{R}$ of the form
\begin{equation}\label{f(x)}
f(x) = a(x-c)^{-\nu} + b(x-c)^{-1} + s(x),
\end{equation}
where $c \in \mathbb{R}$, $\Re[\nu] > 1$, $\nu \neq 1$, and $s(x)$ is integrable on $[c, C]$. For any $\delta$ such that $c < c+\delta < C$, let $J(\delta) = \int_{c+\delta}^C f(x) \, dx$. Term-by-term integration then yields
$$
J(\delta) = -\frac{a}{\nu-1}(C-c)^{-\nu+1} + b\ln(C-c) + \frac{a}{\nu-1}\delta^{-\nu+1} - b\ln \delta + \int_{c+\delta}^C s(x) \, dx.
$$

When $\delta\to 0$, the function $J(\delta)$ has  no finite limit because of the terms $\frac{a}{\nu-1}\delta^{-\nu+1}-b\ln \delta$, but the remaining terms in the right-hand side have a limit which is called the \textcolor{blue}  { \emph{finite part of the integral}} $\int_{c+\delta}^C f(x)dx$ when $\delta\to 0$. We  use the notation  ${\tt Fp}\int_{c}^Cf(x)dx$ to represent this finite part.  Notice that from \eqref{f(x)} we have 
\begin{eqnarray}\label{HadamarFP}
{\tt Fp}\int_{c}^Cf(x)dx=-\frac{a}{\nu-1}(C-c)^{-\nu+1}+b\ln(C-c)+\int_{c}^Cs(x)dx.
\end{eqnarray}

The definition \eqref{HadamarFP}   is easily extended to the case when the integration is taken along a smooth arc $\gamma$ in the complex plane.  Other regularization methods  might be considered to deal with  divergent integrals, see for instance \cite[Ch.2 p.19]{AKT95}.

 \medskip
%Observe that   in the domain $\Omega$, the algebraic function $S_0(z)$ splits into $M$ single-valued branches  given by $w_1(z),\dots,$ $ w_M(z)$, see above. 
%We   denote $S_{0}(z)=w_j(z),\; j=1,\dots, M$; the choice of the branch will be made according to the index $j$ of  the function $w_j$.

Within the domain $\Omega$, the algebraic function $S_0(z)$ resolves into $M$ single-valued branches $w_1(z), \dots, w_M(z)$. We write $S_0(z) = w_j(z)$ for $j \in \{1, \dots, M\}$ to specify the branch associated with a given index $j$.

Having fixed $S_0(z)$, we recursively determine the higher-order terms $S_k(z)$ for $k \ge 1$ as single-valued functions on $\Omega_\tau$. This process yields $M$ formal WKB solutions $\{\psi_j\}_{j=1}^M$ of \eqref{OPer} in $\Omega_\tau$. The following definition is fundamental to establishing the Borel summability of $\psi_j$ as stated in Theorem~\ref{Main}.

  \begin{nota}\label{fitn}
Take   $z_k\in \mathcal{Z}$ and let  $\psi_j$ be a WKB-solution. Denote by $\mathcal{D}(z_k)$ the set of indices $n$ for which the integral $\dsty \int_{z_k}^zS_n(\zeta)d\zeta$ diverges, cf. Lemma \ref{Ext}.
\end{nota}

 It might happen that the WKB-solutions  converge in some subdomain of $\Omega_\tau$, as in e.g. Theorem \ref{EC}. But, generally, $\psi_j,\; j=1,\dots, M$ are diverging in the whole $\Omega$. Due to   the important discovery of Voros \cite{Vor} and  Ecalle \cite{Ec84},  the  use of the Borel resummation technique (or the Borel-Laplace method) with respect to a large parameter $\eta$ rescues the situation; we recall the definitions of  Voros \cite{Vor} below. 
 
\begin{definition}\label{BT}
 Let $\eta>0$ be a large parameter and $y_0,f_n$,  $\alpha\in \RR\setminus \ZZ_{\leq 0}$ be constants. For an infinite series 
 $$\dsty f(\eta)=\exp(\eta y_0)\sum_{n=0}^{\infty}f_n\eta^{-(n+\alpha)},$$ the \textcolor{blue}{Borel transform} $f_B(y)$  and the \textcolor{blue}{Borel sum} $F(\eta)$ of $f$ are defined as 
$$
 f_B(y)=\sum_{n=0}^{\infty}\frac{f_n}{\Gamma(n+\alpha)}(y+y_0)^{n+\alpha-1}\quad \text{  and  } \quad 
 F(\eta)=\int_{-y_0}^{\infty}e^{-y\eta}f_B(y)dy,
$$
respectively provided that the right-hand sides exist.  Here $\Gamma(s)$ is  Euler's $\Gamma$-function and the integration path is taken parallel to the positive real axis. 
 
\end{definition}

 \begin{remark}{\rm 
In accordance with Definition \ref{WKBS}, we fix $\alpha = 1/2$ throughout this manuscript. Note that Borel transform conventions vary in the literature; while we adopt the definition above, alternative formulations appear in works such as \cite{AKT91, Vor}, where the transform is defined as
$$f_B(y)=\sum_{n=0}^{\infty}\frac{f_n}{\Gamma(1+n+\alpha)}(y+y_0)^{n+\alpha}.$$
}\end{remark}

\medskip

\medskip

  Let $\psi_{j,B}(z,y,z^*)$ denote the Borel transform of the WKB solution $\psi_j(z,\eta,z^*)$, defined as 
\begin{equation}\label{BTR}
\psi_{j,B}(z,y,z^*)=\sum_{n\geq 0}\frac{f_n(z,z^*)}{\Gamma(n+\frac{1}{2})}\left(y+y_0(z,z^*)\right)^{n-\frac{1}{2}},
\end{equation}
where $y_0(z,z^*) = \int_{z^*}^z S_0(\zeta) \, d\zeta$ and the coefficients $f_n$ are determined recursively once $S_0$ is fixed (cf. Lemma \ref{Ext}). We denote the Borel sum of $\psi_j(z,\eta,z^*)$ by $\Psi_{j}(z,\eta,z^*)$.

 \medskip

For $\dsty \psi(z,\eta)=\exp(\eta y_0(z))\sum_{n\geq 0}\phi_n(z)\eta^{-(n+\alpha)}, \alpha>0,\alpha\notin\ZZ$, it is immediate from the definition that 
$$
\dsty \left[\frac{\partial \psi}{\partial z}\right]_B=\frac{\partial }{\partial z}\psi_B\quad  \text{and} \quad \dsty \left[\eta^m \psi\right]_B=\left(   \frac{\partial}{\partial y}\right)^m\psi_B, m=1,2,\ldots
$$

 In particular, if $\psi(z,\eta)$ is a formal solution of the linear differential equation \eqref{Eq1}, 
then its Borel transform $\psi_B(z,y)$ satisfies the  linear  partial differential equation
\begin{eqnarray}\label{PDE}
\LL_B\left(z,y, \frac{\partial u}{\partial z},\frac{\partial u}{\partial y}\right)=\sum_{k=1}^{M}\rho_{k}(z)\frac{\partial^k u}{\partial z^k}- \frac{\partial^M u }{\partial y^M}=0,
\end{eqnarray}
which coincides with the Borel transform of the operator $\LL$.

 The behavior of the WKB-solutions crucially depends on the critical  points of the characteristic equation since  these solutions do not  provide a single-valued  fundamental  system   in  full neighborhoods of the critical points, cf. \cite[Def.3.1-2 p.39]{wasow2012}. This difficulty leads to the notion of turning points. 
\begin{definition}[see Def. 1.2.1 p.21 \cite{H15vir}, \cite{Ko00,Fed93,ka11}]\label{Tp}
Let $P$  be a differential operator of the WKB-type in an open set $U\subset \CC_z$. A  critical point $a$ of \eqref{Alg}  is called a  turning point of $P$.  When \eqref{Alg} reduces to an algebraic equation with coefficients in $\mathcal{H}(U)$ and $b_n(a)=0$, we  additionally say   that $a$ is of \textcolor{blue}{ pole-type}, and if $b_n(a)\neq 0$, we  refer to $a$ as an  \textcolor{blue}{ordinary turning point}.  When two roots $\zeta_j(z)$ and $\zeta_{j^{\prime}}(z), j\neq j^{\prime}$ of the symbol equation merge at a turning point $a$, we say that  $a$ has \textcolor{blue}{ type $(j,j^{\prime})$}. (Notice that if more than two roots collide at $a$ then several types are assigned to $a$.) 

\medskip
\noindent
 If $a$ is  a   turning point of type $(j,j^{\prime})$, then a curve emanating from the point $a$ and defined by the equation 
$$
 \Im \left[\int_{a}^z(\zeta_{j}(\xi)-\zeta_{j^{\prime}}(\xi))d\xi\right]=0,$$
  is called a \textcolor{blue}{Stokes curve of type $(j,j^{\prime})$ emanating from $a$}.   
  
  \medskip
   We denote by $\mathcal{S}_{a,j}$ the set of all Stokes curves of type $(j,j^{\prime}), 1\leq j^{\prime}\leq M, j\neq j^{\prime}$ emanating from $a$. 
 
\medskip
\noindent 
An ordinary turning point of a linear ODE at which exactly  two roots $\zeta_j$ and $\zeta_{j^{\prime}}$ of its symbol equation  collide is called \textcolor{blue}{simple}, that is 
$$\left.\frac{\partial \sigma(P)}{\partial z}\right|_{(z,\zeta)=(a,\zeta_j(a))}\neq 0.$$

%The  \textcolor{blue}{multiplicity of a pole-type  turning point} $z=a$ is defined as the multiplicity of the pole of $b_n(z)$ at $z=a$.

%\Boris{In our case all multiple roots might collide!} 
%\Jorge{We still need to include that case}
\end{definition}

%\begin{remark} {\rm 
%Let us consider the equation
%$$w^{\prime\prime}-\lambda q(z)w=0,$$
%where $q$ is meromorphic  in some domain $D\subset \CC$ and $\lambda>0$ is a large parameter.  In \cite[Ch.3 p.81]{Fed93} it is studied the geometry of Stokes curves for $q$ belonging to some particular classes of meromorphic functions. In particular, first order poles of $q$ are called turning points of order $-1$. 
%}
%\end{remark} 

\begin{remark}{\rm 
In Section \ref{SpOp}, we examine the Euler-Cauchy operator. In this instance, the standard definition of Stokes curves is inapplicable because the associated integral fails to converge in the usual sense. One might consider defining the curves via the Hadamard finite part as
\begin{equation}\label{STdeg}
\Im \left[{\tt Fp}\int_{0}^z(\zeta_{j}(\xi)-\zeta_{j^{\prime}}(\xi))\,d\xi\right] = \Im \left[(e^{\frac{2\pi (j-1)i}{M}}-e^{\frac{2\pi (j^{\prime}-1)i}{M}}) \ln z \right] = 0.            
\end{equation} 
However, as shown in the sequel, no Stokes phenomenon occurs in this case.
}\end{remark}

%
%
%We study the  Euler-Cauchy operator in Section \ref{SpOp}.   Observe that the Stokes curves can not be defined  since the integral is divergent. In particular, there is no Stokes phenomenon.  This   operator is kind of exotic,  observe that  $z=0$  is a turning point and no Stokes curves emanating from there.  

\begin{remark} {\rm Although for a generic linear ODE depending on a parameter,   all its ordinary turning points are typically simple, equation~\eqref{OPer} we consider   below is highly non-generic.  
Namely, one can  easily observe from the characteristic equation \eqref{chareq} that the set of ordinary turning points of  \eqref{OPer} coincides with  the zero locus of $\rho_M(z)$. Moreover  each of these zeros is a pole-type   turning  point of every type $(j,j^{\prime})$, for all $1\leq j< j^{\prime}\leq M$. Notice also that in our case, the definition of the Stokes curves emanating from the turning point $a$ reduces to 
$$ \Im \left[\int_{a}^z(w_{j}(\xi)-w_{j^{\prime}}(\xi))d\xi\right]=0.$$
}
\end{remark}

% 
% \begin{remark}{\rm 
%As in  \cite[p.30]{H15vir} and \cite[p.275]{KKT10}, the notion of an ordinary turning point can be also given in terms of singularities of the bicharacteristic strips, defined below. The present definition is also justified by this  fact. } 
%
%\end{remark}
%

\begin{remark}{\rm 
  Let $a$ be an ordinary  turning point of type $(j, j^\prime)$. Some authors define  a Stokes curve as  given by 
\begin{eqnarray}\label{Stc}
\Re\left[\int_{a}^z(\zeta_{j}(\xi)-\zeta_{j^{\prime}}(\xi))d\xi\right]=0,
\end{eqnarray}
emanating from $a$,  
see e.g.  \cite[p.292]{Fed93}.   

}

\end{remark}

It is well-known that  in the second order case the Stokes regions (the regions where the Borel sum of  the WKB-solutions is well-defined)  are domains in the $z$-plane bounded by the Stokes curves, cf. \cite[p.26]{kawai2005algebraic}). For  linear ODE of order greater than $2$, Stokes regions  are much more difficult to describe since the totality of the Stokes curves emanating from the (original) turning points  are not enough to describe the boundaries of the Stokes regions.  As was first noticed  in \cite{Berk82}, see also \cite{AokiKawTak94},   the Borel summability of  the WKB-solutions may fail on  \textcolor{blue}{\emph{new Stokes curves}} obtained from  ordered  crossing points of the original Stokes curves, in the terminology  of   \cite{Berk82}.  Thus, new Stokes curves emanating from  new turning points are a natural generalization of the original Stokes curves emanating from the original turning points.  

Due to results of Voros (see \cite{Vor}) who  first recognized that the Borel transform is a solution of a linear partial differential operator and to microlocal analysis  \cite{AokiKawTak94,SKK73},   new Stokes curves can be defined  as the Stokes curves emanating from  "new" singularities of the bicharacteristic strip. These are  baptised in  \cite{AokiKawTak94}, where this concept was introduced  as "new turning points" or  "virtual turning points".

\begin{definition}[see p.29, \cite{H15vir}, \cite{Ho90,SKK73,DuHo94}] 
\label{BS}
 A   \textcolor{blue}{bicharacteristic strip} $\mathcal {BS}(t)$ associated with a  linear partial differential operator   is a complex-analytic  curve $\mathcal {BS}(t)=(z(t),y(t);\zeta(t),\epsilon(t))_{t\in \CC}$ in the cotangent bundle $T^*\CC^2_{(z,y)}$ with coordinates $(z,y;\zeta, \epsilon)$ where $\zeta$ is dual to $z$ and $\epsilon$ is dual to $y$ defined by the following system of Hamilton--Jacobi equations:  
\begin{eqnarray}
 \label{Ha1}\frac{dz}{dt}&=&\frac{\partial \sigma}{\partial \zeta}\\
\label{Ha2} \frac{dy}{dt}&=&\frac{\partial \sigma}{\partial \epsilon}\\ 
\label{Ha3} \frac{d\zeta}{dt}&=&-\frac{\partial \sigma}{\partial z}\\
\label{Ha4} \frac{d\epsilon}{dt}&=&-\frac{\partial \sigma}{\partial y},\\
\label{Ha5} \sigma(z,y,\zeta,\epsilon)&=&0, 
\end{eqnarray}
where $\sigma $ denotes the principal symbol of  the operator.  The image of the projection of a bicharacteristic strip $\mathcal {BS}(t)$ to the base  $\CC^2_{(z,y)}$ is called a \textcolor{blue}{bicharacteristic curve} and is denoted by  $\mathcal{BC}(t):=\{(z(t),y(t))\}_{t\in \CC}$.

\end{definition}

\begin{remark}{\rm  One can check that since the initial condition  $\mathcal {BS}(t_0)$ of a bicharacteristic strip  lies on the hypersurface $\sigma(z,y,\zeta,\epsilon)=0$ then the whole bicharacteristic strip $\mathcal {BS}(t)$ lies on it as well. }
\end{remark}

A fundamental result of the  microlocal analysis  claims that  the singularities of solutions of a linear partial differential equation with simple  (in the sense of microlocal analysis \cite[Ch.II]{SKK73})  characteristics,  propagate along  the bicharacteristic strips, see also \cite[Cor.7.2.2]{DuHo94}. Notice that by \eqref{Ha3},  for a  WKB-type differential operator with ordinary simple  turning points  one has by definition  $\dsty \frac{\partial \sigma}{\partial z}(a,0,\zeta_0,1)\neq 0$ which implies   that the bicharacteristic strip emanating from $(a,0,\zeta_0,1)$ is locally non-singular in $T^*\CC^2_{(z,y)}$.  The singularities of the   Borel transform   belong to the same non-singular bicharacteristic strip and  coalesce at a turning point. Such singularities are then called "cognate", as they belong to the same bicharacteristic strip.  (Notice that on the bicharacteristic curve other  singularities might exist as well).  
  The most basic one among such singularities is a simple self-intersection point on $\mathcal {BC}(t)$ at which two of its smooth local branches intersect transversally, while the  lifts of these two local branches to the respective bicharacteristic strip $\mathcal {BS}(t)$ are  disjoint. The projections of such self-intersection points from $\mathcal {BC}(t)$ to $\CC_z$ were baptized \textcolor{blue}{\emph{virtual turning points}} in \cite{AokiKawTak94}, where they were first introduced and studied. 
 
 \begin{definition}[\cite{AokiKawTak94,AKKT05,AHKK08,H15vir}]\label{VTP}
Let $P$ be a differential operator of the WKB-type with the principal symbol $\sigma_0(z,\zeta)$  and assume that its Borel transform $P_B$ is well-defined.  Assume additionally  that the bicharacteristic strip  is non-singular at the turning points.   A  \textcolor{blue}{virtual turning point} of $P$ is defined  as the $z$-component of a self-intersection point of a bicharacteristic curve $\mathcal {BC}(t)$. If the self-intersection is  associated with the factor $(\zeta-\zeta_j(z)\eta)$ and $(\zeta-\zeta_k(z)\eta)$ of the principal symbol $\dsty \sigma(P)=\prod_j (\zeta-\zeta_j(z)\eta)$, then the virtual turning point is said to be of type $(j,k)$.

If $z^*$ is a virtual turning point  of type $(j,k)$, the curve emanating from $z^*$
$$ \Im \left[\int_{z^*}^z(\zeta_{j}(\xi)-\zeta_{k}(\xi))d\xi\right]=0,$$ 
is called a \textcolor{blue}{new Stokes curve of type $(j, k)$}.
\end{definition}

\begin{remark}
{\rm   In the case of ordinary turning points of a linear differential operator of the WKB-type of multiplicities greater than $1$  the  singularities of solutions  propagate along the so-called bicharacteristic chains,  as shown in \cite{KKO74}, see also \cite[Ch.3]{H15vir}.  In this case  singularities bifurcate along two mutually tangent bicharacteristic curves at a double turning point where the simple characteristic condition is violated.}
\end{remark}

Observe that an exactly solvable operator can be considered as a linear differential operator with poles at the zeros of $\rho_M(z)$. Virtual turning points of operators with  pole-type original turning points have been previously considered in \cite{KKT10}. The authors  specifically considered a third order differential operator with a pole at $z=0$ constructed from  the  Berk–Nevins–Roberts operator \cite{Berk82} by using a singular coordinate transformation. However it turns out that,  in general,  the analysis  of the operator $\LL$ can not be reduced to that  of an operator with ordinary turning points by means of a coordinate transformation. Namely,   in Section \ref{CoordChang} we provide an example showing that such factorization  does not exist  in a neighborhood of a turning point of a cubic exactly solvable differential operator. 

\medskip
For this reason, when extending the concept of  virtual turning points to  $\LL$ in order to  analyze  the  propagation of singularities of its Borel transform  $\LL_B$ we follow a different approach suggested in   \cite[Sect. VII p.240]{DuHo94}, see also \cite{Ho90} and \cite[p.44]{Ni73}. 

Namely, by \cite[Cor.7.2.2]{DuHo94},  for a linear differential operator with complex coefficients and   principal complex symbol $p(z,\zeta), z=(z_1,\ldots,z_n), \zeta=(\zeta_1,\ldots, \zeta_n)$,  the singularities of its solutions propagate along $\mathcal {BS}(t)$ provided that
 \begin{itemize}
 \item[a)] $\nabla_{\zeta}\Re[p]\neq 0$  and $\nabla_{\zeta}\Im[p]\neq 0$ are linearly independent,
 \item[b)] $H_p\overline{p}=0$,
\end{itemize}  
 where $\dsty H_pq=\frac{1}{\imath}\sum_j \frac{\partial p}{\partial \zeta_j}\frac{\partial q}{\partial z^j}-\frac{\partial p}{\partial z^j}\frac{\partial q}{\partial \zeta_j}$ is the Hamiltonian operator and  $\frac{\partial }{\partial z}=\frac{1}{2}\left( \frac{\partial }{\partial x}-\imath \frac{\partial }{\partial y}  \right )$.   
 
 For a linear differential operator with holomorphic  coefficients, condition $b)$ is automatically satisfied since $$H_p\overline{p}= [p,\overline{p}]=0, $$ where $[p,\overline{p}]$ denotes  the Poisson bracket of $p$ and $\overline{p}$,   and condition $a)$ reduces to  $\nabla_{\zeta}p\neq 0$. 
 
 \medskip
 Virtual turning points for  the operator $\LL$ can not be defined by formally following the approach  for  ordinary turning points since  the expression $\dsty \frac{\partial \sigma_0}{\partial z}$ has a singularity at each turning point. Therefore the  notion of "cognate" singularities in the  sense of ordinary turning points does not apply in our case. 
 
 Below we describe  the Stokes regions for the third order exactly solvable operators and  consider the ordered crossings of their Stokes curves following the original approach of Berk–Nevins–Roberts in \cite{Berk82}, but without providing the rigorous definition of a  virtual turning point. 
In Theorem \ref{Main} b), we describe  the singularities of the bicharacteristic strip as the initial step toward the understanding of  this concept.  We plan to return to this notion  for exactly solvable operators  in a future publication.

\begin{nota}\label{LbSt}
 Let $z=a$ be a   turning point of type $(j, j^\prime)$ (ordinary or pole-type). Then each segment of the Stokes curve emanating from the point $a$ is labeled by either $(j > j^\prime)$ or by $(j < j^\prime)$, depending on  whether 
$$\Re\left[\int_a^z(w_j(\xi)-w_{j^\prime}(\xi))d\xi     \right]>0$$
or 
$$\Re\left[\int_a^z(w_j(\xi)-w_{j^\prime}(\xi))d\xi     \right]<0,$$
cf.  {\rm \cite[Def. 1.2.2 p.22]{H15vir}}.
 
\end{nota}

\begin{definition}[see p.38, \cite{H15vir}] 
\label{OC}
Consider two Stokes curves of types $(j_1,j_2)$ and $(j_2,j_3)$ and assume that they are  crossing at a point $C$. We say that  they define  an \textcolor{blue}{ordered crossing} at $C$  if either $j_1< j_2<j_3$ or $j_1>j_2>j_3$. Following \cite{Berk82}, we introduce a  \textcolor{blue}{new Stokes curve emanating from $C$} as given by 
$$ \Im \left[\int_{C}^z(\zeta_{j_1}(\xi)-\zeta_{j_3}(\xi))d\xi\right]=0.$$
  (By calling it  a "new Stokes curve" we  distinguish it  from the usual Stokes curves emanating from a usual turning point). We denote by $\mathcal{N}$  the set of all new Stokes curves.
\end{definition}

\begin{remark}{\rm 
For linear differential operators where a rigorous notion of a virtual turning point is established, a \emph{new Stokes curve} is defined as a Stokes curve emanating from such a point, as in Definition \ref{VTP}. Since a formal definition of a virtual turning point for the operator \eqref{OPer} is currently lacking, we adopt the classical definition of new Stokes curves proposed in \cite{Berk82}.
}\end{remark}

\begin{definition}[see Def. 1.4.3 p.38, \cite{H15vir}] \label{Inert}
We say that the Stokes curve is \textcolor{blue}{{ inert near $z_0$}} if  there is no Stokes phenomena,  i.e. if there is no discontinuous change of the asymptotic   near a point $z_0$ lying on a Stokes curve. If a Stokes curve is inert  near all its points,  we simply call it  \textcolor{blue}{{ inert}}. 
 \end{definition}

% 
%However, in our case one still has that  $\nabla_{\zeta}p\neq 0$ at each turning point, as  swill be shown, which implies that at least of set $B$ is connected. 
  
%  Following \cite[p.302]{Ho90} we say that  a bicharacteristic strip non-singular if $\nabla_{\zeta}p\neq 0$.

\subsection{Formulation of the main results}

\begin{theorem}\label{Main}
Let $z_k\in\mathcal Z$ and $z\in\Omega$, and assume that 
$\rho_M(z)\neq (z-a)^M$. Let  $\mathcal{S}_{z_k,j}$ be as  in Definition \ref{Tp}, and $\mathcal{N}$ be as in Definition \ref{OC}. If $\psi_{j,B}(z,y,z_k)$ is the Borel transform of $\eta^{-\frac{1}{2}}\exp ({\tt Fp}\int_{z_k}^zS(\zeta,\eta)d\zeta )$, then    

\begin{itemize}
\item[a)]   $\psi_{j,B}(z,y,z_k)$ has  singularities in the curve   $\dsty y=-\int_{z_k}^{z}w_j(t)dt$. The remaining singularities, if any,  are included in  the set  
$$\dsty \left\{(z,y):y=-\int_{z_k}^{z}w_l(t)dt, l=1,\ldots,M; l\neq j\right\};$$ 

additionally,  $\psi_{j,B}(z,y,z_k)$ is of exponential type when  $$y\in [-y_0(z,z_k), t\Re[-y_0(z,z_k)]+\imath(\Im[-y_0(z,z_k)]) ], t>0; z\notin \mathcal{S}_{z_k,j}\cup \mathcal{N};$$ 

\item[b)] Let $\mathcal L_B$ denote the Borel--transformed operator and
$\mathcal{BS}(t)$ a bicharacteristic strip associated with its principal
symbol in the sense of Definition~\ref{BS}.  
Then the projection $\mathcal{BC}(t)$ of $\mathcal{BS}(t)$ to the base
$\mathbb C^2_{(z,y)}$ contains no singular points in the $z$--variable
other than the turning points $z_k\in\mathcal Z$.

\item[c)]   $\eta^{-\frac{1}{2}}\exp ({\tt Fp}\int_{z_k}^zS(\zeta,\eta)d\zeta )$ is Borel summable provided that $z\notin \mathcal{S}_{z_k,j}\cup \mathcal{N}$.

\end{itemize}
\end{theorem}

\begin{theorem}\label{VTPE}
For $M\geq 2$,  take a Stokes curve $\kappa$ of the equation \eqref{OPer}  of type $(1,j)$,   emanating from a   turning point $z_k\in \mathcal{Z}$  of \eqref{OPer}, and  going to $\infty$. Assume that  $\ell$ neither connects   $z_1$ with another turning point nor is a loop connecting $z_1$ to itself. Then,  if  two Stokes regions $U_i, i=1,2$   
share  a subset of $\kappa$,  then one of the  following two situations occur:

\medskip
Case 1. For $(1>j)$, 
the Borel sums of the  WKB-solutions $\Psi_{1,i}(z,\eta,z_k)$   and $\Psi_{j,i}(z,\eta,z_k)$ on $U_i$ continue  analytically  from $U_1$ to $U_2$ (and in the opposite direction  from $U_2$ to $U_1$). Moreover,  
$$
  \begin{cases}
  \Psi_{1,1}=\Psi_{1,2}\pm c_j\Psi_{1,2}(z,\eta) &  \\
   \Psi_{j,1}=\Psi_{j,2}. 
\end{cases} 
$$

\medskip
Case 2.   For $(1<j)$, one has  
$$
  \begin{cases}
  \Psi_{1,1}=\Psi_{1,2} &  \\
   \Psi_{j,1}=\Psi_{j,2}\pm c_j\Psi_{1,2}(z,\eta).
\end{cases}
$$
 Here $c_j$ 
is the ``\textcolor{blue}{\emph{alien derivative}}"  of $\psi_{1,B}$ in the sense of Ecalle (cf.  \cite{delpham99,sa07}) whose sign  $\pm$ correspond to the
counter-clockwise and clockwise crossing of $\kappa$ of the path of analytic continuation from $U_1$ to $U_2$ with the turning point as the reference point. (The number $c_j$ satisfies the equation   $\Delta_{y=y_0(z,z_k)}\psi_{j,B}(z,y)= c_j\psi_{1,B}(z,y)$).

\end{theorem}

\begin{remark}{\rm 
 In  \cite{Ko00}, it is studied WKB-solutions of    the differential equation  
\begin{eqnarray}\label{Koik}
\frac{d^2 u}{dz^2}-\eta^2\left(\frac{Q_0(z)}{z}+ \frac{1}{\eta^2} \frac{Q_2(z)}{z^2}   \right)u=0,
\end{eqnarray}
where $Q_0(z)$ and $Q_2(z)$ are functions analytic  near $0$ and satisfying the condition $Q_0(0)\neq 0$ and $\eta$  denotes a large parameter.  By assuming the Borel summability of the WKB solutions  of  the equation  \eqref{Koik}, the author obtained  the following connection formula   for the Borel sums $\Psi_+$ and $\Psi_-$:
  
  \medskip\noindent
  
   If $\Re[\int_0^z\sqrt{\frac{Q_o(t)}{t}}dt]>0$ along the Stokes curve, one has 
\begin{eqnarray*}
  \begin{cases}
  \Psi_+= \Psi_+ + 2\imath \cos(\pi \sqrt{1+4Q_2(0)})Y_-\\
  \Psi_-=\Psi_-,  
\end{cases} 
\end{eqnarray*}
when we cross the Stokes curve in the counterclockwise direction relative to $z = 0$. 

Similarly, 
 if $\Re[\int_0^z\sqrt{\frac{Q_o(t)}{t}}dt]<0$ along the Stokes curve, one has 
\begin{eqnarray*}
  \begin{cases}
  \Psi_+= \Psi_+ \\
  \Psi_-=\Psi_-+ 2\imath \cos(\pi \sqrt{1+4Q_2(0)})\Psi_+,
\end{cases} 
\end{eqnarray*}
when we cross the Stokes curve in the counterclockwise direction relative to  $z = 0$.

Notice that \eqref{Koik} is,  when $Q_2(z)=0$ and $\dsty Q_0(z)=\frac{1}{z-z_0}, z_0\neq 0$,  a  differential equation of the form \eqref{Eq1} for $M=2$. Therefore, the connection coefficient $c_j$ of Theorem \ref{VTPE} for this case  reduces to $c_2=\pm 2\imath$, whose sign depends on the  counterclockwise or clockwise  direction relative to  $z = 0$

}
\end{remark}

 \begin{theorem}\label{Glob}
 For $M=3$, assume that all three zeros of $\rho_3(z)$ are simple. Then,
 \begin{itemize}
 \item[a)] for each zero $z_k$ of   $\rho_3(z)$,  there are three Stokes curves emanating from it and they are of the types $(1,2), (1,3)$ and $(2,3)$.  Moreover, the curve $(2,3)$ might be  a closed Jordan curve crossing $\supp {\mu^\MM}$ and the curves  $(1,2), (1,3)$ are the Jordan arcs   connecting  $z_k$ with $\infty$, see Fig.~\ref{f_example} for an example;

 \item[b)]  There exist  new Stokes curves  emanating from the intersections of the (initial)  Stokes curves  defining ordered  crossings. 
\end{itemize}

\end{theorem}

As a consequence of the geometry of the Stokes complex  we obtain a description of $\supp {\mu^\MM}$ for $M=3$. 

\begin{coro}\label{CoTh3}
 For $M=3$,    assume that the zeros of $\rho_3$ are simple. Then, 
 $$\supp {\mu^{\MM}}=\bigcup_{k=1}^3\left(\left\{\Re\left[\int_{z_k}^z(w_{1}(\xi)-w_{2}(\xi))d\xi\right]=0\right\}\bigcap \left\{\Im\left[\int_{z_k}^z(w_{1}(\xi)-w_{2}(\xi))d\xi\right]\geq 0\right\}\right). $$

\end{coro}

\section{Technical results}\label{PR}

To study the  geometry of the Stokes complex we need   to understand, for each $ 1 \leq j<j^{\prime} \leq M$,  the structure of the  complete analytic function  defined by  $\int_{z_k}^z(w_j(t)-w_{j^\prime}(\xi))d\xi$, where $z_k$ is an ordinary turning point of \eqref{OPer}. Notice that  
\begin{eqnarray} \label{RootR}
\int_{z_k}^z(w_{j}(\xi)-w_{j^{\prime}}(\xi))d\xi=  (e^{\frac{2(j-1)\pi\imath}{M}}-e^{\frac{2(j^{\prime}-1)\pi\imath}{M}})\int_{z_k}^zw_{1}(\xi)d\xi, \quad  z\in \Omega, 1 \leq j<j^{\prime} \leq M.
\end{eqnarray}

To construct the Riemann surface we proceed as follows. 
Take a branch cut in $\Omega$ consisting of $\supp {\mu^\MM}$ and  a Jordan arc $\tau$ connecting  $\infty$ and a point  $p_0\in \supp {\mu^\MM}$. For  a small disk $D\subset \Omega$  such that $\tau\cap D=\emptyset$, define  the  function element
\begin{eqnarray}\label{primitive}
b_1(z;z_k)= \begin{cases}
\int_{z_k}^z w_1(\xi)d\xi,  z\in D,\quad \rho_M(z)\neq (z-z_1)^M; \\
 \ln (z-z_1), z\in D,\quad  \rho_M(z)= (z-z_1)^M.
\end{cases}
\end{eqnarray}
For $z\in \Omega_{\tau}$,  define  $b_1(z;z_k)$  as the analytic continuation  to $\Omega_{\tau}$ of the function element $(b_1,D)$. 
% (Since $\Omega_{\tau}$ is simply connected  $b_1(z;z_k)$ is a well-defined analytic function in $\Omega_{\tau}$). 

\medskip
Further, denote by   $(\mathcal{R},\rho)$  the Riemann surface of  the complete analytic function  $\mathcal{F}$  obtained from the function element  $(b_1,D)$,  where $\rho(\mathcal{R})=\Omega$ is the projection map and $\Omega$ is the  base space,  cf. \cite[Defs. 2.7 p.215 \& 5.14 p.232]{conw84}. Take $\omega\in \supp {\mu^\MM}$  and    a Jordan arc  $\tau$   connecting   $\infty$  and $\omega$   such that $\tau\cap D=\emptyset$, where $\tau$ is oriented so that $\omega$ is the endpoint.  Set 
 \begin{eqnarray}
\begin{split}\label{Bsets}
\mathfrak{B}_0:= & b_1(\Omega_{\tau};z_k)\cup b^+_{1}(\tau;z_k), \text{ where } b^{+}_{1}(z_0;z_k)=\lim_{\begin{subarray}{c}
                           z\rightarrow z_0\\
                           z\in \tau^{+}
                          \end{subarray}}b_1(z;z_k), z_0\in \tau,\\
 \mathfrak{B}_k:=&\mathfrak{B}_0+2k\pi\imath \text{  and  }  \mathfrak{B}:=\bigcup\limits_{k\in \ZZ}\mathfrak{B}_k,
 \end{split}
\end{eqnarray}
and  the sum is understood in the sense $S+z=\{s+z:s\in S\}$.

\begin{nota}\label{boundary}
For $\omega\in \supp {\mu^\MM}$, let  $ \tau$ be  a Jordan arc connecting $\infty$ with $\omega$ and oriented towards the endpoint $\omega$. Let  $\mathfrak{G}=(b_1, \Omega_\tau)$ be a function element, and $\mathfrak{B}_0$ be as in \eqref{Bsets}.  Denote by $\dsty \mathfrak{b}_0$ the component of $\partial \mathfrak{B}_0$ contained in $\partial \mathfrak{B}$.  
 
\end{nota}

 \begin{lemma}\label{homeo}
 Let    $(\mathcal{R},\rho)$ be  the Riemann surface of the complete analytic  function $\mathcal{F}$ with the base  $\Omega$,   where $\rho(\mathcal{R})=\Omega$ is the projection map.    Then the following properties are valid: 
 \begin{itemize}
\item[a)] The map
\begin{eqnarray*}
\mathcal{F}: \mathcal{R}& \longrightarrow & \CC \\
(z,[\phi]_z) & \longrightarrow & \phi(z),
\end{eqnarray*}  
  defines a homeomorphism between  $\mathfrak{R}$ and $\mathfrak{B}$.
 
  \item[b)] The definition of the set $\mathfrak{B}$ given by \eqref{Bsets} is independent both of the choice of a point $\omega\in \supp {\mu^\MM}$ and of  the choice of a  Jordan arc connecting $\omega $ and $\infty$.  
  
   \item[c)]  If  $\rho_M(z)\neq (z-z_1)^M$ and $\mathfrak{b}_0$ is as in Notation \ref{boundary}, then  the simply connected $\mathfrak{B}$-region     is bounded by  the  polygonal curve $\partial \mathfrak{B}= \mathfrak{b}_0+2k\pi\imath, k\in \ZZ$  so that  $\mathfrak{B}$ is to the right of  $\partial \mathfrak{B}$ (where $\partial\mathcal{B}$ is oriented  from $-\imath\infty$ to $+\imath\infty$). If $\rho_M(z)= (z-z_1)^M$ then $\mathfrak{B}$ coincides with $\CC$.

\end{itemize} 
 
\end{lemma}

 \proof   
a)  Let    $[b_1]_z$ be a germ of $b_1$ at $z\in D$. Obviously,  any other germ  $[\phi]_z$   can be obtained by analytic continuation of $[b_1]_z$. Hence, if $D^{\prime}\subset \Omega$ is any simply connected subset  and  $(\Phi,D^{\prime})$ is an element in   $[\phi]_z$,  then 
\begin{eqnarray}\label{germ}
\Phi(z)=\int_{\Gamma} \frac{d\xi}{\sqrt[M]{\rho_M(\xi)}}+ b_1(z;z_k)=2n\pi\imath+ b_1(z;z_k),
\end{eqnarray}
where $\Gamma$ is a closed curve encircling both  $\supp {\mu^\MM}$ and  $z$, while the number $n=n(\infty,\Gamma)$ equals the  index (i.e. the winding number) of $\Gamma$   with respect to $\infty$. Since $\sqrt[M]{\rho_M(\xi)}\sim \xi$ as $\xi\to\infty$, the differential
$\frac{d\xi}{\sqrt[M]{\rho_M(\xi)}}$ has residue $1$ at infinity; therefore, the integral over
$\Gamma$ equals $2\pi i$ times the winding number of $\Gamma$ around $\infty$.

   Consider a cut in $\Omega$ defined by any Jordan arc $\nu$ connecting  $\infty$ and   $z_0\in \supp {\mu^\MM}$. Using  \eqref{germ}, we  have that  for  any two  function elements $(\Phi,\Omega_\nu)$  and $(\Psi,\Omega_\nu)$  contained in the respective germs $[\phi]_z$ and  $[\psi]_z$ of the complete analytic function  $\mathcal{F}$  obtained from $(b_1,D)$,  
\begin{eqnarray}\label{im0}
\Phi(z)=\Psi(z)+2k\pi\imath, \quad z\in \Omega_\nu,
\end{eqnarray}
for some $k\in \ZZ$.

\medskip
Denote by $B=Im[\mathcal{F}]=\bigcup\limits_{(z,[\phi]_z)\in \Omega\times \mathcal{F} }\mathcal{F}(z,[\phi]_z)  $ the image of $ \mathcal{R}$ under the map $\mathcal{F}$.  Observe that $\mathcal{F}$ is  injective. Indeed, let $(p_1,[\phi]_{p_1}), (p_2,[\psi]_{p_2})\in \mathcal{R}$ be such that $(p_1,[\phi]_{p_1})\neq (p_2,[\psi]_{p_2})$. We can have three alternatives: 
\begin{itemize}
\item $p_1\neq p_2, [\phi]_{p_1} \neq [\psi]_{p_2}$;
\item  $p_1\neq p_2$  and  $[\phi]_{p_1} =[\psi]_{p_2}$;
\item  $p_1=p_2$ and  $[\phi]_{p_1}\neq [\psi]_{p_1} $.
\end{itemize}
Notice that the elements in the second alternative do not belong to the domain of the map
$\mathcal F$, since in a complete analytic function a germ is always attached to a unique base point; in particular, distinct base points cannot carry the same germ, cf.\ \cite[Def.~2.1, p.~214]{conw84}. 

 For a given pair $p_1,p_2\in \Omega$, consider a cut in $\Omega$  by taking a Jordan arc $\nu$ connecting  $\infty$  and a point of  $\supp {\mu^\MM}$  and such that $p_1,p_2\in \Omega_\nu$.
 
Suppose that  $p_1\neq p_2, [\phi]_{p_1} \neq [\psi]_{p_2}$,  and $\mathcal{F}(p_1,[\phi]_{p_1})= \mathcal{F}(p_2,[\psi]_{p_2})$. Let us choose   $(\Phi,  \Omega_\nu)\in [\phi]_{p_1}$ and $(\Psi,\Omega_\nu)\in [\psi]_{p_2}$. Using \eqref{im0}  we obtain
%By continuing analytically the element $(b_1,D)$ to $ \Omega_\nu$ and using \eqref{im0}  we obtain
$$\Phi(p_1)=\Phi(p_2)+2m\pi\imath, \quad m\in \ZZ\setminus\{0\},$$
hence
\begin{eqnarray}\label{pathenc}
\int_{[p_1,p_2]_{\gamma}}  \frac{d\xi}   {m\sqrt [M]  {\rho_M(\xi)} } =2\pi\imath,
\end{eqnarray}
where  $[p_1,p_2]_{\gamma}$ is the path  connecting the points $p_1$ and $p_2$ by the arc $\gamma$   such that $[p_1,p_2]_{\gamma}\subset \Omega_\nu$; here we take the branch of $\sqrt[M]{\rho_M}$ defined in $\Omega_\nu$ satisfying
$\sqrt[M]{\rho_M(\xi)}\sim \xi$ as $\xi\to\infty$. Therefore, if $[p_2,p_1]_{\gamma^{\prime}}\subset \Omega$ is a path such that $[p_1,p_2]_{\gamma}\cup [p_2,p_1]_{\gamma^{\prime}}$ encloses $\supp {\mu^\MM}$ and satisfies  $n(\infty,[p_1,p_2]_{\gamma}\cup [p_2,p_1]_{\gamma^{\prime}})=1$, then from \eqref{pathenc} we get
\begin{eqnarray}\label{inj}
\int_{[p_2,p_1]_{\gamma^{\prime}}} \frac{d\xi}  {\sqrt[M]{\rho_M(\xi)}}= 0.
\end{eqnarray}
Take now a segment $\tau$  from an arbitrary  zero of $\rho_M$ to $\infty $ in such a way  that $p_2$ and $p_1$  are not over this segment and construct a region $U, \partial U\subset \Omega_{\tau}$, whose boundary is  defined by the Jordan arcs $\gamma_1$ and $\gamma_3$ sufficiently close to $\gamma$; $\gamma_2$ encloses $\supp\mu$  and is also sufficiently close to $\supp \mu$, and the arc $\gamma$ is near $\infty$, as shows Figure \ref{U_exa}.  By  \cite[Lem.4]{Berk82}, the map $\dsty g(z;p_2)=\int_{[p_2,z]_{\gamma^{\prime}}}\frac{d\xi}{\sqrt[M]{\rho_M(\xi)}}$ transforms each edge of $\supp \mu$ to straight lines and   $g(\tau^{+};p_2), g(\tau^{-};p_2)$ are disjoints   arcs  without  self-intersections.  Hence, $g(z;p_2)$ is injective in  the part of  boundary of $U$ limited by the arcs $\gamma_1,\gamma_2$, and $\gamma_3$. Since $\gamma$ is near the $\infty$ point,  by using the approximation $\dsty \frac{1}{\sqrt[M]{\rho_M(z)}}\sim \frac{1}{z}$ we can also show   that  $g(z;p_2)$ is also injective over $\gamma$. Since $g$ is continuous on $\overline{U}$, holomorphic in $U$, and injective on $\partial U$,
by \cite[Th.~4.5, p.~118, Vol.~II]{Mark65} it is injective in $U$; therefore,
\eqref{inj} is not possible.

  \begin{figure}[h!]
\centering
        \includegraphics[width=0.3\textwidth]{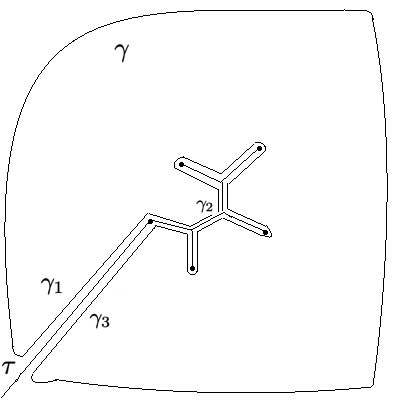}
         \caption{The  region $U$ }\label{U_exa}
\end{figure}

\medskip
Assume now  that  $p_1=p_2$ and  $[\phi]_{p_1}\neq [\psi]_{p_1} $ and take   $(\Phi, \Omega_{\nu})\in [\phi]_{p_1}$ and $(\Psi, \Omega_{\nu})\in [\psi]_{p_1}$. By  \eqref{im0} one has that $\Phi(p_1)\neq \Psi(p_1)$, i.e.  $\mathcal{F}(p_1,\phi(p_1)) \neq \mathcal{F}(p_2,\psi(p_2))$.  Therefore, we conclude that $\mathcal{F}$ is an injective map. 
  
  \medskip
  Now, the openness of non-constant holomorphic maps between Riemann surfaces implies  that $\mathcal{F}$ is an open map, see \cite[Th.6.14 p.238]{conw84}. Hence, $\mathcal{F}$ is an  injective open map and therefore $\mathcal{F}$ is a homeomorphism between   $ \mathcal{R}$ and  $B$.

We  prove now  that $B=\mathfrak{B}$. Let $\tau$ be  a Jordan arc as   in \eqref{Bsets} and  suppose that $a\in B$. Then there exists $(z,[\phi]_z)\in  \mathcal{R} $ such that $a=\mathcal{F}(z,[\phi]_z)=\phi(z)$. On the one hand, assume that $z\in \Omega_{\tau}$. Picking   $(\Phi,  \Omega_\tau)\in [\phi]_{z}$ and using \eqref{germ} one has 
\begin{eqnarray*}
\Phi(z)= 2j\pi\imath+ b_1(z;z_k),
\end{eqnarray*}
for some $j\in \ZZ$. Therefore by  \eqref{Bsets},  we get $a\in \mathfrak{B}_j$.  On the other hand, if $z\in \tau\setminus \{\omega\}$, let  $(p_n)\subset \tau^+$ be a sequence converging to $z$ and  pick  $(\Phi,  \Omega_\tau)\in [\phi]_{z}$. By \eqref{germ} one has 
\begin{eqnarray}\label{rightimplic}
\Phi(p_n)= 2j\pi\imath+b_1(p_n;z_k),
\end{eqnarray}
for some $j\in \ZZ$. Since $\dsty \lim_{\begin{subarray}{c}
                           p_n\rightarrow z\\
                           p_n\in \tau^+
                          \end{subarray}}b_1(p_n;z_k)=b_1(z;z_k)$, the relation \eqref{rightimplic}  gives that  $a=2j\pi\imath+b_1(z;z_k)$. Hence from \eqref{Bsets} we obtain  $a\in \mathfrak{B}_j$ and therefore $B\subset \mathfrak{B}$. And conversely, suppose that $a\in \mathfrak{B}$. We have that there exists $z\in \Omega$ such that  $a=2j\pi\imath+b_1(z;z_k)$, for some $j\in \ZZ$. Hence, using  \eqref{germ} one obtains that there exists $(z,[\phi]_z)$ such that $a=\phi(z)$. This completes the proof of a).

 \smallskip
    The item  b) follows  immediately  from the item a) since  $\mathfrak{B}$ is the image of the Riemann surface $\mathfrak{R}$ under the map $\mathcal{F}$. 
   
   \smallskip
  To settle item c)  let us first assume that $\rho_M(z)\neq  (z-z_1)^M$, pick $\omega\in \supp {\mu^\MM}$, and let  $ \tau$ be any Jordan arc connecting  $\infty $ and $\omega$. By  \cite[Lem. 4]{BerRull02}, any  primitive of the function $\dsty \frac{1}{\sqrt[M]{\rho_M(z)}}$ locally defined  in a simple connected domain maps the smooth curve segments of $\supp {\mu^\MM}$ to   lines. Now, the analytic continuation  of the element $(\phi_1,D)$ to $\Omega_{\tau}$   is a conformal map in   $\Omega_{\tau}$. Therefore,  the set $\mathfrak{b}_0$ is a piecewise linear curve  and  $\phi_1(\Omega_{\tau})$  is simply connected. Hence, by  item b) and  \eqref{Bsets},  $\mathfrak{B}$ is  simply connected as well. 
 On the other hand, one also has that 
\begin{eqnarray}\label{limright}
\dsty \lim_{\begin{subarray}{c}
                           z\rightarrow \infty\\
                           z\in  \Omega_{\tau}
                          \end{subarray}} \Re[\phi_1(z)] =+\infty.
\end{eqnarray}
Therefore,  by traversing $\partial \mathfrak{B}$ from $-\imath\infty$ to $+\imath\infty$ one obtains that $\mathfrak{B}$ is to the right of  $\partial \mathfrak{B}$.

% Notice that $\partial(\mathcal{B})$ defines in the Riemann sphere  a Jordan closed curve containing the infinite point. Using the Jordan curve theorem, we have that  $\partial(\mathcal{B})$ divides the complex plane in two {\color{red} simply connected} regions and that 
%

Consider now the case   $\rho_M(z)=  (z-z_1)^M$. Pick a horizontal ray $\ell$ connecting $z_1$ and $+\infty+\Im[z_1]\imath$. The analytic continuation of the function   given by  \eqref{primitive} to $\Omega_{\ell}$  defines  the   analytic function $\ln (z-z_1)$, which is a conformal  map between $\Omega_{\ell}$ and  $-\pi<\arg z< \pi$. Hence,  the assertion  follows immediately from item b) and the definition of  $\mathfrak{B}_k$ and  $\mathfrak{B}$ in    \eqref{Bsets}.  
 \lqqd

\medskip
The properties of the    complete analytic function  defined by  $\int_{z_k}^z(w_j(t)-w_{j^\prime}(t))dt$ are as follows. 

 \begin{lemma}\label{Complete}
Let $z_k$ be an ordinary turning point of \eqref{OPer},  $\tau$ be a Jordan arc  connecting  $\infty$ and a point  $z_0\in \supp {\mu^\MM}$, and  choose a small disk $D\subset \Omega$  such that $\tau\cap D=\emptyset$. Then, 
\begin{itemize}
\item[a)] The complete analytic function $\mathcal{F}_{(j,j^{\prime})}$ obtained using  the analytic continuation of a function element $ (\int_{z_k}^z(w_j(t)-w_{j^\prime}(t))dt,D)$ satisfies the equation 
$$ \mathcal{F}_{(j,j^{\prime})}=  (e^{\frac{2(j-1)\pi\imath}{M}}-e^{\frac{2(j^{\prime}-1)\pi\imath}{M}})\mathcal{F}, \quad      1 \leq j<j^{\prime} \leq M.$$

\item[b)] The map
\begin{eqnarray*}
\mathcal{F}_{(j,j^{\prime})}: \mathcal{R}_{(j,j^{\prime})}& \longrightarrow & \CC \\
(z,[\phi]_z) & \longrightarrow & \phi(z),
\end{eqnarray*}  
  defines a homeomorphism between  $\mathfrak{R}_{(j,j^{\prime})}= (e^{\frac{2(j-1)\pi\imath}{M}}-e^{\frac{2(j^{\prime}-1)\pi\imath}{M}})\mathfrak{R}$ and $\mathfrak{B}_{(j,j^{\prime})}= (e^{\frac{2(j-1)\pi\imath}{M}}-e^{\frac{2(j^{\prime}-1)\pi\imath}{M}})\mathfrak{B}$.

\end{itemize}

 \end{lemma}

\proof
Item a) is immediate from the relation \eqref{RootR} and  item b) follows from a) of Lemma \ref{homeo}.
 \lqqd

The next two  lemmas  calculate of the angles between the arcs in the support $\supp {\mu^{\MM}}$  for $M=3$. (In bigger generality this was done  in  \cite[p.155]{BerRull02}).

    \begin{lemma}\label{char}
    Let $\rho_M(z)\neq (z-z_1)^M$,   $e_i$ be an edge of $\supp {\mu^{\MM}}$ with a given orientation,  and $C_+(z), C_-(z)$ be the limiting values of $C(w)$  as $w$ approaches $z\in e_i$ from the   left  and from the right sides respectively.    Then, 
    \begin{itemize}
 \item[a)]      $\mu^{\MM}$ is absolutely continuous with respect to the Lebesgue measure; 
 \item[b)]  the unit tangent vector $\tau$ at $z\in e_i$ can be expressed as  $\dsty \tau(z)=\frac{\overline{C_-(z)}-\overline{C_+(z)}}{|\overline{C_-(z)}-\overline{C_+(z)}|}e^{\imath \frac{\pi}{2}}$;
 
  \item[c)]  if $v$ is  a vertex   of $\supp {\mu^\MM}$ such that $\rho_M(v)\neq 0$ then the  degree of $v$ is strictly greater than  $2$.

 \end{itemize}

\end{lemma}
   \proof

(a) According to \cite[Lem. 4]{BerRull02}, the measure $\mu^{\mathcal{M}}$ consists of the union of finitely many smooth curve segments $\{e_i\}_{i=1}^N$. We show that $\mu^{\mathcal{M}}$ is absolutely continuous with respect to the Lebesgue measure on every proper open subarc $(\alpha, \beta)_{\gamma} \subset e_i$. Fix an orientation for $(\alpha, \beta)_{\gamma}$ by traversing the arc from $\alpha$ to $\beta$. Let $\tau$ be a Jordan arc connecting a zero of $\rho_M$ to $\infty$, and let $b_1(z; z_k)$ for $z \in \Omega_{\tau}$ be the primitive defined in \eqref{primitive}.

From   \cite[Th. 2]{BerRull02} we have that the Cauchy transform 
\begin{eqnarray}\label{tm}
C(z)=\int \frac{d\mu^{\MM} (w)}{z-w}
 \end{eqnarray}
 of the measure $\mu^{\MM}$ 
satisfies for almost all $z\in \CC$, the equation $$\dsty C^M(z)=\frac{1}{\rho_M(z)}.$$
   Let  $$\dsty U^{\mu^{\MM}}(z)=\int\frac{1}{\log |z-w|}d\mu^{\MM}(w)$$  be the logarithmic potential of $\mu^\MM$.     Using \eqref{primitive}  and  \eqref{tm} notice that 
\begin{eqnarray}\label{potL}  
 U^{\mu^\MM}(z)=-\Re\left[b_1(z;z_k)\right]+c, \quad  z\in \Omega_{\tau},
 \end{eqnarray}
where $c\in \RR$ is a constant. Define $\dsty H(z)=-b_1(z;z_k), z\in \Omega_{\tau}$ and define   $H_+$ and $H_-$ as    the restrictions of $H$ to a neighborhoods of $(\alpha,\beta)_{\gamma}$ to the left and to the right  of   $(\alpha,\beta)_{\gamma}$ respectively. Further denote by  $H_+(\omega+0)$ and by $H_-(\omega+0)$  the  non-tangential limits of $H_+$ and  $H_-$ respectively when $z\rightarrow \omega\in (\alpha,\beta)_{\gamma}$, cf.  \cite[p. 89-90 Ch. II]{ST}.   It follows from  \eqref{potL} and \cite[Th.1.4 Ch. II]{ST}  that for every $z_0,z_1\in (\alpha,\beta)_{\gamma}$ such that  $z_0$ precedes $z_1$,
\begin{eqnarray}\label{musupp}
\mu^{\MM}((z_0,z_1)_{\gamma})=\frac{1}{2\pi \imath}\left(H_+(z_1+0)-H_+(z_0+0)-H_-(z_1-0)+H_-(z_0-0)\right).
\end{eqnarray}
    Notice that for any  edge $e_i$ of $\supp {\mu^{\MM}}$, the functions $H_+(z)$ and $H_-(z)$   are of the class $C^{1}$ for $z\in \mathring{e}_i$. Therefore, if $r(t), t\in [a,b]$ is a parameterization of the smooth arc defined by $e_i$, one obtains that $H_+(r(t))$ and $H_-(r(t))$  are   absolutely continuous  in $[a,b]$.   \cite[Prop. 3.32]{foll} implies that $\mu^{\MM}$ is an absolutely continuous measure with respect to the Lebesgue measure $dz$ on the arc $(\alpha,\beta)_{\gamma}$ which proves  a).

%Hence,  for every $\epsilon>0$ there exists a $\delta>0$ such that   
%\begin{eqnarray*} 
% \sum_{k=1}^{n}|H_+(a_k+0)-H_+(b_k+0)|+|H_-(a_k-0)-H_-(b_k-0)| <\epsilon,
%\end{eqnarray*}
%for every disjoint class $\{(a_k,b_k), 1\leq k\leq n\}$ of open subarcs in $(\alpha,\beta)$ such that  $\dsty \sum_{k=1}^{n}m((a_k,b_k))<\delta$.

%
% by \cite[Th. 6.11]{rud}  the measure  $\mu $ is absolutely continuous with respect to the Lebesgue measure $m$.   It follows that   if $(\alpha,\beta)\subset e_i$ is a proper  subarc, there exists a unique $f\in L^1(\alpha,\beta)$ such that  
%\begin{eqnarray}\label{abs0}
%\int_{(\alpha,\beta)}d\mu=\int_{(\alpha,\beta)}fdz.
%\end{eqnarray}
%Now, let $E\subset e_i$ be  any Borel measurable subset, we have that there exists   a collection  $(a_n,b_n)$  of proper subarcs of $e_i$ such that    $\dsty \bigcup_{n\in \NN}(a_n,b_n)=E$.  By writing $\dsty  E_n=\bigcup_{k\leq n}(a_k,b_k)$ one has that $E_n\subset E_{n+1}$ and $\dsty \bigcup_{n\in \NN}E_n=E$. By the Lebesgue dominated convergence theorem \cite[Th. 1.34]{R87} and \eqref{abs0} we obtain
%\begin{eqnarray}\label{abs1}
%\int_{E}d\mu=\lim_{n\rightarrow\infty}\int_{E}\mathbbm{1}_{E_n}d\mu=\lim_{n\rightarrow\infty}\int_{E}\mathbbm{1}_{E_n}fdz,
%\end{eqnarray}
%and 
%$$\lim_{n\rightarrow\infty}\int_{E}\mathbbm{1}_{E_n}fdz=\int_{E}fdz,$$

(b)  By the preceding item a), there exists a unique $f\in L^1(\alpha,\beta)$ such that  $d\mu^{\MM}=fdz$. Hence by the Sokhotski-Plemelj formula \cite[(17.2) p.42]{Muskh53}, 
\begin{eqnarray}\label{Cauchy}
f=\frac{1}{2\pi \imath} (C_--C_+).
\end{eqnarray}
Since $\mu^{\MM}$ is a positive measure, then using \eqref{Cauchy} one immediately obtains that 
$$\frac{1}{2\pi \imath} \int_{(a,b)}(C_-(z)-C_+(z))dz \geq 0,  $$
where $(a,b)_{\gamma}$ is any subarc of $\supp {\mu^{\MM}}$. Hence, the tangent vector $\tau$  to any smooth subarc $(a,b)_{\gamma}$ of $\supp {\mu^{\MM}}$ at $z$ satisfies the relation 
\begin{eqnarray*} 
\arg[\tau(z)]=\arg[\overline{C_-(z)}-\overline{C_+(z)}]+ \frac{\pi}{2} \mod 2\pi, \quad z\in (a,b),
\end{eqnarray*}
which completes the proof of b).

c)  Observe that the degree of $v$ is necessarily greater than $1$, since otherwise $v$ is a branch point and by our assumption,  $v$ is not a zero of $\rho_M$.    Suppose now  that the degree of $v$  equals   $2$ and set   $e_1=(v_1,v), e_2=(v,v_2)$.   By \cite[Lem. 4]{BerRull02},  $\supp{\mu^{\MM}}$ is the union of finitely many smooth curve segments and  has connected complement. Hence we may assume   that  the Jordan arc $e_1\cup e_2$ is not differentiable at $v$.   
 
%  and suppose that $v,v_1$ and $v_2$ are not zeros of $\rho_M$.
 
 Consider a small Jordan curve $\mathcal{C}$  enclosing $v$ and consider the arc $int(\mathcal{C})\cap (e_1\cup e_2)$ oriented by traversing it  from $int(\mathcal{C})\cap e_1$ to $int(\mathcal{C})\cap e_2$.  Denote by $V_+$ the left   and by $V_-$   the right sides.    Let  $C_+(z)$ and $C_-(z)$ be the limiting values of $C(w)$  when $w$ approaches $z$ from the left and from the right sides of the arc respectively. Notice  that  $C_+(z)$ and $C_-(z)$ are continuous functions in a neighborhood of $z=v$.  
 
%  Since $\supp \mu$ has connected complement we have that any two points in  $V_+$  (or  $V_-$) can be connected  through a curve contained in the complement of  $\supp \mu$  and enclosing a fixed number of zeros of $\rho_M$
 
 By item b),  
\begin{eqnarray}\label{tau1}
\lim_{\begin{subarray}{c}
                           z\rightarrow v\\
                           z\in e_i
                          \end{subarray}}\arg[C_+(z)-C_-(z)]=\arg[\tau_i(v)]+ \frac{\pi}{2} \mod 2\pi,
\end{eqnarray}  
where $\tau_i(z)$ is  the unit tangent vector   at $z$. 
On the other hand, the functions  $C_+(z)$ and $C_-(z)$ are continuous  in a neighborhood of $z=v$. Therefore,   from \eqref{tau1} one obtains  
\begin{eqnarray}\label{tau2}
\arg[\tau_1(v)]- \arg[\tau_2(v)] =0 \mod 2\pi.
\end{eqnarray}

By our  assumption, the Jordan arc $e_1\cup e_2$ is not differentiable at $v$. Thus the relation \eqref{tau2} 
is impossible,  which is a contradiction.  This completes the proof that the  degree of $v$ is strictly greater than  $2$. 
 \lqqd

\begin{lemma}\label{angles}
For $M=3$,  suppose that $\rho_3(z)$ has three distinct roots. Then, $\mu^{\MM} $ consists of one vertex and three edges connecting the zeros of $\rho_3(z)$ contained in the convex hull of the zeros of $\rho_3(z)$, see Figure~\ref{supp}.  Moreover,  all angles  between the arcs $\mathfrak{r}_i, i=1,2,3$ of   $\supp {\mu^\MM}$ at the common intersection point $v$ are  equal to $\frac{2\pi}{3}$.
 
\end{lemma}

\proof  a)   Let   $\{z_1,z_2,z_3\}$ be the  zeros of $\rho_3$. Using the connectivity of    $\supp {\mu^{\MM}$}, we can find   a Jordan arc    connecting  $z_1$ and $z_2$. Since $z_1$ and $z_2$ are branch points,  $\supp {\mu^{\MM}}$ contains  at least one vertex  $v\notin \{z_1,z_2\}$ and this vertex can be connected to $z_3$ through a Jordan arc $\mathfrak{r}_3$.  Now,  $\mathfrak{r}_3$ does not contain vertices other than $v^{\prime}$. Indeed, if  $v^{\prime}\in \mathfrak{r}_3$ is such a vertex, then by item c) of  Lemma \ref{char} the degree of  $v^{\prime}$ is strictly greater than 2, hence there exists $z_{\omega}\in \supp {\mu^{\MM}}$ such that $z_{\omega}\notin \{z_1,z_2,z_3\}$ and $z_{\omega}$ is an extreme point. Consequently, $z_{\omega}$ is also a branch point, which a contradiction. Therefore,   $\mathfrak{r}_3$ is an edge of  $\supp{ \mu^{\MM}}$.  Using a similar argument one obtains that the Jordan arcs $\mathfrak{r}_1$ and  $\mathfrak{r}_2$  connecting $z_1, v$  and $z_2, v$  respectively are edges of $\supp{ \mu^{\MM}}$. This proves the  assertion that  $\supp{\mu^{\MM}}$ consists of three smooths arcs $\mathfrak{r}_1,\mathfrak{r}_2,\mathfrak{r}_3$ and  a vertex $v$ of degree $3$. The assertion that  $\supp{\mu^{\MM}} $ is contained   in the interior of the convex hull of the zeros of $\rho_3$ follows  from \cite[Th. 3]{BerRull02}.

Let us show that all angles at the common intersection point $v$ are equal to $\frac{2\pi}{3}$.  Consider the orientation in each $\mathfrak{r}_i$ obtained by traversing each arc so that $v$ is the end point. Denote by $C_{i,+}(z)$ and $C_{i,-}(z)$ the limiting values of the Cauchy transform $C(z)$   at $z\in \mathfrak{r}_i$ from the left and from the right respectively.  Now, consider a small disk $D$ centered at $v$ and notice that $C$ is a continuous function in the  connected component $V_i, i=1,2,3$ of $D\setminus (D\cap (\mathfrak{r}_1\cup\mathfrak{r}_2\cup\mathfrak{r}_3))$, see Figure \ref{supp}. Set 
$$\lim_{\begin{subarray}{c}
                           z\rightarrow v\\
                           z\in \mathfrak{r}_i
                          \end{subarray}}C_{i,+}(z) =C_{i,+}, \quad  \lim_{\begin{subarray}{c}
                           z\rightarrow v\\
                           z\in \mathfrak{r}_i
                          \end{subarray}}C_{i,-}(z) =C_{i,-}. $$

Without loss of generality, we assume that the edges $\{\mathfrak{r}_i\}$ are ordered such that a counterclockwise circuit around $v$ encounters them in the sequence $(\dots, \mathfrak{r}_3, \mathfrak{r}_1, \mathfrak{r}_2, \mathfrak{r}_3, \mathfrak{r}_1, \dots)$, as illustrated in Figure~\ref{supp}.

 %\begin{figure}[h!]
%\centering
%        \includegraphics[width=0.25\textwidth]{sup3.png}
%         \caption{ A typical  picture of  $\supp{\mu^{\MM}}$ for $M=3$.}\label{f_00}
%\end{figure}

 By the continuity of  $C$, in each region $V_i$ one has 
\begin{eqnarray}\label{contcond}
C_{i,-}=C_{i-1,+}.
\end{eqnarray} 

 By \eqref{contcond} and b) of Lemma \ref{char},  if  $\tau_i$ is the tangent vector of $\mathfrak{r}_i$ at $v$, then 
 \begin{eqnarray}
\begin{split}\label{argexp}
 \arg\left[\frac{\tau_i}{\tau_{i-1}} \right]&= \arg\left[\frac{\dsty \lim_{\begin{subarray}{c}
                           z\rightarrow v\\
                           z\in \mathfrak{r}_i
                          \end{subarray}}(C_{i,+}(z)-C_{i,-}(z))}{\dsty \lim_{\begin{subarray}{c}
                           z\rightarrow v\\
                           z\in \mathfrak{r}_j
                          \end{subarray}}(C_{i-1,+}(z)-C_{i-1,-}(z) )}  \right]\\
 &  =   \arg\left[\frac{C_{i,+}-C_{i,-}}{C_{i,-}-C_{i-1,-} }  \right].
\end{split}
 \end{eqnarray}  
Writing 
\begin{eqnarray}
\begin{aligned}\label{argexp1}
C_{i,+}-C_{i,-}&=\rho e^{\frac{\theta_0+2k\pi}{3}\imath}-\rho e^{\frac{\theta_0+2(k+1)\pi}{3}\imath},\\
C_{i,-}-C_{i-1,-}&=\rho  e^{\frac{\theta_0+2(k+1)\pi}{3}\imath}-\rho e^{\frac{\theta_0+2(k+2)\pi}{3}\imath},
\end{aligned}  
\end{eqnarray}
and   using \eqref{argexp} and \eqref{argexp1} one obtains
$$ \left|  \arg\left[\frac{\tau_i}{\tau_{i-1}}\right] \right|=\frac{2}{3}\pi.$$
 Hence, the angle $\alpha_{i,i-1}$ between the arcs $\mathfrak{r}_i,\mathfrak{r}_{i-1}$ at the vertex  $v$ equals 
\begin{eqnarray}\label{angleij}
\alpha_{i,i-1}=\frac{2\pi}{3}.
\end{eqnarray}

 \lqqd

\begin{lemma}\label{Fj}
For $M\geq 2$, let  $U\subset\Omega$ be an open subset, $V\subset \CC$, and   $\phi(z,\eta):U\times V\rightarrow \CC$  be a formal power series of the form $\dsty \phi(z,\eta)=\sum_{k=1}^{\infty}\phi_k(z)\eta^{-k}, \phi_k\in \mathcal{H}(U)$. Then   $v=e^{\eta \int^{z} (\phi(t,\eta)+ w_j(t))dt}$ is  a solution of 
\begin{eqnarray}\label{DF}
v^{(M)} +\sum_{k=1}^{M-1} \frac{\rho_k }{\rho_M } v^{(k)} - \eta^{M} \frac{v }{\rho_M  }=0,
\end{eqnarray}
if and only if, $\phi$ is a solution of 
$$F_j(w,\ldots,w^{(M-1)},z,\eta)=0, $$
where
\begin{eqnarray*}
F_j(x_0,\ldots,x_{M-1},z,\eta)=-\frac{1}{\rho_M(z) }+\sum_{l=1}^{M}\sum_{u=1}^{l}\sum_{c_0,\ldots, c_{l-u} \in \pi(l,u)}(x_0+w_j(z))^{c_0}\cdots (x_{l-u}+w^{(l-u)}_j(z))^{c_{l-u}}f_{c_0,\ldots, c_{l-u}}(z,\eta).
\end{eqnarray*}
Here $\pi(l,u)$ stands for the set of partitions of $l$ into $u$ summands,
\begin{eqnarray*}
 c_0+\cdots+(l-u+1)c_{l-u}=l\\
  c_0+\cdots+c_{l-u} = u,\\
  c_0,\ldots,c_{l-u}\geq 0,
\end{eqnarray*}
and $w_j$ is as in \eqref{wj}. Finally,  $\dsty f_{c_0,\ldots,c_{l-u}}(z,\eta)=\frac{l!}{c_0!\cdots c_{l-u}!(1!)^{c_0}\cdots ((l-u+1)!)^{c_{l-u}}}\eta^{u-M}\frac{\rho_l(z) }{\rho_M(z) }$. 

\end{lemma}
\proof
For any integer $1\leq k\leq M$,   denote by $P_{k}(y^{\prime},\dots ,y^{(k)})$ the polynomial in the   variables $(y^{\prime},\dots,y^{(k)})$  defined  by the relation

$$ P_{k}(y^{\prime},\dots ,y^{(k)})={e^{-y}} \, \left(e^{y} \right)^{(k)}. $$

According to the Fa\`a di Bruno formula \cite[Th.A p.137]{comtet74}, $P_{k}(y^{\prime},\dots ,y^{(k)})$ can be expressed as
\begin{equation} \label{faa}
P_{k}(y^{\prime},\dots ,y^{(k)})=\sum_{l=1}^{k} {\bf B}_{k,l}(y^{\prime},\ldots,y^{(k-l+1)}).
\end{equation}

  The relation \eqref{faa}   and the variable change   $v=e^{y}$ provide  that \eqref{DF} can be expressed as

\begin{equation} \label{DiffEqn03}
\sum_{k=1}^{M} \frac{\rho_k }{\rho_M } P_{k}(y^{\prime},\dots ,y^{(k)})- \frac{ \eta^{M}}{\rho_M }=0.
\end{equation}

By multiplying the relation  \eqref{DiffEqn03} by $\dsty \frac{1}{ \eta^{M}}$ and using the expression  for $P_k$, one has that  $v=e^y, 
y^{\prime} =(w+ w_j) \eta$ is a solution to \eqref{DF} if and only if $w$ is a solution of the  Riccati equation
\begin{eqnarray}\label{Operator01}
 \sum_{k=1}^{M} \eta^{-M}\frac{\rho_k }{\rho_M } P_{k}(\eta(w+ w_j),\dots ,\eta(w+ w_j)^{(k-1)})- \frac{1}{\rho_M  }= 0,
\end{eqnarray}
where 
\begin{eqnarray*}
P_{k}(\eta(w+ w_j),\dots ,\eta(w+ w_j)^{(k-1)})=\sum_{l=1}^{k}\eta^{l}{\bf B}_{k,l}(w+ w_j,\ldots,(w+ w_j)^{(k-l)}).
\end{eqnarray*}
Changing the variables $x_0=w,\ldots ,x_{M-1}=w^{(M-1)}$ and  rearranging appropriately, we have that the relation \eqref{Operator01}  can be  equivalently expressed as 
\begin{eqnarray}\label{Alt}
 \sum_{k=0}^{M-1}\eta^{-k}\sum_{l=M-k}^{M}\frac{\rho_l }{\rho_M }{\bf B}_{l,M-k}(x_0+w_j,\ldots,x_{l-M+k}+w^{(l-M+k)}_j)- \frac{1}{\rho_M  }=0.
 \end{eqnarray}
Further we get 
\begin{multline*}
-\frac{1}{\rho_M }+\sum_{k=0}^{M-1}\eta^{-k}\sum_{l=M-k}^{M}\frac{\rho_l }{\rho_M }{\bf B}_{l,M-k}(x_0+w_j,\ldots,x_{l-M+k}+w^{(l-M+k)}_j)=-\frac{1}{\rho_M }+ \\
\sum_{k=0}^{M-1}\sum_{l=M-k}^{M}\sum_{c_0,\ldots c_{l-M+k} \in \pi(l,M-k)}\eta^{-k}\frac{\rho_l }{\rho_M }\frac{l!}{c_0!\cdots c_{l-M+k}!1!\cdots (l-M+k+1)!}(x_0+w_j)^{c_0}\cdots (x_{l-M+k}+w^{(l-M+k)}_j)^{c_{l-M+k}}\\
= -\frac{1}{\rho_M }+   \sum_{l=1}^{M}\sum_{u=1}^{l}\sum_{c_0,\ldots c_{l-u} \in \pi(l,u)}(x_0+w_j)^{c_0}\cdots (x_{l-u}+w^{(l-u)}_j)^{c_{l-u}}\frac{l!}{c_0!\cdots c_{l-u}!1!\cdots (l-u+1)!}\eta^{u-M}\frac{\rho_l }{\rho_M },
\end{multline*}
where $\pi(l,u)$ stands for the set of partitions of $l$ into $u$ summands,
\begin{eqnarray*}
0&\leq &c_0+\cdots+(l-u+1)c_{l-u}=l\\
  0&\leq & c_0+\cdots+c_{l-u} \leq u,
\end{eqnarray*}
which finishes the proof. \lqqd 

Let $\zeta,f_n, n\geq 0$ be holomorphic functions defined in  a domain $U\subset \CC$.  The formal power series 
 $$\dsty  \exp(\eta \zeta(z))\sum_{n=0}^{\infty}f_n(z)\eta^{-(n+\frac{1}{2})},$$
is said to be \textcolor{blue}{ \emph {pre-Borel-summmable in $U$}} if for each compact set $K\subset U$, there exists $A_k$ and $C_K$ such that 
$$\sup_K|f_n(z)|<A_KC_K^n\Gamma(1+n).$$

 \begin{lemma}\label{pre-borel}
Let $U\subset \Omega$ be a simply connected  open set. Then,   any WKB-formal solution is pre-Borel-summable in    $U$.  
 \end{lemma}

\proof
For convenience, we introduce the parameter $\epsilon = 1/\eta$. Following \cite[Chap.II, Prop. 2.1.2]{SKK73}, it suffices to establish the pre-Borel summability of the formal series $\phi(z,\epsilon) = \sum_{k=1}^{\infty} S_k(z) \epsilon^{k}$ on every compact subset $K \subset U$. Here, the coefficients $S_k$ are determined such that $\phi$ is a formal solution of the Riccati equation   
$$
F_j\left(x, x', x'', \dots, x^{(M-1)}, z, \epsilon\right) = 0,
$$
where $F_j$ is defined as in Lemma~\ref{Fj}. This operator can be expanded as
$$
F_j(x_0, x_1, \dots, x_{M-1}, z, \epsilon) = \sum_{|m|=0}^{N} x_0^{m_0} \cdots x_{M-1}^{m_{M-1}} F_{j,m_0, \dots, m_{M-1}}(z, \epsilon),
$$
where $\dsty F_{j,m_0, \dots, m_{M-1}}(z, \epsilon) = \sum_{m=0}^{N_{m_0,m_1,\ldots,m_{M-1}}} F_{j,m_0, \dots, m_{M-1};m}(z) \epsilon^m$ is a polynomial in $\epsilon$ with coefficients in $\mathcal{H}(U)$. Note that $F_{j,m_0, \dots, m_{M-1}}$ is of Gevrey order $1$ in $\epsilon$ uniformly in $z \in U$; specifically, there exist constants $K_1, K_2 \geq 0$ such that 
\begin{equation}\label{Gev}
|F_{j,m_0, \dots, m_{M-1}}(z)| \leq K_1 m! (K_2)^m,
\end{equation}
for $z \in U$ and all indices within the specified ranges.

The proof rests on Sibuya's theorem regarding the Gevrey summability of formal power series, see Theorem~\ref{Sib} in Appendix \S~\ref{append:III}. (An interested reader should take a look at this material before reading the proof). Using  Lemma \ref{Fj}, a direct  calculation shows that 
$$
\frac{\partial F_j}{\partial x_{M-1}}\left(\phi, \phi', \dots, \phi^{(M-1)}, z, \epsilon\right) = \epsilon^{M-1}.
$$
Hence, condition \eqref{Dxl} for  Theorem~\ref{Sib} is satisfied. Furthermore, applying \eqref{Alt} yields
\begin{align}
\label{x0} \frac{\partial F_j}{\partial x_0}\left(\phi,\frac{d\phi}{dz},\frac{d^2\phi}{dz^2}, \ldots, \frac{d^{M-1}\phi}{dz^{M-1}},z,\epsilon\right) &= M(\phi+w_j)^{M-1} + \sum_{k=1}^{M-1} \epsilon^k \sum_{l=M-k}^{M} \frac{\rho_l}{\rho_M} \frac{\partial \mathbf{B}_{l,M-k}}{\partial x_0}, \\
\label{x1} \frac{\partial F_j}{\partial x_1}\left(\phi,\frac{d\phi}{dz},\frac{d^2\phi}{dz^2}, \ldots, \frac{d^{M-1}\phi}{dz^{M-1}},z,\epsilon\right) &= \epsilon \frac{M(M-1)}{2}(\phi+w_j)^{M-2} + \sum_{k=2}^{M-1} \epsilon^k \sum_{l=M-k}^{M} \frac{\rho_l}{\rho_M} \frac{\partial \mathbf{B}_{l,M-k}}{\partial x_1},
\end{align}
where the arguments of the Bell polynomials $\mathbf{B}_{l,M-k}$ are $(\phi+w_j, \dots, \phi^{(l-M+k)}+w_j^{(l-M+k)})$.

From \eqref{x0}, \eqref{x1}, \eqref{ahm}, and \eqref{hk} of \S~\ref{append:III}, we have that  $h_1=0$ and $h_2=1$. This corresponds to Case $A$ in \cite[Th. 1.2.1]{Si03}. Consequently, from \eqref{x0} the relation \eqref{CRelCas} in the Appendix reduces to 
$$
\left.\mathcal{T}[y]\right|_{\epsilon=0} = M w_j^{M-1} y.
$$
Additionally, we find $m_{h_1}=0$ and $m_{h_2}=1$. From \eqref{rhonu} in the Appendix, we obtain
\begin{eqnarray}\label{rho2}
\rho_2=1.  
\end{eqnarray}
 By Theorem~\ref{Sib}, the formal series $\phi$ possesses a Gevrey order of $\max(1/\rho_2, s)$, which by \eqref{Gev} and \eqref{rho2} equals $1$. Thus, $\phi$ is pre-Borel summable on every compact subset $K \subset U$. It follows that any formal WKB solution is likewise pre-Borel summable on $K \subset U$.
\lqqd

%The following classic results are fundamental in the proof of Lemma \ref{ACont}. 

%  \begin{theorem}[P\'olya \cite{polya29}]\label{Polyat}
%  The series $\dsty \sum_{k=0}^{\infty}f_nz^n$, whose domain of convergence is the unit disk $\mathbb{D}=\{z|:|z|<1\}$ extends analytically to $\CC$, possibly except the arc $\partial \mathbb{D}\setminus \{z\in \mathbb{D}:\arg[z]<\sigma \}$ if and only if there exists an entire function of exponential type $\varphi(\zeta)$ interpolating the coefficients $f_n$ such that  the indicator function  $\dsty h_{\varphi}=\limsup_{r\rightarrow\infty}\frac{\ln|\varphi(re^{\imath\theta)})|}{r} $  for the  entire function $\varphi$ satisfies 
%   $$h_{\varphi}(\theta)\leq \sigma |\sin \theta|, \quad |\theta| \leq \pi. $$
%  
%  
%  
%\end{theorem}

 \begin{lemma}\label{Ext}
Given $z^*\notin \mathcal{Z}$, $z_k\in \mathcal{Z}$, let  $U_{z_k}$ be a neighborhood of $z_k$, and    
$$\dsty \psi_j(z,\eta,z^*)=\eta^{-\frac{1}{2}}\exp\left(\int_{z^*}^zS(\zeta,\eta)d\zeta\right), z\in U_{z_k}\cap\Omega$$  be a WKB-solution of \eqref{OPer}.  Then, there exists $\alpha_n\in \ZZ_{\geq 0}$ such that, 
$$S_n((z-z_k)^{M})(z-z_k)^{\alpha_n}\in \mathcal{H}(U_{z_k}), \quad  n\geq 0.$$
 \end{lemma} 

  \proof
   Using  $\dsty \eta=\frac{1}{\epsilon}$ and  letting  $u_1=y,u_2=\epsilon y^{\prime},\ldots,u_{M}= \epsilon^M y^{(M)}$, we have that  \eqref{Eq1} can be expressed as 
 \begin{equation}\label{Alt0}
\epsilon u^{\prime}=M(z,\epsilon)u,
\end{equation}
where $u=(u_1,\ldots ,u_M)^t$ and 
\begin{eqnarray*}
\dsty M(z,\epsilon)&=& \det 
\begin{pmatrix}
           0        &1           &     0       &   \cdots        &          0  \\
        0            & 0         &1            &      \cdots      &          0   \\
            &      &    \vdots           &        &     \\
                        &      &    \vdots           &        &  1   \\
 -\frac{1}{\rho_M(z)} &  -\frac{\rho_{1}(z)}{\rho_M(z)}\epsilon^{M-1}  &  \cdots      & & -\frac{\rho_{M-1}(z)}{\rho_M(z)}\epsilon
\end{pmatrix}\\
&=&\sum_{k=0}^{M-1}M_k(z)\epsilon^k.
\end{eqnarray*}  

Notice that  $M(z,0)$ is diagonalizable in $\Omega$ with the eigenvalues given by $(w_j(z))_{j=1}^{M}$, where $w_j$'s are defined as in \eqref{wj}; note also  that $w_j$ is a solution of the equation $\rho_M(z)w^M-1=0$.    Let $\mathcal{M}(U_{z_k})(w_1,\ldots,w_M)$ be the smallest functional field  containing  $w_1,\ldots,w_M$ and  the set of meromorphic functions in $U_{z_k}$,   cf. \cite[Th.3 p.512 \& Def.1 p.517]{DuFo04} \& \cite[Case I p.315]{Fra21}. Then  by \cite[Cor.25 p.494]{DuFo04}  there exists an invertible matrix $Q(z)$ with elements in the field $\mathcal{M}(U_{z_k})(w_1,\ldots,w_M)$  such that 
$$M_0(z)=Q(z)\Lambda_0(z)Q^{-1}(z),$$
where $\Lambda_0(z)=diag(w_1(z),\ldots, w_M(z))$.

Hence, if $U$ is a fundamental system for  \eqref{Alt0}, we have that  using the substitution $U=QY$ one obtains the matrix equation
\begin{equation}\label{Alt2}
 \epsilon Y^{\prime}=(\Lambda_0(z)+A(z,\epsilon))Y, \quad z\in U_{z_k}\cap \Omega,
\end{equation}
where $\dsty A(z,\epsilon)=\sum_{k=1}^{M-1}A_k(z)\epsilon^k$ is a matrix polynomial of degree $M-1$ in the variable $\epsilon$,   with $A_k(z)=Q(z)M_k(z)Q^{-1}(z)$ and $A_k\in \mathcal{M}(w_1,\ldots,w_M)$.

Using the formal series   $\dsty \Lambda(z,\epsilon)=\sum_{k=0}^{\infty}\Lambda_k(z)\epsilon^k, P(z,\epsilon)=\sum_{k=0}^{\infty}P_k(z)\epsilon^k, P_0(z)=I$, substituting  $Y=PZ$ in \eqref{Alt2}, and collecting powers of $\epsilon$, we obtain the recurrence relation 
$$\Lambda_0 P_r-P_r\Lambda_0=\sum_{s=0}^{r-1}(P_s\Lambda_{r-s}-A_{r-s}P_s)+  P^{\prime}_{r-1}, r>0, $$
which can be expressed as 
$$\Lambda_0 P_r-P_r\Lambda_0=\Lambda_r-H_r.$$
Here $H_r$ depends only on $P_j,P_j^{\prime}$ and $\Lambda_j$ with $j<r$. Choosing $\Lambda_r=diag(\lambda_{j,j}(r)), \lambda_{j,j}(r)=h_{j,j}(r)$, where $H_r=(h_{j,j}(r))$, we define  $P_r$ as the  solution of the non-homogeneous Sylvester equation 
\begin{eqnarray}\label{Sylvs}
\Lambda_0 P_r-P_r\Lambda_0=C
\end{eqnarray}
in the field $\mathcal{M}(U_{z_k})(w_1,\ldots,w_M)$, where $C$ is an anti-diagonal matrix with entries in $\mathcal{M}(U_{z_k})(w_1,\ldots,w_M)$. Notice that the entries of the matrix in the leftt-hand side of \eqref{Sylvs} are 
$$(\lambda_i(0)-\lambda_j(0))p_{i,j}(r),$$
where $P_r=(p_{i,j}(r))$. Using the fact that the eigenvalues of $\Lambda_0$ are distinct, we immediately obtain  that for each $r$, there exists a unique solution  $P_r\in \mathcal{M}(U_{z_k})(w_1,\ldots,w_M)$.  

 Therefore, the transformation  $Y=PZ$ reduces  \eqref{Alt2} to 
\begin{equation}\label{Alt3}
\epsilon Z^{\prime}=\Lambda Z,
\end{equation}
 where  $\dsty \Lambda(z,\epsilon)=\sum_{k=0}^{\infty}\Lambda_k(z)\epsilon^k, P(z,\epsilon)=\sum_{k=0}^{\infty}P_k(z)\epsilon^k, P_0(z)=I$. 
 
 Relations  \eqref{Alt2} and \eqref{Alt3} imply that a formal fundamental system for \eqref{Alt0} is given by $U=QPe^{\frac{1}{\epsilon}\int^z\Lambda(t,\epsilon)dt}$.  Recalling that $u_1=y$ and that the matrices $Q$ and $\Lambda$ have entries in  the field $\mathcal{M}(U_{z_k})(w_1,\ldots,w_M)$, the first row of $U$ gives the desired WKB-solution
 $$\exp\left( \epsilon^{-1}\int_{z^*}^zS_{0}(\zeta)d\zeta\right)\sum_{n\geq 0}\epsilon^{n+\frac{1}{2}}\psi_{n}(z).$$ 
 
 Finally, by Lemma \ref{pre-borel} we have that $S(z,\epsilon)=\dsty \frac{1}{\epsilon}\sum_{k=0}^{\infty}S_k(z)\epsilon^{k}$  is pre-Borel-summable  in every open simply  connected subset $U\subset \Omega$, which means that for every $n\geq 0$,  one has that  $S_n((z-z_k)^{M})(z-z_k)^{\alpha_n}\in \mathcal{H}(U_{z_k})$, for some $\alpha_n\in \ZZ_{\geq0}$.  \lqqd

% where $y_0(z)=$

% 
% 
% Writing  $QP=\alpha_0(z)+\alpha_1(z)\epsilon+\alpha_2(z)\epsilon^2+\ldots $, we have by \cite[5c) p.141]{comtet74},
% $$\dsty QP=\exp\left( \ln\alpha_0(z)+\sum_{k=1}^{\infty}\epsilon^k {\bf L}_n\left(\frac{\alpha_1(z)}{\alpha_0(z)},\ldots,\frac{\alpha_n(z)}{\alpha_0(z)}\right)   \right),$$
%where ${\bf L}_n$ is the $n$-th logarithmic polynomial cf. \cite[Th.A p.140]{comtet74}.
% 
% 

 \begin{lemma}\label{Sng}
Let  $z^*\notin\mathcal{Z}$ be a fixed reference point and  $z\in \Omega$.  Then,  the Borel transform $\psi_{j,B}(z,y,z^*)$   has singularities  at $\dsty y=-\int_{z^*}^{z}w_j(t)dt$ and the remaining singularities are included in  the set  $\dsty \left\{(z,y,z^*):y=-\int_{z^*}^{z}w_l(t)dt, l=1,\ldots,M; l\neq j\right\}$. 
 \end{lemma} 
 
  \proof Let $$\dsty \psi_{j,B}(z,y,z^*)=\sum_{n\geq 0}\frac{f_n(z,z^*)}{\Gamma(n+\frac{1}{2})}(y+y_0(z,z^*))^{n-\frac{1}{2}}$$ be the Borel transform of $\psi_j$, 
where  $
\dsty y_0(z,z^*)=\int_{z^*}^zw_j(\zeta)d\zeta$.  By Lemma \ref{pre-borel}, there exists a simply connected open set
$\Omega^\circ \subset \Omega \setminus \mathcal Z$ containing $z^*$, such that  the expression given by \eqref{BTR} is  an analytic  solution of \eqref{PDE} when $z\in \Omega$ and    $|y+y_0(z,z^*)|<\delta$, for sufficiently small  $\delta>0$. We shall regard $\psi_{j,B}(z,y,z^*)$ as a single-valued holomorphic
function on a domain of the form
$\Omega^{\circ} \times \Omega_Y$, where $\Omega^{\circ} \subset
\Omega \setminus \mathcal Z$ is a fixed simply connected domain
containing $z^*$, and $\Omega_Y \subset \mathbb C$ is maximal with
respect to analytic continuation in the $y$–variable.

Notice that $\psi_{j,B}(z,y,z^*)$   has the obvious  singularity  $\dsty y=-\int_{z^*}^{z}w_j(t)dt$.  Suppose that    $\psi_{j,B}(z,y,z^*)$ has another  singularity at a point $(z,y)=(z^{\prime},y^{\prime}), z^{\prime}\in \Omega$. All subsequent arguments are carried out in a
simply connected domain $\Omega^\circ \subset \Omega \setminus \mathcal Z$
that contains both $z^*$ and $z'$. We have that   $(z^{\prime},y^{\prime})$ is in the singular support of $\psi_{j,B} $, considered  as a  distribution which we denote by $u_{\psi}$. From \cite[Cor.7.2.2 p.249]{DuHo94} and \cite[Th.6 p.44 (complex version)]{Ni73}, one obtains that   the bicharacteristic curve  $(z(t),y(t),\zeta(t),\epsilon(t))$  defined by the equations  \eqref{Ha1}-\eqref{Ha5} and emanating from $\left(z^{\prime},y^{\prime},  \frac{1}{\sqrt[M]{\rho_M(z^{\prime})}} ,1\right)$ belongs to  $\mathcal{W\!F}(u_{\psi})$, which is the \textcolor{blue}{\emph{wave front}} of $u_{\psi}$. A straightforward calculation shows that %{\color{red} Is better now?}
\begin{eqnarray}\label{BicAn}
(z(t),y(t),\zeta(t),\epsilon(t))=\left(\Psi_l^{-1}(t,z^{\prime}),-Mt+y^{\prime}, \frac{1}{\sqrt[M]{\rho_M(z(t))}} ,1\right),
\end{eqnarray}
where $\Psi_l^{-1}$ denotes the inverse  of   $\dsty \Psi(z,z^{\prime})=\frac{1}{M}\int_{z^{\prime}}^z\frac{dt}{\sqrt[M]{\rho_M(t)}}, z\in \Omega$. Here, the index $l$ refers to the branch of the chosen  root. The relation  \eqref{BicAn} implies that the $y$-component of $(z(t),y(t),\zeta(t),\epsilon(t))$ can be expressed as 
\begin{eqnarray}\label{BicAn1}
 y_l(z)=-\int_{z^{\prime}}^z\frac{dt}{\sqrt[M]{\rho_M(t)}}+y^{\prime}.
\end{eqnarray}
  Since  $\mathcal{W\!F}(u_{\psi})$ is a  closed set (cf. \cite[\S 8 p.41]{Ni73}), we obtain that the    point $(z^*,y_j(z^*))$ is in the singular support of the distribution $u_{\psi}$.

 We want to prove that
\[
\lim_{\substack{z\to z^*\\ z\in\Omega^{\circ}}} y_l(z)=0.
\]
Indeed, by the definition of $y_0(z,z^*)$ and of the coefficients
$f_n(z,z^*)$, and since $\Omega^{\circ}$ is simply connected and contains
$z^*$, we obtain
\begin{eqnarray}
\begin{aligned}\label{f_n}
\lim_{\substack{z\to z^*\\ z\in\Omega^{\circ}}} y_0(z,z^*)&=0,\\
\lim_{\substack{z\to z^*\\ z\in\Omega^{\circ}}} f_0(z,z^*)&=1,\\
\lim_{\substack{z\to z^*\\ z\in\Omega^{\circ}}} f_n(z,z^*)&=0,
\qquad n\geq 1.
\end{aligned}
\end{eqnarray}

Using the expression for $\psi_{j,B}$ together with
\eqref{f_n}, we conclude that $(z^*,0)$ is the unique point of the
singular support of the distribution $u_{\psi}$ lying above the base
point $z=z^*$.

Therefore, taking the limit $z\to z^*$ with $z\in\Omega^{\circ}$ in
\eqref{BicAn1}, we obtain
\[
y^{\prime}=\int_{z^{\prime}}^{z^*}\frac{dt}{\sqrt[M]{\rho_M(t)}}.
\]
This shows that, for any $z^{\prime}\in\Omega^{\circ}$, the
$y$-components of the singularities of the Borel transform
$\psi_{j,B}$ with reference point $z^*$ satisfy
\[
y_l=-\int_{z^*}^{z^{\prime}}\frac{dt}{\sqrt[M]{\rho_M(t)}}.
\]

 \lqqd

 The following two theorems play an important role in the estimation of the growth of the analytic continuation of the power series  of Lemma \ref{ACont}. 

\begin{theorem}[LeRoy \& Lindel\"of \cite{lind05}, \cite{Dien57} pp. 340-345, \cite{Ara85}]\label{Leroy}
 For a function $\varphi \in \mathcal{H}(\{z:\Re[z]\geq 0\})$ of exponential type $\sigma<\pi$, the series 
 $$\dsty f(z)=\sum_{k=0}^{\infty}\varphi(n)z^n$$
 admits an analytic continuation to the sector $\CC\setminus \{z\in \mathbb{D},|\arg[z]|\leq 2\sigma \}$. Moreover, $f(z)\rightarrow 0 $ when $z\rightarrow \infty$ in each angular domain $\CC\setminus \{z\in \mathbb{D},|\arg[z]|\leq \frac{\beta}{2} \}, \beta \in (2\sigma,2\pi)$.
  
\end{theorem}

\begin{theorem}[Arakelyan \cite{Ara85} Th. 1.1]\label{Ara}
For $|z|<1$ and  $\sigma\in [0,\pi)$, a power series  
$\dsty \sum_{k=0}^{\infty}f_nz^n$ 
admits an analytic continuation to the sector  $\CC\setminus \{z\in \mathbb{D},|\arg[z]|\leq \sigma \}$,  if and only if there exists a function $\phi\in \mathcal{H}(\{z: \Re[z]> 0\})$ of the inner exponential type at most $\sigma$, such that 
$$c_n=\phi(n), \quad n=0,1,\ldots $$
\end{theorem}

To establish the Borel summability of $ \dsty \psi_j(z,\eta,z^*), z^*\notin \mathcal{Z}$  we introduce the Stokes curves  
\begin{eqnarray}\label{ST*}
 \mathcal{S}_{z*,j}=\bigcup_{j^{\prime}: j^{\prime}\neq j}\left \{z\in \Omega: \Im \left[\int_{z^*}^z(w_{j}(\xi)-w_{j^{\prime}}(\xi))d\xi\right]=0 \right \},
\end{eqnarray}
relative to $z^*$.

As observed in \cite{Berk82}, for higher order operators a Stokes phenomenon
can occur in a neighborhood of intersection points of Stokes curves
emanating from turning points, cf.\ Notation~\ref{LbSt}. Consequently, the
Borel sum of $\psi_j(z,\eta,z^*)$, with $z^*\notin\mathcal Z$, may fail to be
well-defined along Stokes curves produced by ordered crossings. We therefore
consider ordered crossings of curves in $\mathcal S_{z^*,j}$ together with
Stokes curves emanating from the turning points $z_k\in\mathcal Z$, and define
the corresponding collection of potentially obstructing curves as in
Definition~\ref{OC}. We denote by $\mathcal N_{ext}=\mathcal N\cup\mathcal
N_{z^*}$ the union of all such curves, where $\mathcal N_{z^*}$ refers to
those arising when $z^*$ is taken as reference point.

Although the class of operators considered here does not allow one to relate
different turning points through a single bicharacteristic strip, the
intersection of Stokes curves still signals a genuine analytic obstruction.
Indeed, at such points distinct phase integrals may exhibit identical
exponential growth, so that competing exponential contributions coexist after
analytic continuation. This coincidence provides the mechanism by which a
Stokes phenomenon may arise, independently of any direct microlocal
connectivity between turning points.

Accordingly, we assume that at ordered intersections of Stokes curves—including those involving curves emanating from turning points—the Borel transform may exhibit an obstruction to analytic continuation in the $z$-variable. Following the classical exact WKB analysis of \cite{Berk82}, we therefore exclude neighborhoods of these ordered intersections and restrict our analysis to points $z \notin \mathcal{N}_{ext}$, where the Stokes phenomenon is not expected to occur.

Away from $\mathcal N_{ext}$, standard results on analytic continuation with
respect to parameters ensure that the singularities of the Borel transform
depend holomorphically on $z$ and remain separated from the Laplace direction.
In particular, no obstruction to Borel summation arises outside
$\mathcal N_{ext}$.

From a microlocal viewpoint, propagation of singularities implies that the
wavefront set of $\psi_{j,B}$ is transported along the bicharacteristic strip.
Away from ordered intersections of Stokes curves, the microlocal contributions
associated with distinct phases remain separated after projection, preventing
any interaction in the Laplace direction. At points belonging to
$\mathcal N_{ext}$, however, this microlocal separation may fail at the level
of projections, allowing different phase contributions to interact and
producing a possible obstruction to microlocal analyticity in the Laplace
direction.

 \begin{lemma}\label{ACont}
Fix $z\in \Omega$ and $z\notin \mathcal{S}_{z*,j}\cup \mathcal{N}_{ext}$.   Then,   $\psi_{j,B}(z,y,z^*)$  can be analytically  continued through the horizontal strip  $[-y_0(z,z^*)\pm \imath\delta, t\Re[-y_0(z,z^*)]+\imath(\Im[-y_0(z,z^*)]\pm\delta) ], t>0, |\delta|<\delta_0 $. Moreover, $\psi_{j,B}(z,y,z^*)\rightarrow 0$ as $y\rightarrow \infty$ through the strip, in particular $\psi_{j,B}(z,y,z^*), y\in [-y_0(z,z^*),+\infty-\imath \Im[y_0(z,z^*)])$ is of exponential type.
 \end{lemma} 
 
 \proof   By Lemma~\ref{Sng}, for $z\notin\mathcal{S}_{z^*,j}$, the Borel transform
$\psi_{j,B}(z,y,z^*)$ has only a finite number of isolated singularities
in the $y$--plane, located at the points
\[
y=-\int_{z^*}^{z} w_l(t)\,dt,\qquad l=1,\ldots,M.
\]
Moreover, no singularity other than $y=-y_0(z,z^*)$ lies on the horizontal
line through $y=-y_0(z,z^*)$. Consequently, $\psi_{j,B}(z,y,z^*)$ can be
analytically continued along the ray
\[
y=-y_0(z,z^*)+t,\qquad t>0,
\]
and in a horizontal strip around this ray, avoiding the remaining singular
points.

Furthermore, consider the shifted series
\[
(-y+y_0(z,z^*))^{\frac12}\psi_{j,B}(z,-y,z^*)
=\sum_{n=0}^{\infty}\frac{f_n(z,z^*)}{\Gamma(n+\frac12)}(-y+y_0(z,z^*))^{n}.
\]
By Theorem~\ref{Ara}, the coefficients
$c_n=\Gamma(n+\tfrac12)^{-1}f_n(z,z^*)$
admit a holomorphic interpolation in $\{\Re\xi>0\}$ of inner exponential
type strictly less than $\pi$. Consequently, Theorem~\ref{Leroy} implies
that $\psi_{j,B}(z,y,z^*)\to 0$ exponentially as $y\to+\infty$ within the
horizontal strip considered above.

Finally, we explain the role of the assumption $z\notin\mathcal{N}_{ext}$.
By construction, the set $\mathcal{N}_{ext}$ collects points where ordered
intersections of Stokes curves occur, including those involving curves
emanating from turning points. At such points, distinct phase integrals may
exhibit identical exponential growth after analytic continuation, so that
their associated Borel singularities may align with the Laplace integration
direction. In this situation, the singularity of $\psi_{j,B}(z,y,z^*)$ at
$y=-y_0(z,z^*)$, which is guaranteed by Lemma~\ref{Sng}, may fail to be
isolated with respect to horizontal continuation, and the construction of a
horizontal strip avoiding all singularities in the $y$--plane can no longer
be ensured.

Away from $\mathcal{N}_{ext}$, no such alignment occurs. The singularities of
the Borel transform remain separated from the Laplace direction and depend
holomorphically on the parameter $z$. Consequently, $\psi_{j,B}(z,y,z^*)$
admits analytic continuation in a horizontal strip containing the ray
$y=-y_0(z,z^*)+t$, $t>0$, along which the Laplace integral defining the Borel
sum is well defined. This justifies the exclusion of $z\in\mathcal{N}_{ext}$
in the statement of the lemma and completes the proof.

 \lqqd

\begin{proposition}\label{sum_z*}
Take $z^*\in \Omega$ and let
\[
\dsty \psi_j(z,\eta,z^*)
=\exp\left( \eta \int_{z^*}^z S_0(\zeta)\,d\zeta \right)
\sum_{n=0}^{\infty}\frac{\phi_n(z)}{\eta^{n+\frac{1}{2}}},
\qquad z\in U_{z_k}\cap\Omega,
\]
be a WKB solution of \eqref{OPer} with reference point $z^*$.
Then $\psi_j$ is Borel summable at $z$ provided that
\[
z\notin \mathcal{S}_{z^*,j}\cup \mathcal{N}_{ext}
\]
\end{proposition}

\proof
By Definition~\ref{BT}, the Borel sum $\Psi_j(z,\eta,z^*)$ is well-defined
whenever the Laplace integration path
\[
[-y_0(z,z^*),+\infty-\imath \Im y_0(z,z^*)]
\]
does not intersect any singularity of the Borel transform
$\psi_{j,B}(z,y,z^*)$ and the latter decays exponentially along this path.

By Lemma~\ref{Sng}, the Borel transform $\psi_{j,B}(z,y,z^*)$ has a
singularity at
\[
y=-\int_{z^*}^{z} w_j(t)\,dt,
\]
while the remaining possible singularities are located at the points
\[
y=-\int_{z^*}^{z} w_l(t)\,dt, \qquad l\neq j.
\]
If $z\notin \mathcal{S}_{z^*,j}$, none of these singularities lies on the
horizontal ray issuing from $y=-y_0(z,z^*)$.

Finally, we clarify the role played by the assumption $z\notin\mathcal{N}_{ext}$.
 By Lemma~\ref{ACont}, for such values of $z$ the
Borel transform $\psi_{j,B}(z,y,z^*)$ admits analytic continuation in a
horizontal neighborhood of the Laplace direction, and no new singularities
appear in this region as $z$ varies outside $\mathcal{N}_{ext}$. In
particular, no collision or alignment of Borel singularities with the
Laplace direction occurs.

It follows that the Laplace integral defining the Borel sum of
$\psi_j(z,\eta,z^*)$ is well defined for all $z\notin\mathcal{N}_{ext}$ and
depends analytically on $z$ in this domain. Moreover, since the singularity
structure of $\psi_{j,B}$ remains stable under variation of $z$ away from
$\mathcal{N}_{ext}$, no Stokes phenomenon is encountered as long as $z$ does
not cross a curve in $\mathcal{N}_{ext}$.

This establishes the existence   of the Borel
sum of $\psi_j(z,\eta,z^*)$ for all $z\notin\mathcal{N}_{ext}$, which completes
the proof of the proposition. \lqqd

\section{Proofs of the main results}\label{Proo}

%denote by  $[\alpha,\beta]$   the closed line segment connecting the points $\alpha$ and $\beta$, and  let $\lambda_{1,0}$ be  the branch of the root of $\dsty \frac{1}{\sqrt{(\zeta-\alpha)(\zeta-\beta)}}, z\in \CC\setminus [\alpha,\beta]$ which coincides with $\dsty \frac{1}{z}$ near $\infty$.
 
 We start with Theorem~\ref{Main}.

 \begin{proof} 
a)  Choose $z^*\in \Omega$, $S_n$ as in Definition \ref{WKBS}, and $\mathcal{D}(z_k)$ defined in Notation  \ref{fitn}.  Using Lemma \ref{Ext}, for $n\in \mathcal{D}(z_k)$, we can  write 
\begin{eqnarray}\label{RS}
S_n(z)=\mathfrak{b}_n(z)+\mathfrak{g}_n(z),
\end{eqnarray}
where $\dsty \mathfrak{b}_{n}(z)=\sum_{l=0}^na_{n,l} (z-z_k)^{\frac{m_{n,l}}{M}}$ such that $\frac{m_{n,l}}{M}\in \QQ_{\leq -1}$,      and $\mathfrak{g}_n$ is  such that $\mathfrak{g}_n((z-z_k)^{M})(z-z_k)^{\beta_n}\in \mathcal{H}(U_{z_k})$, for some integer $\beta_n$ satisfying $0\leq \beta_n<M$.  

Here and in what follows, the index $j$ labels a fixed branch of the
characteristic root used in the construction of the WKB solution $\psi_j$.
For each fixed $j$ and $n\geq 0$, Lemma~\ref{Ext} yields a local expansion of
$S_{j,n}$ near $z_k$ of the form
\[
S_{j,n}(z)=\sum_l a_{j,n,l}\,(z-z_k)^{m_{n,l}/M},
\qquad m_{n,l}\in\mathbb{Z},
\]
where the set of exponents $m_{n,l}/M$ is independent of the choice of branch,
while the coefficients $a_{j,n,l}$ may depend on $j$ through the chosen sheet.
No explicit relation between the coefficients on different branches is assumed
or needed in the sequel. By \eqref{RS}, we obtain   
\begin{eqnarray}\label{Dec}
\int_{z^*}^zS(\zeta,\eta)d\zeta = r_j(z,\eta,z^*)+s_j(z^*,\eta),
\end{eqnarray}
where
\begin{eqnarray*}
\dsty r_j(z,\eta,z^*)&=& 
  \sum_{\begin{subarray}{c}
n\geq 0: n\in \mathcal{D}(z_k),\\
m_{n,0}<-M
 \end{subarray}}\eta^{-n-1}\int_{\infty}^z\sum_{l=1}^na_{j,n,l} (\zeta-z_k)^{\frac{m_{n,l}}{M}}d\zeta + \sum_{\begin{subarray}{c}
n=0: n\in \mathcal{D}(z_k),\\
m_{n,0}=-M
 \end{subarray}}\eta^{-n-1}  a_{j,n,0}\ln (z-z_k)+\\
& & \sum_{\begin{subarray}{c}
n\geq 0,\\
 n\in \mathcal{D}(z_k)
 \end{subarray}} \eta^{-n-1} \int_{z^*}^z\mathfrak{g}_{j,n}(\zeta)d\zeta  +
  \sum_{\begin{subarray}{c}
n\geq 0,\\
 n\notin \mathcal{D}(z_k)
 \end{subarray}} \eta^{-n-1}\int_{z^*}^zS_n(\zeta)d\zeta,  \\ 
\dsty s_j(\eta,z^*)&=& -\sum_{\begin{subarray}{c}
n=0:\\
m_{n,0}<-M
 \end{subarray}}^{\infty}\eta^{-n-1}\int_{\infty}^{z^*}\sum_{l=1}^na_{j,n,l} (\zeta-z_k)^{\frac{m_{n,l}}{M}}d\zeta  + \sum_{\begin{subarray}{c}
n=0:\\
m_{n,0}=-M
 \end{subarray}}^{\infty} \eta^{-n-1}a_{j,n,0}\ln (z^*-z_k).
\end{eqnarray*}
Notice that  we have
$${\tt Fp}\int_{z_k}^zS(\zeta,\eta)d\zeta =\lim_{z^*\to z_k}r_j(z,\eta,z^*).$$
 
% We have that its Borel sum $\Psi_{j}(z,\eta,z^*)$ can not vanish an infinite number of times for $\eta>\eta_0$, $\eta_0$ large. Indeed,  
%$$\dsty \eta^{\frac{1}{2}}e^{-\eta\int_{z^*}^zS_0(\zeta)d\zeta}\Psi_{j}(z,\eta,z^*)\sim \exp\left(\sum_{n=1}^{\infty}\frac{\int_{z^*}^zS_n(\zeta)d\zeta}{\eta^{n-1}}\right),$$
%the assumption that $\Psi_{j}(z,\eta,z^*)$ that  vanishes an infinite number of times  implies that the term  $\dsty \sum_{n=1}^{\infty}\frac{\int_{z^*}^zS_n(\zeta)d\zeta}{\eta^{n-1}}$ has an accumulation point at $\eta=\infty $ of logarithmic singularities. On the other hand,  
%$$e^{\int_{z^*}^z S_1(\zeta)d\zeta}=\lim_{\eta\to\infty}\exp\left(\sum_{n=1}^{\infty}\frac{\int_{z^*}^zS_n(\zeta)d\zeta}{\eta^{n-1}}\right),$$ 
%  which is a contradiction. 

 Fix  $z^* \in U_{z_k}\cap \Omega$ and $z\in \Omega$.  By Proposition \ref{sum_z*},   $\psi_j(z,\eta,z^*)$ is Borel  summable provided that $z\notin \mathcal{S}_{z*,j}\cup \mathcal{N}_{ext}$.  By \cite[Th.188 p.237]{Hard49} the formal series  $\widetilde{\psi}_{j}(z,y,z^*)=\eta^{\frac{1}{2}}e^{-\int_{z^*}^z S_0(\zeta)d\zeta}\psi_{j}(z,y,z^*)$ is Borel summable and  from  \cite[Prop.4.109 p.108]{Cost08},     
\[
\ln\bigl(\widetilde{\psi}_j(z,\eta,z^*)\bigr)
=
\ln \phi_0(z,z^*)
+
\sum_{k\ge 1}\frac{(-1)^{k-1}}{k}
\left(
\sum_{n\ge 1}
\frac{\phi_n(z,z^*)}{\phi_0(z,z^*)}\,\eta^{-n}
\right)^k,
\]
 is also Borel summable for $z\notin \mathcal{S}_{z*,j}\cup \mathcal{N}_{ext} $ and   large positive $\eta$. Therefore,  for $z\notin \mathcal{S}_{z*,j}\cup \mathcal{N}_{ext}$,  the formal series of $\ln (  \psi_j(z,\eta,z^*))$ is Borel summable.  For $z^*\in \Omega$ close to $z_k$,   express   the Borel transform $\psi^S_{j,B}$ of $\dsty \ln (  \psi_j(z,\eta,z^*))$   as 
\begin{eqnarray}\label{Rs}
\psi^S_{j,B}(z,y,u^{*M}+z_k)&=&\psi^S_{r,j,B}(z,y,u^{*M}+z_k) + \psi^S_{s,j,B}(y,u^{*M}+z_k)-\frac{1}{2}\frac{\gamma+\ln y}{y},
\end{eqnarray}
where   $u^*$ belongs to a small   punctured  neighborhood of $0$,  $\gamma$ is the  Euler-Mascheroni constant, and $\psi^S_{r,j,B}$ and $\psi^S_{s,j,B}$  are the Borel transforms of the functions $r_j$ and $s_j$ respectively defined in \eqref{Dec}. Here the term $-\frac12(\gamma+\ln y)/y$ arises from the Borel transform
of the logarithmic factor $-\frac12\ln\eta$ in
$\ln(\psi_j(z,\eta,z^*))$, since
\[
\mathcal{B}\!\left[-\tfrac12\ln\eta\right](y)
= -\tfrac12\,\frac{\gamma+\ln y}{y}.
\]

 By expanding 
\begin{eqnarray*}
\psi^S_{r,j,B}(z,y,u^{*M}+z_k)  &=&f(y,z)+\sum_{n\in \ZZ_{\geq 0}}  f_n(y)u^{*(n+1)},\\
\psi^S_{s,j,B}(y,u^{*M}+z_k)  &=&f_{-1}(y)\ln u^*+ \sum_{n\in \ZZ_{<-1}}  f_n(y)u^{*(n+1)},
\end{eqnarray*}
 and by using   the Cauchy integral formula 
\begin{eqnarray*}
f_n(y)&=&\frac{1}{2(n+1)\pi\imath }\int_{\Gamma(0,\epsilon)}\frac{\frac{d}{d\zeta}\psi^S_{j,B}(z,y,\zeta^{M}+z_k)}{\zeta^{n+1}}d\zeta, \, n\neq -1,\\
 f_{-1}(y)&=&\frac{1}{2\pi\imath }\int_{\Gamma(0,\epsilon)}\frac{d}{d\zeta}\psi^S_{j,B}(z,y,\zeta^{M}+z_k)d\zeta, n=-1,
\end{eqnarray*}
for the Taylor coefficients,  we get  that  the  set of singularities of    $f_n(y), n\geq 0$ and   $f(y,z)-\frac{1}{2}\frac{\gamma+\ln y}{y}$  with respect to the variable $y $ is included in the set of singularities of $\psi^S_{j,B}$.  Here $\Gamma(0,\epsilon)$ is a small circle of radius $\epsilon$ around $0$.  Therefore,   the set of singularities of  $\psi^S_{r,j,B}(z,y,z^*)-\frac{1}{2}\frac{\gamma+\ln y}{y}$  in the $y$ variable is included in the set of singularities of $\psi^S_{j,B}(z,y,z^*)$. Hence, a singularity of   the Borel transform $\psi_{r,j,B}(z,y,z^*) $ of $\eta^{-\frac{1}{2}}\exp(r_j(z,\eta,z^*))$  is a singularity of   $\psi_{j,B}(z,y,z^*)$   as well. By definition, since
\[
r_j(z,\eta,z^*) \;=\; \eta \int_{z^*}^z w_j(t)\,dt \;+\; O(1),
\qquad \eta\to+\infty,
\]
the Borel transform $\psi_{r,j,B}(z,y,z^*)$ of 
$\eta^{-\frac12}\exp\!\big(r_j(z,\eta,z^*)\big)$
has a singularity at
\[
y \;=\; -\int_{z^*}^z w_j(t)\,dt .
\]
Since $\psi_{r,j,B}(z,y,z^*)$ has a singularity at
$y=-\int_{z^*}^z w_j(t)\,dt$ and, by the above Cauchy argument, any singularity
of $\psi_{r,j,B}(z,y,z^*)$ is also a singularity of $\psi_{j,B}(z,y,z^*)$,
it follows that the singular set of $\psi_{r,j,B}(z,y,z^*)$ is contained in
the singular set described in Lemma~\ref{Sng}.

\medskip
Using Theorems \ref{Leroy} and \ref{Ara}, we get 
\begin{eqnarray}\label{ExpTyp}
\lim_ {\begin{subarray}{c}
                           y\rightarrow +\infty\\
                           y\in \mathcal{S}
                          \end{subarray}} \psi_{r,j,B}(z,y,z^*) =0,
\end{eqnarray}
where   $\mathcal{S}=[-y_0(z,z^*)\pm \imath\delta, t\Re[-y_0(z,z^*)]+\imath(\Im[-y_0(z,z^*)]\pm\delta) ], t>0, |\delta|<\delta_0 $ and $z\notin \mathcal{S}_{z*,j}\cup \mathcal{N}_{ext}$.        This implies   that $\psi_{r,j,B}(z,y,z^*)$  is also of exponential type when $y\in \mathcal{S}$.

By definition, one has 
\begin{eqnarray}\label{*tok}
\psi_{j,B}(z,y,z_k)= \lim_{z^*\to z_k}  \psi_{r,j,B}(z,y,z^*).
\end{eqnarray}
To prove that the singularities of $\psi_{j,B}(z,y,z_k)$ occur at $\dsty \left\{(z,y):y=-\int_{z_k}^{z}\frac{dt}{\sqrt[M]{\rho_M(t)}}\right\}$  we need  to analyze the  limit in \eqref{*tok}.

  Take    a smooth  path $\mathfrak{c}$ connecting a fixed point  $z^{\prime}$ and $z_k$ as shown in  Figure \ref{Small_V}  and let the reference point $z^*$ vary along the arc $\kappa\subset \mathfrak{c}$ so that $\kappa$ is contained in  small neighborhood  $V_{z_k}$ of  $z_k$, see Figure \ref{Small_V}.
\begin{figure}[h!]
\centering
 \begin{subfigure}{0.45\textwidth}
        \includegraphics[width=0.45\textwidth]{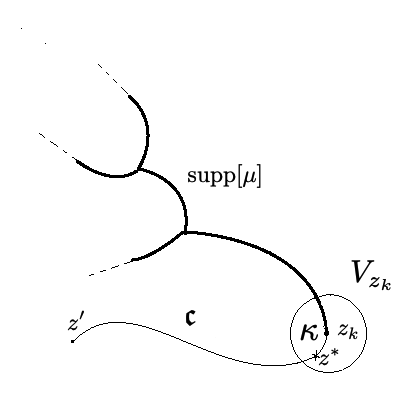}
         \caption{ The neighborhood $V_{z_k}$ and the  integration path $\mathfrak{c}$ for the function $\psi_{r,j,B}$. }\label{Small_V}
 \end{subfigure}
 \hfill
 \begin{subfigure}{0.45\textwidth}
        \includegraphics[width=0.45\textwidth]{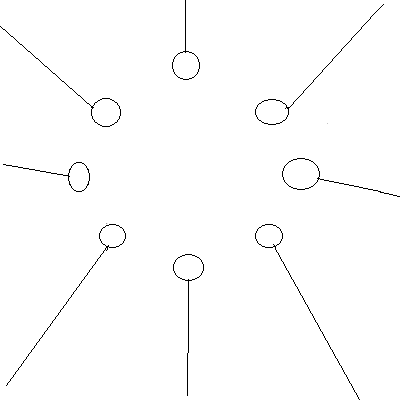}
         \caption{ Region  $ \dsty \Omega_Y$}\label{Small_V_Y}
 \end{subfigure}
 \caption{ Regions $V_{z_k}$ and  $ \dsty \Omega_Y$ in the $z$- and the $y$-spaces respectively. }
\end{figure}

 If $z^*\in  \kappa$,  we can find  small neighborhoods $V_{y_l(z^{\prime},z^*)}, l=1,\ldots, M$ of $\dsty  y_l(z^{\prime},z^*)=- \int_{z^*}^{z^{\prime}}\frac{dt}{\sqrt[M]{\rho_M(t)}}$ such that  for each  $z^*\in  \kappa$,  $\psi_{r,j,B}(z^{\prime},y,z^*)$  is   analytic in the variable $y$  in the region $ \dsty \Omega_Y= \CC\setminus \left[\bigcup\limits_{l=1}^M V_{  y_l(z^{\prime},z^*) }\cup \gamma_Y\right] $, where   $\gamma_Y$ is the  branch cut for $\psi_{r,j,B}$  consisting of  Jordan arcs connecting each  $y_l(z^{\prime},z^*)$ with $\infty$, see Figure \ref{Small_V_Y}. Hence, when   $\rho_M(z)\neq (z-a)^M$, we can write  
 $$\dsty  \psi_{r,j,B}(z^{\prime},y,z^*)=\sum_{n\geq 0}a_n(z^{\prime},z^*)y^n, \quad |y|<\delta(z^{\prime}),$$
where  each term $a_n(z^{\prime},z^*)$  is continuous for $z^*\in  \overline{ \kappa}$ (this follows from the representation \eqref{Dec}, since the $z^*$–dependence
of $r_j(z',\eta,z^*)$ comes only from integrals with lower limit $z^*$ of
locally integrable powers $(\zeta-z_k)^{\alpha}$ with $\alpha>-1$ and of
holomorphic functions, while all singular terms are integrated from $\infty$
and are independent of $z^*$),    which implies that $\psi_{r,j,B}(z^{\prime},y,z^*)$ is also continuous when $(y,z^*)\in \{|y|<\delta_0<\delta(z^{\prime})\}\times \overline{ \kappa}$.  Therefore,   by the Heine–Cantor theorem,  $\psi_{r,j,B}(z^{\prime},y,z^*)$ is uniformly continuous in $\overline{ \kappa}$. Hence, $\psi_{r,j,B}(z^{\prime},y,z^*)$ is a family of functions depending on the variable $z^*$ which   uniformly converges as $z^*$ approaches  $z_k$ along $\kappa$. By \cite[Th. 15.12 p.333 Vol.I]{Mark65},  $\psi_{r,j,B}(z^{\prime},y,z_k)$ is analytic  when $y$ varies in compact subsets of $\Omega_Y$. Since we can consider  arbitrary small neighborhoods $V_{z_k}$ and  $V_{y_l(z^{\prime},z^*)}, l=1,\ldots, M$  then using \eqref{*tok} we obtain that   $\psi_{j,B}(z^{\prime},y,z_k)$  has no  singularities other than $\dsty \left(z^{\prime},-\int_{z_k}^{z^{\prime}}\frac{dt}{\sqrt[M]{\rho_M(t)}}\right)$.  

The statement that  $\psi_{j,B}(z,y,z_k)$ is of exponential type when  $$y\in [-y_0(z,z_k), t\Re[-y_0(z,z_k)]+\imath(\Im[-y_0(z,z_k)]) ], t>0; z\notin \mathcal{S}_{z_k,j}\cup \mathcal{N}$$ is immediate from \eqref{ExpTyp}, \eqref{*tok}, and the definition \eqref{ST*}. Indeed, since in $\Omega$ the $M$ branches of $\sqrt[M]{\rho_M}$ are fixed and
single--valued, the action function
\[
z \longmapsto \int_{z^*}^z \sqrt[M]{\rho_M(t)}\,dt
\]
depends continuously on $z^*\in\kappa$. Consequently, the level sets defining the
Stokes curves $\mathcal{S}_{z^*,j}$ vary continuously and converge, as
$z^*\to z_k$ along $\kappa$, to the Stokes curves $\mathcal{S}_{z_k,j}$, so that
the admissible sectors in the $y$--plane persist in the limit.

\medskip

b)   We analyze the propagation of singularities for the Borel--transformed
operator $\mathcal L_B$ using bicharacteristic strips in the sense of
Definition~\ref{BS}.

From \eqref{Ha4}, we have $\epsilon(t)=\mathrm{const}$ along any
bicharacteristic strip. Without loss of generality, we normalize
$\epsilon=1$. The characteristic equation \eqref{Ha5} then yields
\[
\zeta=\frac{1}{\sqrt[M]{\rho_M(z)}},
\]
where a branch is fixed once and for all.

By item~(a), singularities of $\psi_{j,B}(z,y,z_k)$ occur on the analytic
curves
\[
y=-\int_{z_k}^{z}\frac{dt}{\sqrt[M]{\rho_M(t)}}.
\]
By the analytic propagation of singularities
\cite[Cor.~7.2.2]{DuHo94}, these singularities propagate along
bicharacteristic strips emanating from points of the form
\[
\left(a,-\int_{z_k}^{a}\frac{dt}{\sqrt[M]{\rho_M(t)}},
\frac{1}{\sqrt[M]{\rho_M(a)}},1\right),
\qquad a\neq z_k.
\]

The Hamilton--Jacobi system \eqref{Ha1}--\eqref{Ha5} reduces to
\[
\begin{cases}
\dfrac{dz}{dt}=M\sqrt[M]{\rho_M(z)},\\[0.3em]
\dfrac{dy}{dt}=-M,\\[0.3em]
\dfrac{d\zeta}{dt}=-\dfrac{\rho_M'(z)}{\rho_M(z)},
\end{cases}
\]
with initial conditions
\[
z(0)=a,\quad
y(0)=\displaystyle\int_{z_k}^{a}\frac{dt}{\sqrt[M]{\rho_M(t)}},\quad
\zeta(0)=\frac{1}{\sqrt[M]{\rho_M(a)}}.
\]

The equation for $y$ integrates explicitly as
\[
y(t)=-Mt-\int_{z_k}^{a}\frac{dt}{\sqrt[M]{\rho_M(t)}},
\]
so that $y$ is strictly monotone along the bicharacteristic curve. Since $\frac{dy}{dt}=-M\neq0$, the projection
$t\mapsto(z(t),y(t))$ is injective. Hence the bicharacteristic curve
$\mathcal{BC}(t)$ has no self-intersections. 
Moreover, the equation for $z(t)$ defines an analytic function as long as
$\rho_M(z)\neq 0$. Standard results on analytic ODEs
\cite[\S13.7]{Ince65} imply that $z(t)$ can be continued uniquely until it
reaches a zero of $\rho_M$.

Therefore, along any bicharacteristic curve $\mathcal{BC}(t)$,
no new singularities in the $z$--variable can arise except at points where
$\rho_M(z)=0$, that is, at the turning points $z_k\in\mathcal Z$.
This proves item~(b).

\medskip

c)  Follows immediately from item a). \lqqd
 
 \end{proof}

  Next we prove Theorem \ref{VTPE}.  Our strategy follows  \cite[pp. 5-6]{H15vir} and \cite[pp. 24-25]{kawai2005algebraic}. To understand how the Borel sum $\Psi_{1,1}$ changes when we move from $a$ to $b$ we study the analytic continuation of the Borel transform $\psi_{1,B}(z,y,z_k)$  from $z=a$ to $z=b$ as shown in Figure \ref{FK}.

\begin{figure}[h!]
\centering
        \includegraphics[width=0.4\textwidth]{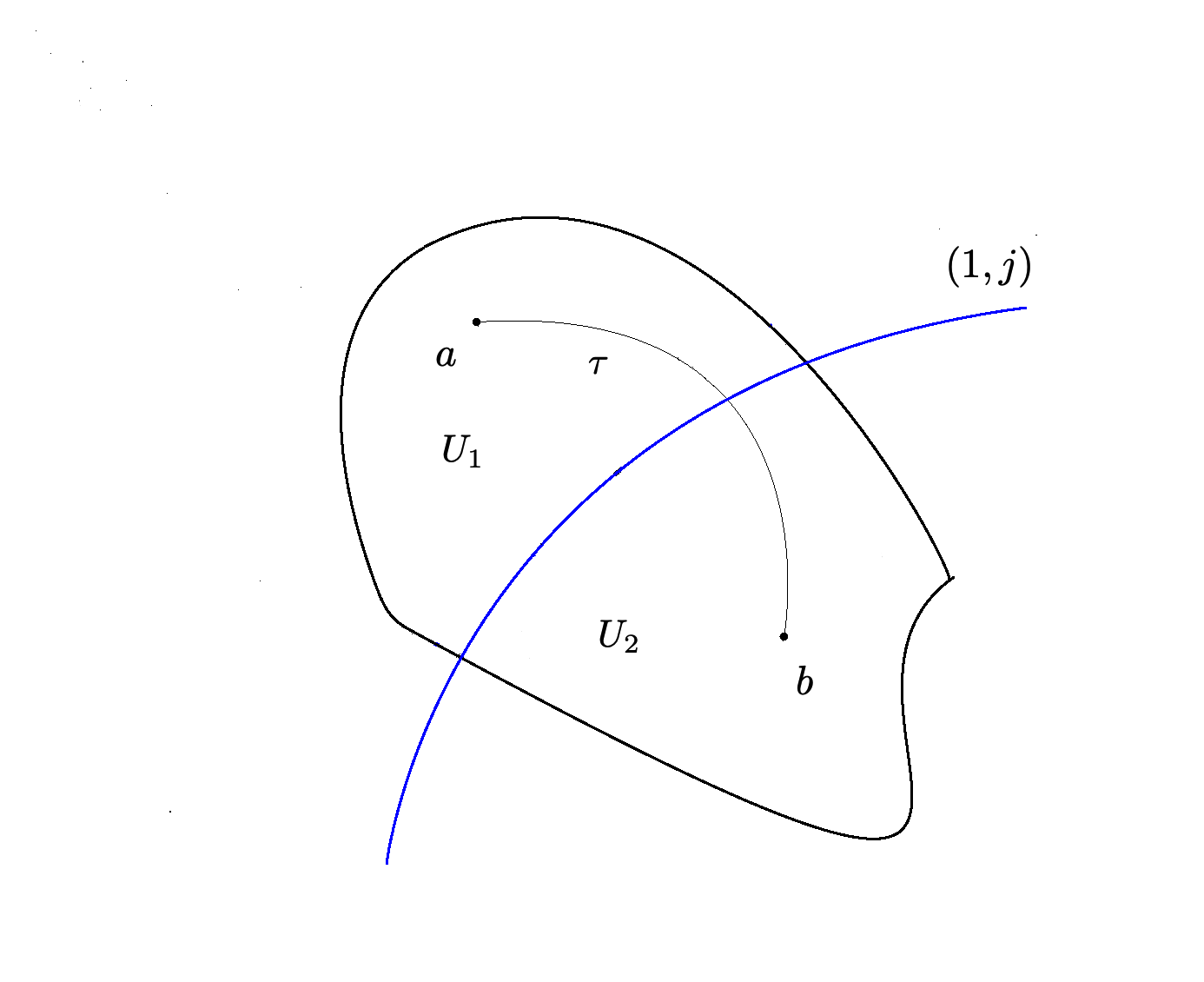}
         \caption{Analytic continuation of the Borel transform $\psi_{1,B}(z,y,z_k)$ along the path $\tau$ from $a\in  U_1$ to $b\in U_2$, crossing the Stokes curve of type $(1,j)$.  }\label{FK}
\end{figure}

While carrying out  the analytic continuation, we  deform  the integration path to $\tilde{\gamma}$, see Figure \ref{F3}. The condition that   $\ell$ does not connect   $z_1$ with another turning point or $z_1$ with itself  implies that we can  apply the  Cauchy theorem giving 
\begin{eqnarray}
\nonumber  \int_{\tilde{\gamma}}e^{-y\eta}\psi_{1,B}(b,y,z_k)dy&=& \int_{\gamma}e^{-y\eta}\psi_{1,B}(b,y,z_k)dy+ \int_{\gamma_0}e^{-y\eta}\psi_{1,B}(b,y,z_k)dy\\
\label{al1}       &=&\Psi_{1,2}(b,\eta,z_k)+ \int_{\gamma_0}e^{-y\eta}\psi_{1,B}(b,y,z_k)dy. 
\end{eqnarray}
Here the path $\gamma_0$ encircles the half-line
$$l_0=\left\{(z,y)\in\CC^2: \Im[y]=\Im\left[-\alpha_j y_0(z,z_k)   \right], \Re[y]>\Re\left[-\alpha_j y_0(z,z_k)   \right]\right\},$$
where $\alpha_j=e^{\frac{\pi (j-1)\imath}{M}}$. Hence, the Borel sum $\Psi_{1,1}$ changes by the factor $\int_{\gamma_0}\psi_{1,B}(b,y,z_k)dy$ when $z$ crosses the Stokes curve from $a$ to $b$. 
We recall that  
\begin{eqnarray}\label{al2}
\Delta_{y=-\alpha_j y_0(z,z_k)   }\psi_{1,B}(z,y,z_k)=l^+_0\psi_{1,B}(z,y,z_k)-l^-_0\psi_{1,B}(z,y,z_k)
\end{eqnarray}
 is the alien derivative of $\psi_{1,B}$, and $l^{\pm}_0\psi_{1,B}$ denotes the analytic continuation of $\psi_{1,B}$ from above $l^+_0$ and below $l^-_0$.

\medskip
By expanding    the second summand  of the last  expression in  \eqref{al1}   as a WKB-solution of \eqref{OPer}  we deduce  that 
\begin{eqnarray}\label{al3}
\Delta_{y=-\alpha_j y_0(z,z_k) }\psi_{1,B}(z,y,z_k)=c_j\psi_{j,B}(z,y,z_k).
\end{eqnarray}
By substituting  \eqref{al2} and \eqref{al3} in \eqref{al1} we get,
\begin{eqnarray}
\begin{aligned}\label{psi1j} 
\Psi_{1,1}(z,\eta,z_k)&\mapsto  \Psi_{1,2}(z,\eta,z_k)+c_j\Psi_{j,2}(z,\eta,z_k),\\
  \Psi_{j,1}(z,\eta,z_k)& \mapsto  \Psi_{j,2}(z,\eta,z_k),
\end{aligned}
\end{eqnarray}
when we cross  the Stokes curve.

\begin{figure}[!ht]
  \begin{subfigure}{0.45\textwidth}
    \includegraphics[width=1\textwidth]{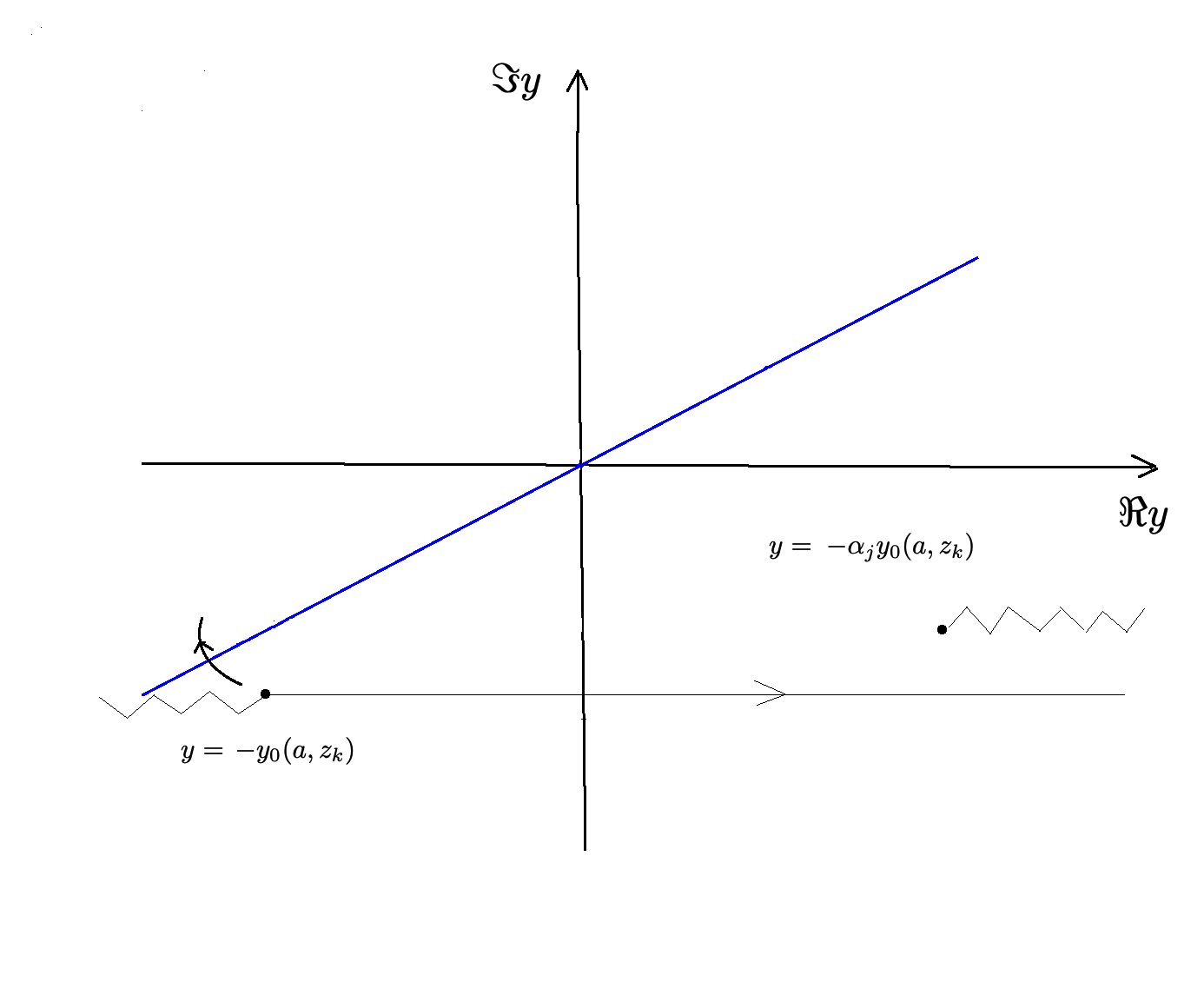}
    \caption{The integration path for the Borel sum of $\psi_1$ at $z=a$.}\label{F1}
  \end{subfigure}
     \begin{subfigure}{0.45\textwidth}
    \includegraphics[width=1\textwidth,right]{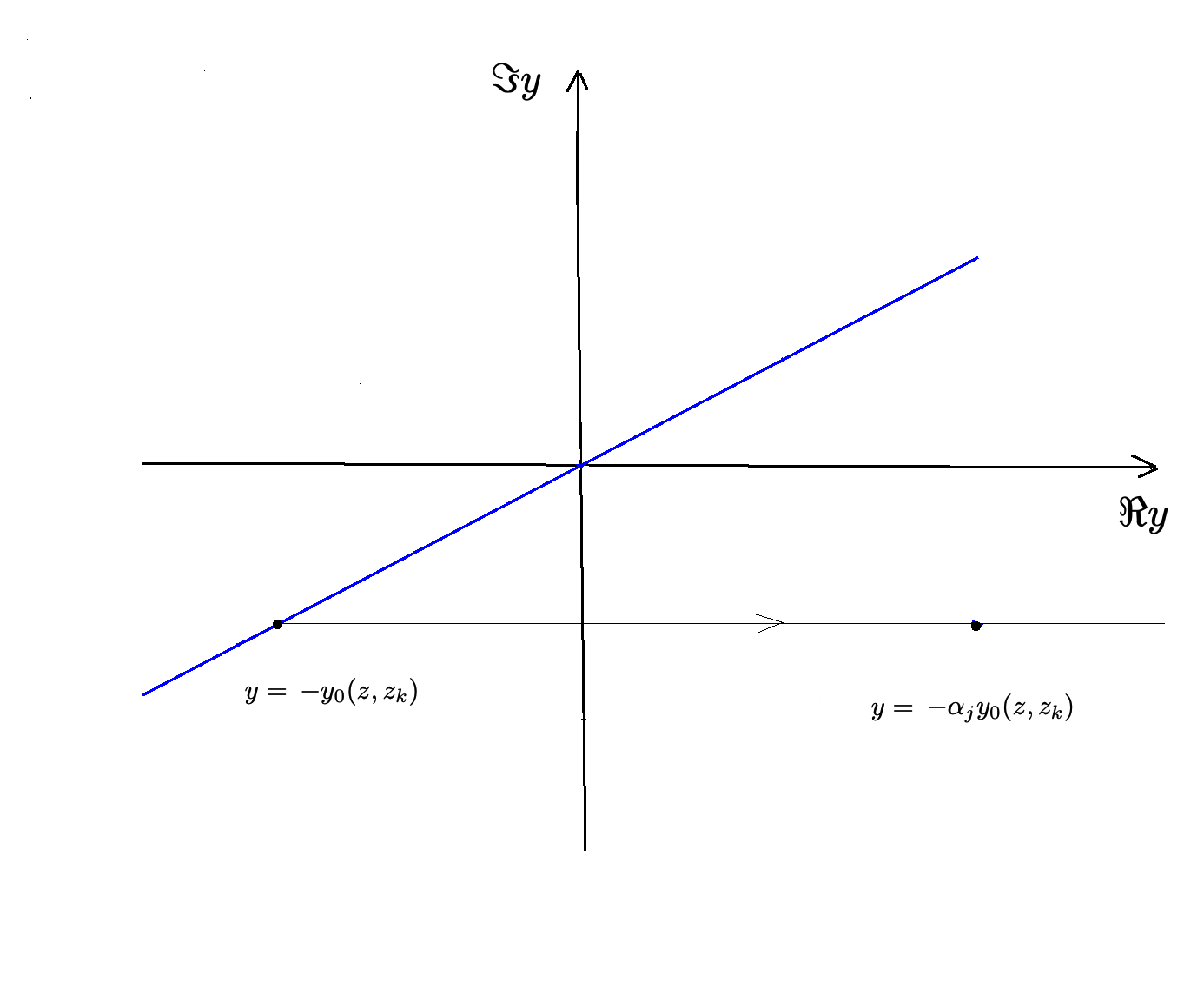}
    \caption{Coincidence of the integration paths when $z$ belongs to  the Stokes curve. }\label{F2}
  \end{subfigure} 
  \\
  \begin{subfigure}{0.45\textwidth}
    \includegraphics[width=1\textwidth,left]{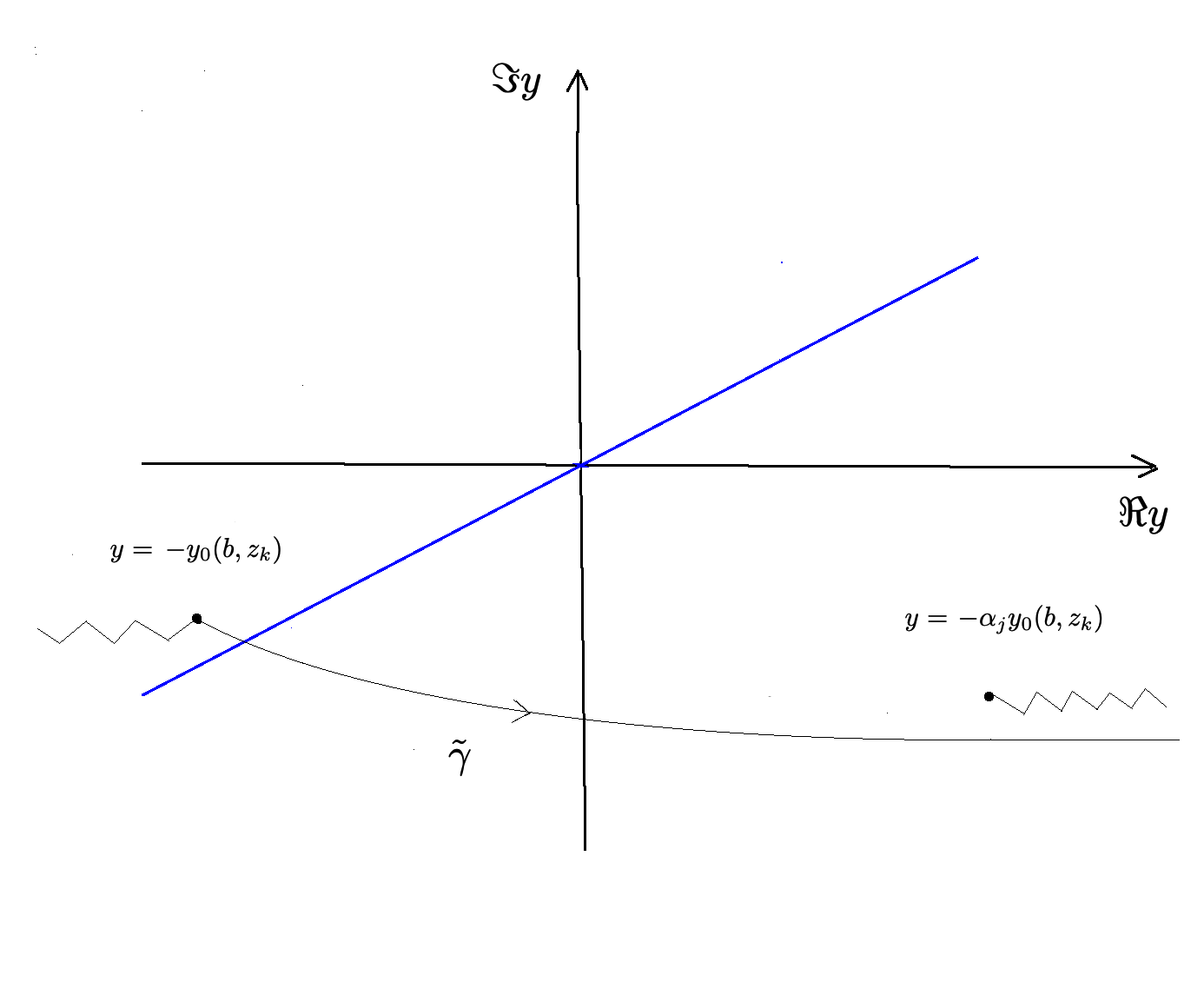}
    \caption{Deformation of the path to continue $\psi_{1,B}$ analytically.}\label{F3}
  \end{subfigure}  
     \begin{subfigure}{0.45\textwidth}
    \includegraphics[width=1\textwidth,right]{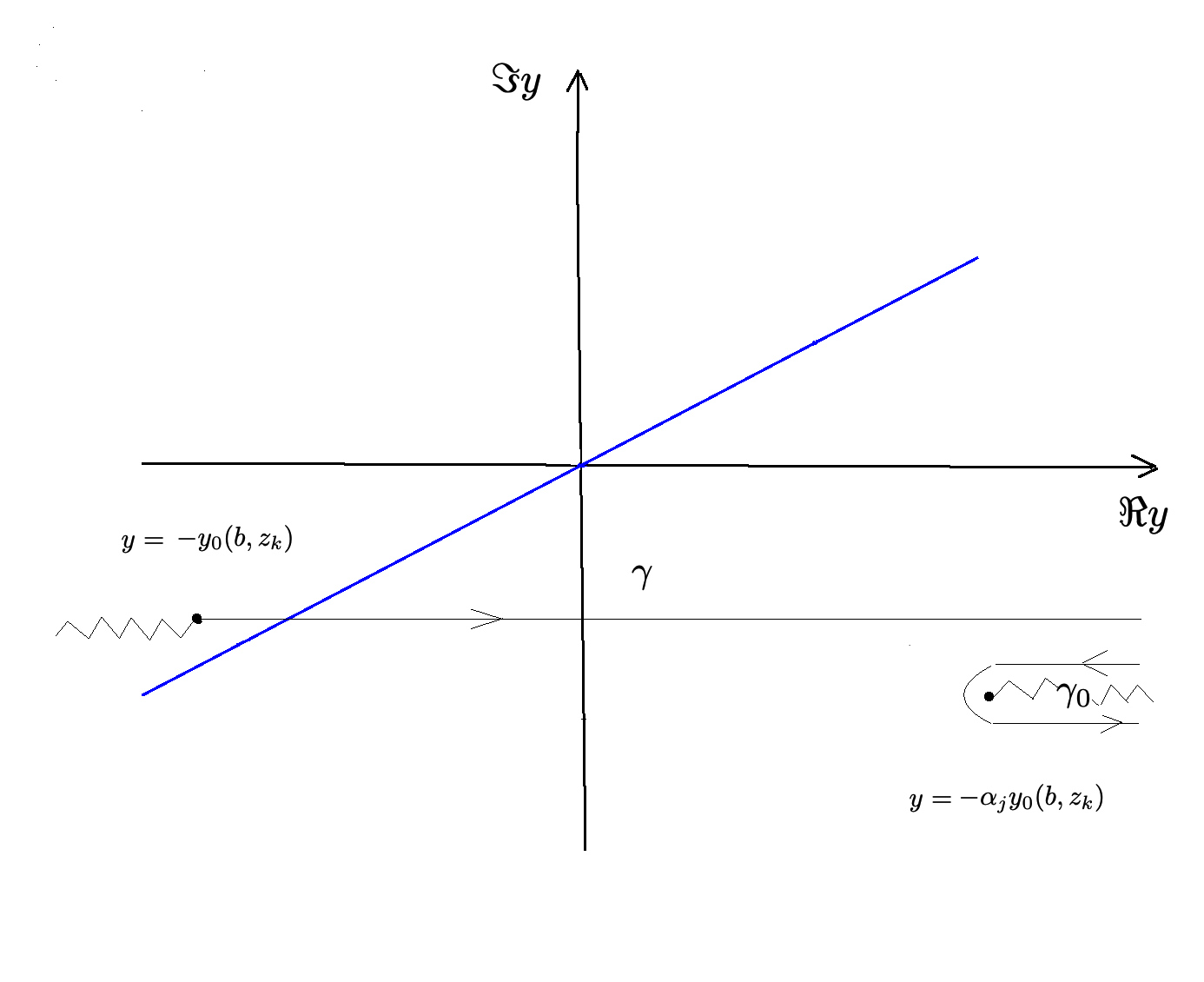}
    \caption{Decomposition of the path $\tilde{\gamma}$ to obtain the analytic continuation of $\psi_{1,B}$.}
  \end{subfigure} 
  \caption{Integration paths for the Borel sums (wiggly lines denote the branch cuts for the Borel transforms). } 
\end{figure}
%
%Now, if $K\subset \Omega$  is a  compact set  as in Figure \ref{FK} and $z\in K$, then by  \cite[Lem. 11]{BoMo22}, for large enough $n$,  the monic  eigenpolynomial of degree $n$ of \eqref{eq:ES}  can be expressed  as 
%\begin{eqnarray}\label{Qn1j}
%Q^{\MM}_n(z)=\Psi_1\left(z,  \sqrt[M]{\lambda_{n}},z_k\right)r\left(z,  \sqrt[M]{\lambda_{n}}\right),
%\end{eqnarray}
%  where    $\dsty \lambda_{n}$ is the eigenvalue associated to the eigenpolynomial    $Q^{\MM}_n$ of \eqref{eq:ES}. Here  we  take the   branch of the root  for which  the sequence $\frac{1}{ \sqrt[M]{\lambda_{n}}}$ converges to  $\dsty \frac{1}{n}$ when $n\rightarrow\infty$. Since $Q^{\MM}_n$ does not have zeros in $K$ for $n$ large enough, we have that     $r$ is an analytic function in $K$.  On the other hand, using \eqref{Qn1j} and \eqref{psi1j}, when we move from $a$ to $b$, we obtain 
%\begin{eqnarray*}
%\begin{aligned} 
%Q^{\MM}_n(z) &\mapsto Q^{\MM}_n(z)+ c_j\Psi_j\left(z, \sqrt[M]{\lambda_{n}},z_k\right)r\left(z, \sqrt[M]{\lambda_{n}}\right),\\
% \Psi_j(z,\eta,z_k)& \mapsto  \Psi_j(z,\eta,z_k).
% \end{aligned}
%\end{eqnarray*}
%Since $Q^{\MM}_n$ is analytic in $K$ we deduce that $c_j=0$. A similar argument applies when we move from $b$ to $a$. Thus we get  the same connection formula again. The first connection formula when $(1>j)$ is settled.

\medskip
Using the same reasoning, for  $(1<j)$ and when  $z$ crosses  from one region to the other along the curve $\tau$, we obtain the second connection formula whcih completes the proof. 
 \lqqd

\medskip
Finally, let us  settle Theorem \ref{Glob}. 

 \begin{proof} a)   By \cite[Th.3]{BerRull02},  each of the Jordan arcs  $\mathfrak{r}_i$ forming $\supp {\mu^\MM}=\mathfrak{r}_1\cup \mathfrak{r}_2\cup \mathfrak{r}_3$ is sent to  straight segments by the mapping $\Psi(z)=\int^zw_1(t)dt$. A direct calculation shows  that  the boundary of the $\mathfrak{B}$-region   is  a piecewise linear  curve  $\dsty  \bigcup_{k=1}^{6}I_k, I_k=[p_k,p_{k+1}],$ 
where 
$$
\dsty \begin{cases}
0\equiv p_1= \lim_{z\rightarrow z_1,\; z\in V^+} \Psi(z),\; p_7= \lim_{z\rightarrow z_1,\; z\in V^-}\Psi(z),\\
p_{2i-1}=\Psi(z_i), i=2,3\\ 
 p_{2i}=\lim_{z\rightarrow v,\;z\in V_i} \Psi(z),\; i=1,2,3.
\end{cases}
$$
The curves $V_i, i=1,2,3$ are shown in Figure \ref{fi00}.
  Notice that 
\begin{eqnarray}\label{impart}
      p_1=p_7+2\pi \imath.                 
\end{eqnarray}                         
 
  \begin{figure}[h!]
\centering
        \includegraphics[width=0.3\textwidth]{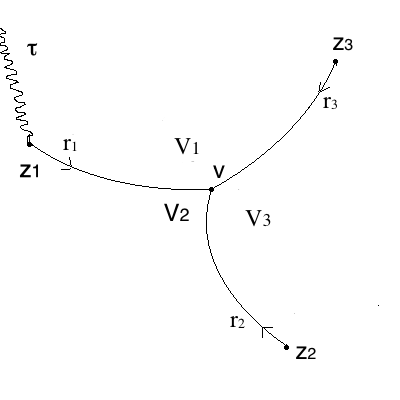}
         \caption{The region $\supp {\mu^\MM}\setminus \mathfrak{r}$. The wiggly line denotes the branch cut defined by $\mathfrak{r}$ }\label{fi00}
\end{figure}

By  Lemma \ref{angles},   the  angles at the point $v$ between the arcs $\mathfrak{r}_i, i=1,2,3$ of   $\supp {\mu^\MM}$ are  $\dsty \frac{2\pi}{3}$. Hence, the interior angles between the line segment $I_k, k=1,\ldots,6$ are given by 
\begin{eqnarray}\label{angleijT}
\dsty \left(\frac{2\pi}{3},\frac{4\pi}{3},\frac{2\pi}{3},\frac{4\pi}{3},\frac{2\pi}{3}\right).
\end{eqnarray}
Therefore the lines segments  $I_k; k=1,3,5$ are parallel to each other as well as the  line segments  $I_k; k, k=2,4,6$.
 
\medskip
By Definition \ref{Tp}, we have that the Stokes curves of type $(j,j^{\prime})$ that emanate from $z_k$ are given by 
\begin{eqnarray}\label{ST1}
\{z\in \Omega: \Im\left[\int_{z_k}^z(w_{j}(\zeta)-w_{j^{\prime}}(\zeta))d\zeta\right]=0\}.
\end{eqnarray}
Hence,  from items a) and b) of Lemma \ref{Complete}  if $ \kappa=\{z:\Im[z]=0\}$   we have that 
\begin{eqnarray}\label{ST2}
\{z\in \Omega: \Im\left[\int_{z_k}^z(w_{j}(\zeta)-w_{j^{\prime}}(\zeta))d\zeta\right]=0\}=\rho\left( \mathcal{F}^{-1}_{(j,j^{\prime})}( \kappa\cap \mathfrak{B}_{(j,j^{\prime})}) \right).
\end{eqnarray}

 In  Figure \ref{fi1}  using   relations \eqref{ST1}, \eqref{ST2}, and item c) of Lemma \ref{homeo}, we show  the Stokes curves emanating from $z_1$ in each of the $\mathfrak{B}_{(j,j^{\prime})}$-regions. While Figure~\ref{fi2} shows the projection of the Stokes curves emanating from $z_1$ onto the $\mathfrak{B}$-region

  \begin{figure}[!ht]
\centering
        \includegraphics[width=0.7\textwidth]{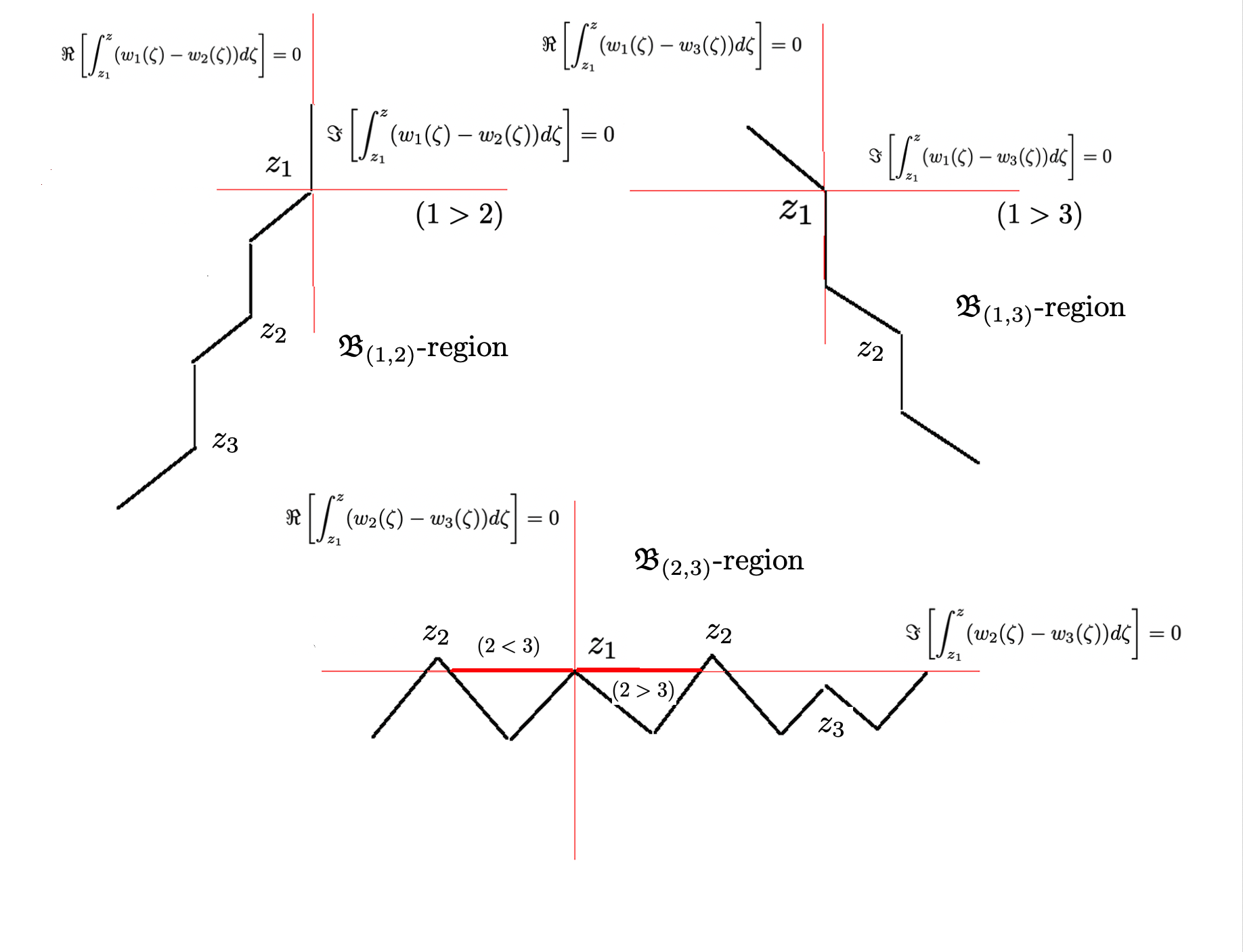}
         \caption{The   Stokes curves emanating from $z_1$ in the $\mathfrak{B}_{(j,j^{\prime})}$-regions. (Notice that  $\mathfrak{B}_{(1,2)}=e^{-\frac{\imath\pi}{6}}\mathfrak{B}, \mathfrak{B}_{(1,3)}=e^{\frac{\imath\pi}{6}}\mathfrak{B}, \mathfrak{B}_{(2,3)}=e^{\frac{\imath\pi}{2}}\mathfrak{B}  $). }\label{fi1}
\end{figure}

\begin{figure}[!ht]
     \centering
    \includegraphics[width=0.5\textwidth]{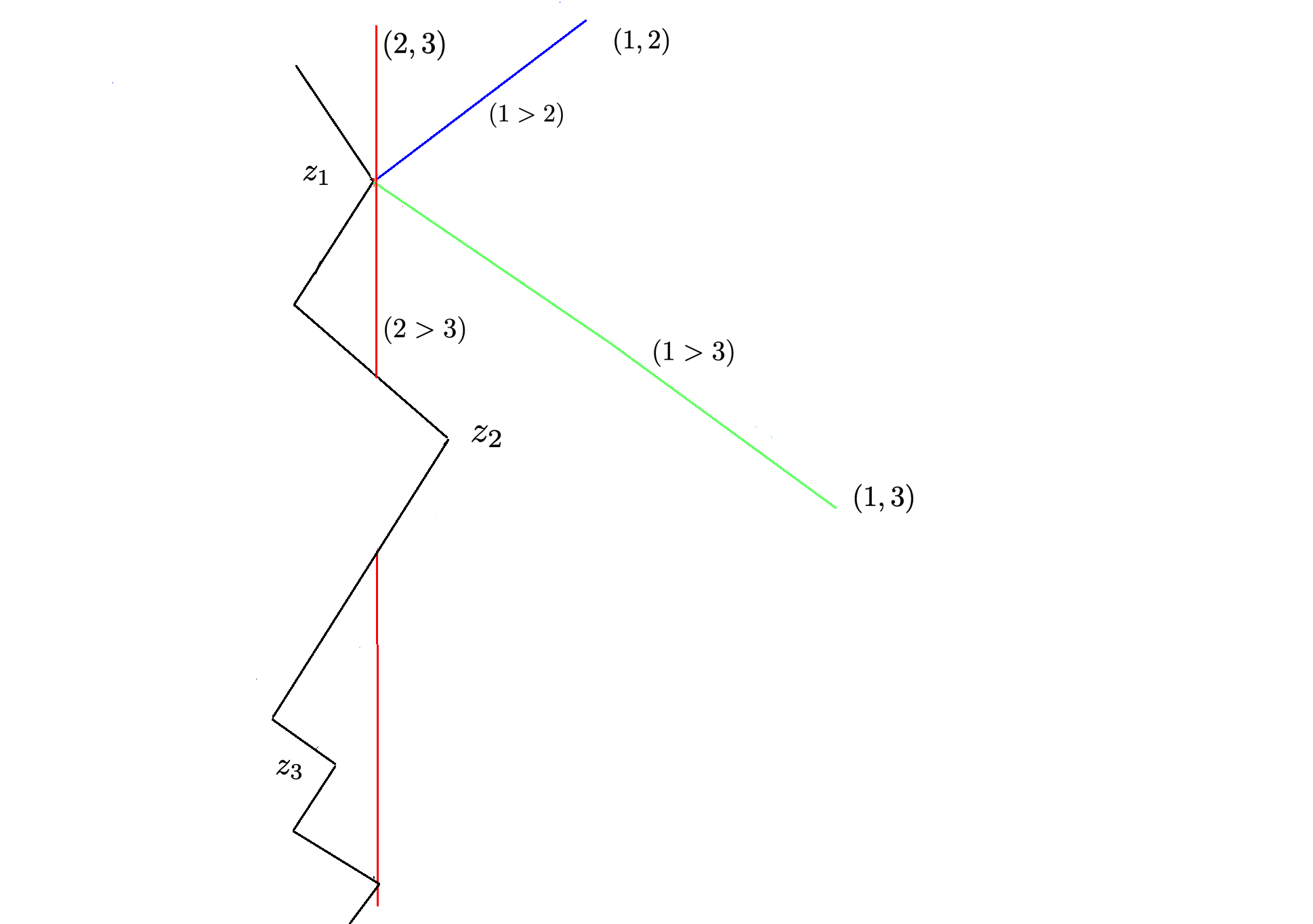}
    \caption{Projection   of the Stokes curves emanating from $z_1$ in the $\mathfrak{B}$-region.  (It might happens that for some special polynomials $\rho_3$, that the red curve passes through the points $z_2$ and $z_3$).} \label{fi2}
\end{figure}

\begin{figure}[!ht]
  \begin{subfigure}{0.5\textwidth}
     \centering
    \includegraphics[width=0.8\textwidth]{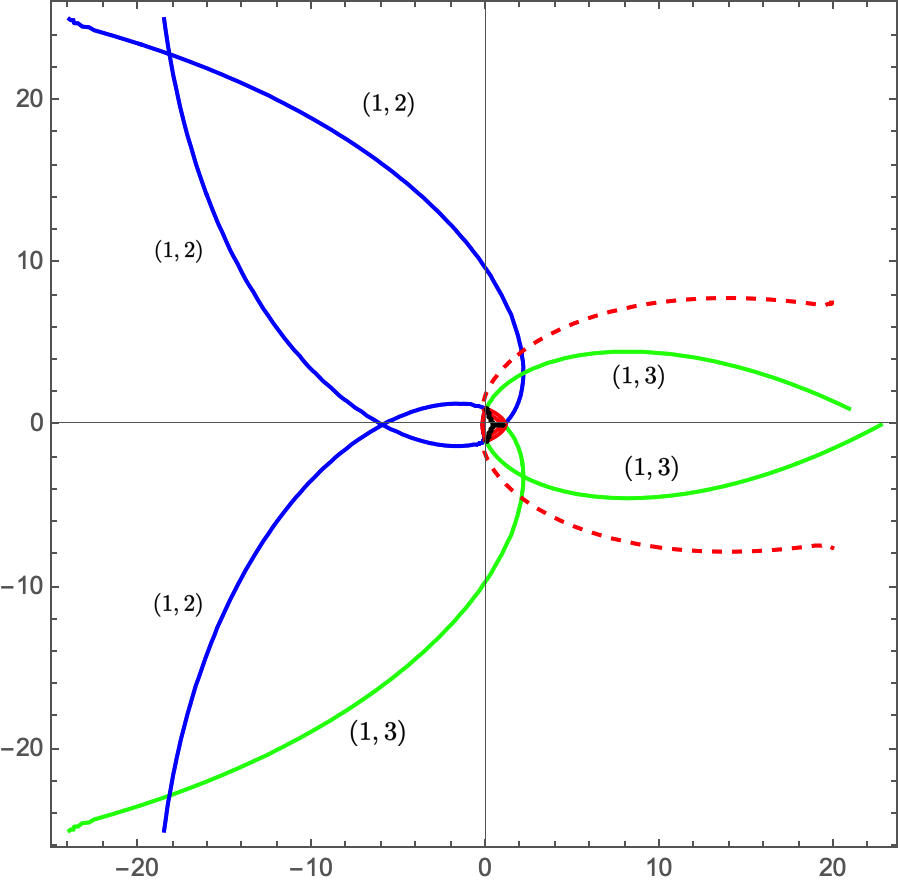}
    \caption{  Stokes curves  of type $(1,2)$ and $(1,3)$.}
  \end{subfigure}
   \begin{subfigure}{0.35\textwidth}
    \centering
    \includegraphics[width=1.1\linewidth]{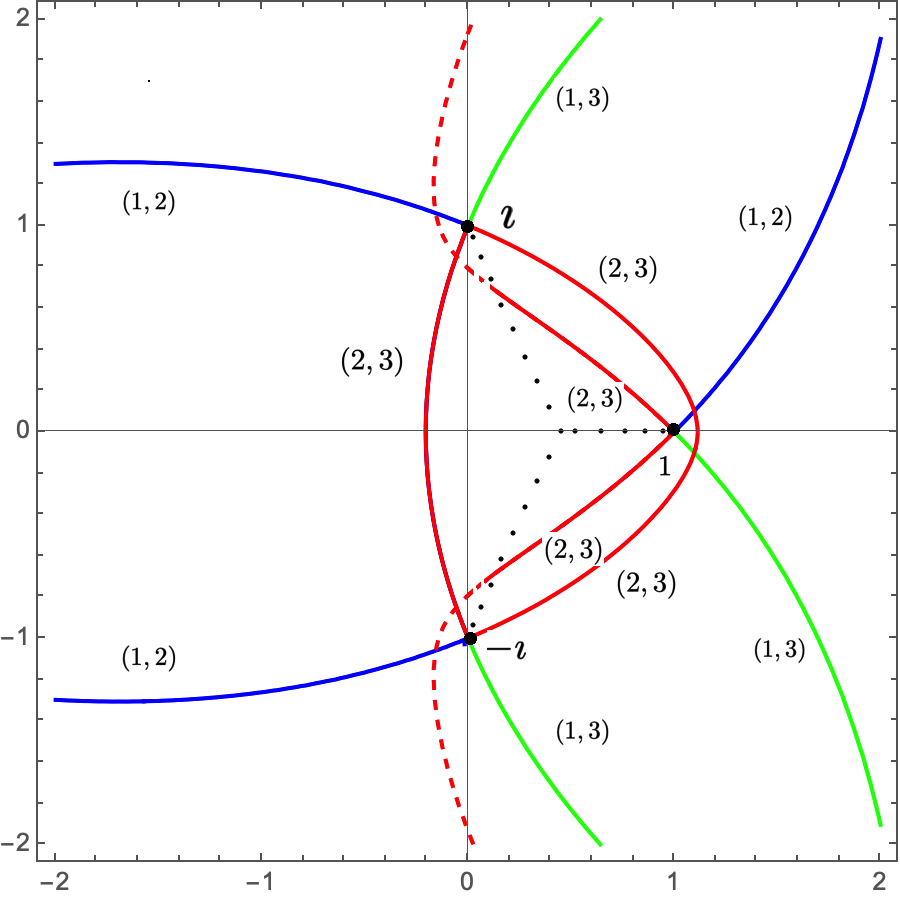}
    \caption{ Stokes curves of type $(2,3)$ in a region closer to $\supp\mu$ (punctured lines in black).  }
  \end{subfigure}  
\caption{Stokes curves   for the operator $\dsty \LL[y]=(z^2+1)(z-1)y^{\prime\prime\prime}- \eta^{3}y$ in the $\Omega$-region. The dashed curves in red indicate the continuation  of the curves of type $(2,3)$ emanating from $z_1=1$ when crossing  the branch cut defined by $\supp\mu$. Also, the Stokes curves  $(2,3)$ emanating from $\pm\imath$ are closed curves and are  piled up.}  \label{f_example}
\end{figure}

A similar argument applies to the remaining roots $z_2$ and $z_3$ and we obtain  the Stokes complex (i.e. configuration of all the Stokes curves) shown in Figure  \ref{fi3}.

\begin{figure}[h!]
\centering
        \includegraphics[width=0.5\textwidth]{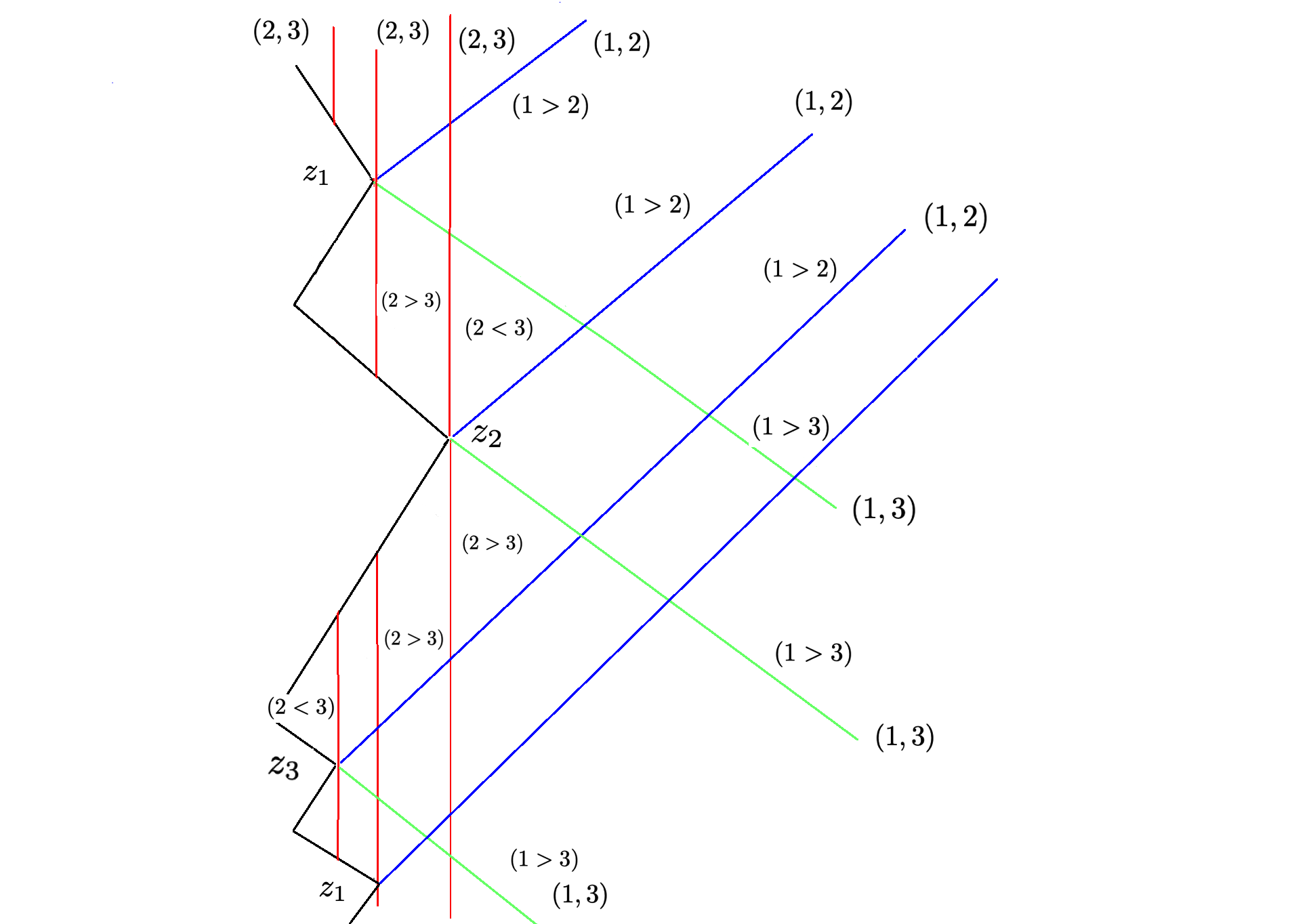}
         \caption{ A typical situation of Stokes curves for $M=3$ in the $\mathfrak{B}$-region.   }\label{fi3}
\end{figure}

The set of all  the Stokes curves emanating from $z_1,z_2$ and $z_3$ in the $\Omega$-region forms the configuration presented in Figure \ref{f_example} for the operator $\dsty \LL[y]=(z^2+1)(z-1)y^{\prime\prime\prime}- \eta^{3}y$.

b) To identify the new Stokes curves we analyze  ordered crossing, which  happen when  $(1,2)$ and $(2,3)$  intersect. Our analysis is performed in the $\mathfrak{B}$-region, the intersection points are shown in Figure \ref{fi4}.

\begin{figure}[H]
  \begin{subfigure}{0.5\textwidth}
    \includegraphics[width=1\textwidth]{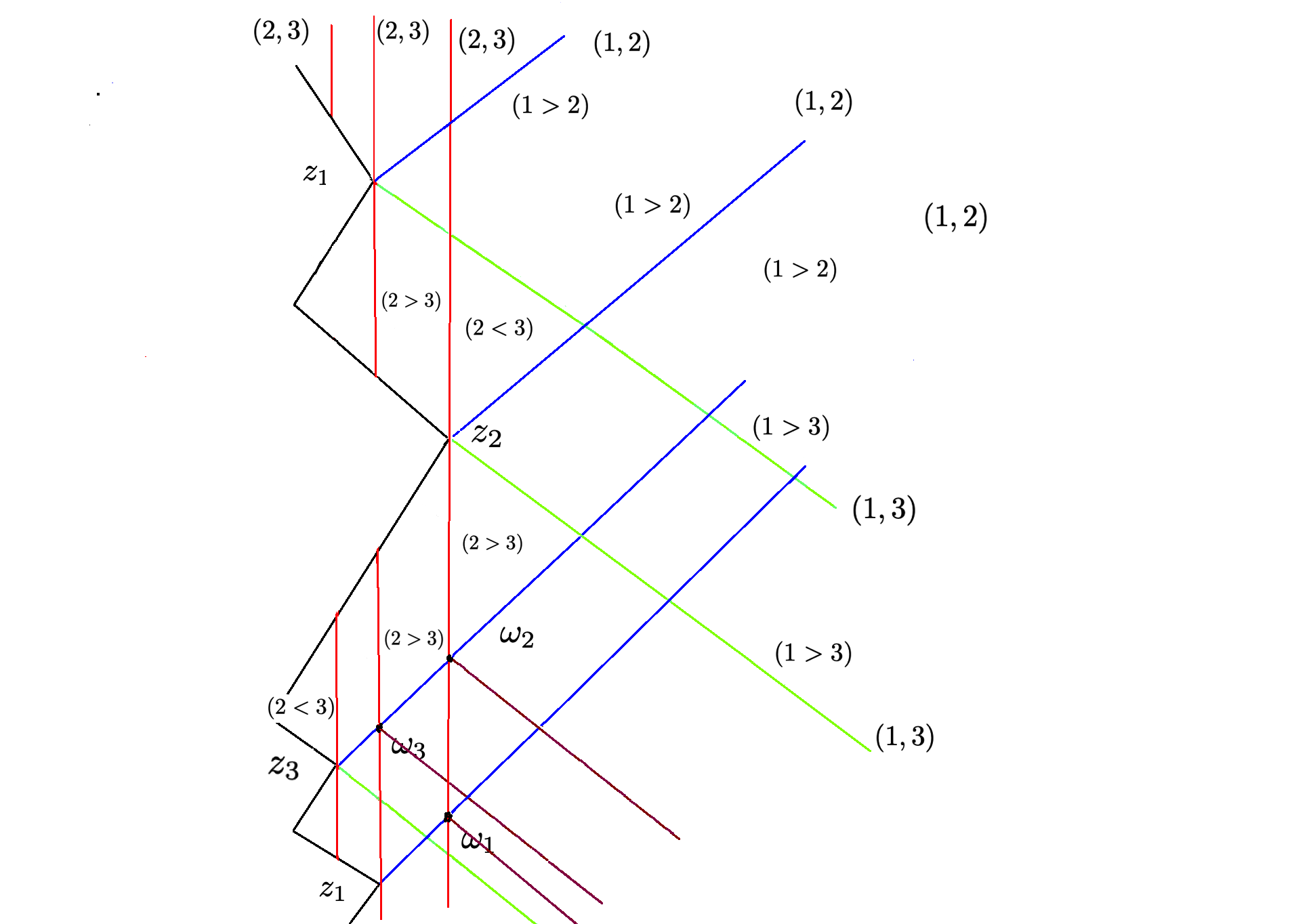}
    \caption{Ordered crossings and the new Stokes curves (in brown). }\label{fi4}
  \end{subfigure}
  \hfill
   \begin{subfigure}{0.5\textwidth}
    \includegraphics[width=1\textwidth,left]{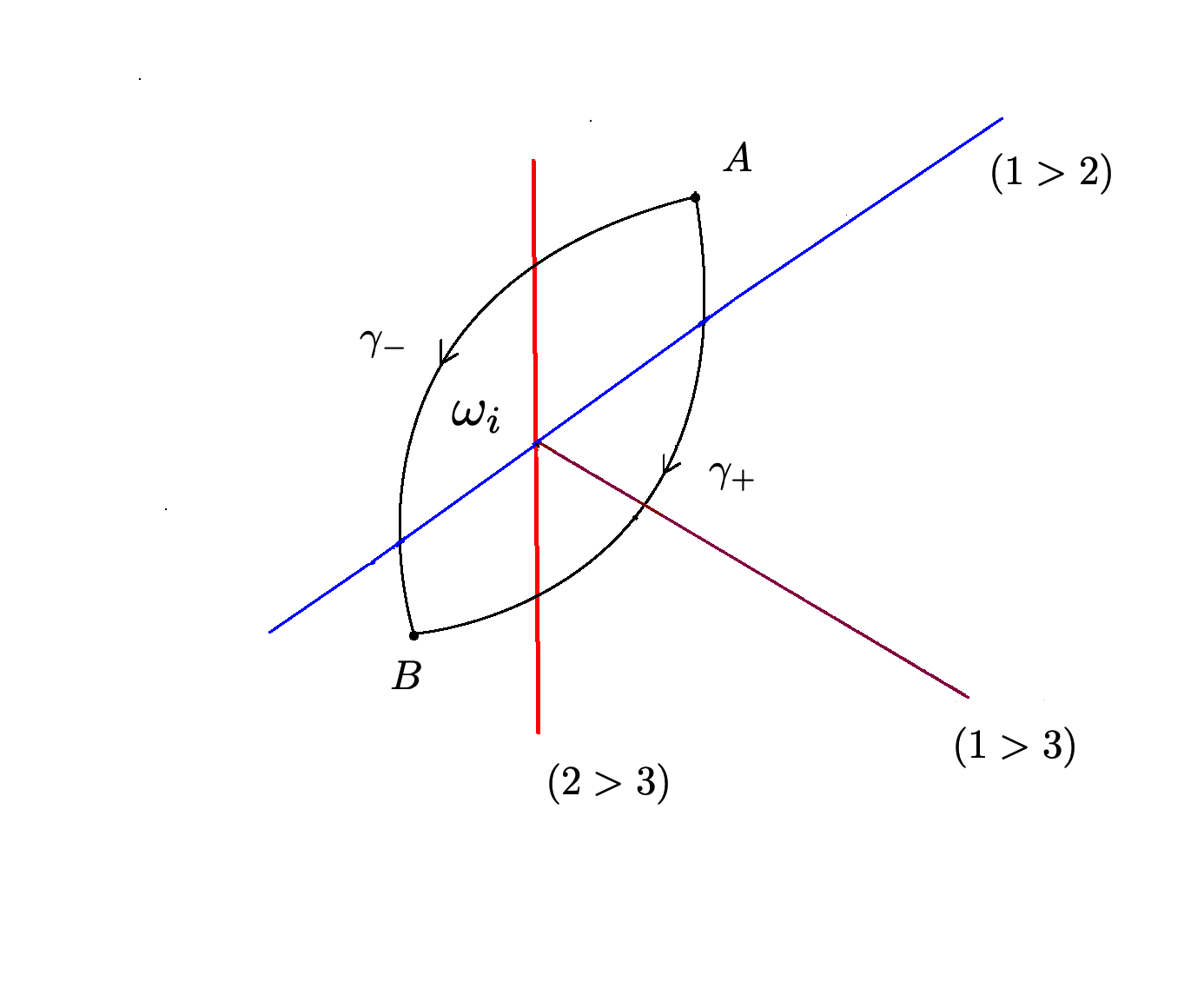}
    \caption{New Stokes curve $(1>3)$ in a neighborhood of $\omega_i$. }\label{fi5}
  \end{subfigure}  
\caption{The original Stokes curves and the new Stokes curves. }  
\end{figure}
% 
%
% 
%Now,  if $\omega_i$ is a crossing point of these curves, it follows from Case 1) of  Theorem \ref{VTPE} {\color{red} that $(1,2)$ is inert.} Hence, by moving  from the point $A$ to the point $B$ along the paths $\gamma_{\pm}$, as in Figure \ref{fi5},  we conclude that $\psi_1\mapsto \psi_1$, i.e., there is no Stokes phenomenon around the intersections. 
 \lqqd
\end{proof}

Corollary \ref{CoTh3} follows from  the relation $\mathfrak{B}_{(1,2)}=e^{-\frac{\imath\pi}{6}}\mathfrak{B}$, see   Figure  \ref{fi1}.

\section{Local structure  of a  third order exactly solvable differential operator near  a turning point }\label{CoordChang}

Factorization of some higher order   differential operators of the WKB-type near  a simple ordinary turning point into lower order differential operators of the same class has been considered in   \cite{AokiKawTak94,AKT95,AKKT04}.  For some  higher order  linear ordinary differential operators, the same   has been done in  \cite{KKT10,ka11} near a pole-type turning point. In this section we provide an example of an exactly solvable  $\LL$ which can not be factorized near of a turning point into operators of lower order with analytic coefficients. Therefore  in general, the study of the WKB-solutions of  exactly solvable operators can not be reduced to the study of lower order differential operators.  More precisely, consider

\begin{eqnarray}\label{L0}
\LL= (z-a)p(z)\frac{d^3}{dz^3}+ p(z)\frac{d^2}{dz^2}+q(z)\frac{d}{dz}-\eta^3,
\end{eqnarray}
where $p$ and $q$ are polynomials of degree $2$ and $1$ respectively such that $p(a)\neq 0$ and $\eta>0$  is a large real number. 

\medskip
  Recall the following notion, see \cite[Def. 3.3]{AKKT04}. 
\begin{definition}
Let $P_0(z,\zeta )$ be  the principal symbol of a differential operator $P$ of the WKB-type on an open set $U\subset \CC_z$ and let $z_*\in U$ be an ordinary  turning point with characteristic value $\zeta_*$. In other words,  the system of equations
$$P_0(z,\zeta)=\partial_{\zeta} P_0(z,\zeta)=0$$
has a solution  $(z_*,\zeta_*)\in U\times \CC_{\zeta}$  and $P_0(z_*,\zeta)$    does not vanish identically as a function of $\zeta$.  The smallest positive integer $m$ such that $\partial^m_{\zeta} P_0(z,\zeta)$ does not vanish is called the \textcolor{blue}{rank of the turning point $z_*$} with the characteristic value $\zeta_*$.
\end{definition}

 Using the  transformation $\dsty z-a=(x-a)^2$ we get   
$$(z-a)^{-1/2}\LL=\LL^*,$$
  where
\begin{eqnarray}\label{L*}
 \LL^*=p((x-a)^2+a)\frac{d^3}{dx^3}-\frac{1}{2} q((x-a)^2+a)\frac{d}{dx}-(x-a)\eta^3.
\end{eqnarray}
Notice that $\eta^{-3}\LL^*$ is of the WKB-type and that   $x=a$ is an ordinary  turning point of rank $3$ with characteristic value $\zeta_*=0$.

Let $V_a$ be  a  neighborhood of $z=a$ and take the cut in $V_a$ by using the arc of $\supp {\mu^{\MM}}$ whose endpoint is $a$. Pick a branch of $\sqrt{z-a}$ in $V_a$ and set  $U^{+}_a=\mathcal{T}^{+}(V_a), U^{-}_a=\mathcal{T}^{-}(V_a)$, and $U_a=U^{-}_a\bigcup U^{+}_a$, where $\mathcal{T}^{\pm}(z)=a\pm \sqrt{z-a}$. 

% Define
%\begin{eqnarray*}
%\dsty F(z)&=&\Im \left[\int_{\alpha(a,z)}(w_j(t)-w_{k}(t))dt\right], \quad  z\in V_a,\\
%\dsty G(x)&=&\Im \left[\int_{\beta(a,x)}((t-a)\mathfrak{w}_j((t-a)^2+a)-(t-a)\mathfrak{w}_{k}((t-a)^2+a))dt\right], \quad x\in U_a,
% \end{eqnarray*}
% where 
%$$
%\mathfrak{w}_i(x)= \begin{cases}%\label{bob2}  
%   w_i(x), & x\in U^{+}_a \\
%    \widetilde{w}_i(x) & x\in  U^{-}_a,
%\end{cases} 
%$$
%here    the paths $\alpha$ and $\beta$  in the  integrals is taken as a straight line connecting the points $a$  and  $z$ or $x$ and  $\widetilde{w}$ is the analytic continuation of $w_i$ from  $U^+_a$ to $U^-_a$ {\color{red} (which can be given explicitly using the monodromy of $\LL^*$ at $a$)}
 
%  By definition, the Stokes curves in $V_a$ of type $(j,k)$  for $\LL$ are  given by 
%\begin{eqnarray}\label{F}
%\{z\in V_a:F(z)=0\}.
%\end{eqnarray}
%  
%    On the other hand,  the Stokes curves of  type $(j,k)$  for $\LL^*$  in $U_a$ are given by  
%\begin{eqnarray}\label{G}
%\{x\in U_a: G(x)=0\} 
%\end{eqnarray}
%
%  
%  
%%   Notice that, 
%%\begin{eqnarray}\label{Strel}
%%F(z)=G(a\pm \sqrt{z-a}).
%%\end{eqnarray}
%
% 
%
%%   branch cut for $\sqrt{z}, z\in V_a$
%  
%A description of the  Stokes curves  of $\LL^*$ in a neighborhood of the turning point  can be studied using the following 

\begin{proposition}\label{SecRed}
For every  sufficiently small neighborhood of $z=a$, there are no differential operators $Q$ and $R$ of the WKB-type such that 
$$\eta^{-3}\LL=QR.$$
Here $Q=\sum_{j\geq 0}\eta^{-j}Q_j(x,\eta^{-1}\frac{d}{dx})$ and $R=\sum_{j\geq 0}\eta^{-j}R_j(x,\eta^{-1}\frac{d}{dx})$ are differential operators of order $1$ and order $2$ in $\frac{d}{dx}$ respectively, such that 
\begin{eqnarray*}
Q_0(a,0)&\neq& 0,\\
R_0(x,\zeta)&=&(\zeta-\zeta_j(x))(\zeta-\zeta_k(x)),
\end{eqnarray*}
 where $Q_0(x,\zeta) $ (resp.  $R_0(x,\zeta))$ denotes the principal symbol of the operator $Q$, (resp. $R$), i.e. $Q_0(x,\frac{\varsigma}{\eta})$ (resp. $R_0(x,\frac{\varsigma}{\eta})$  ) with $\frac{\zeta}{\eta}$ denoted by $\zeta$.
\end{proposition}

\proof
Let  $\LL^*$ be defined as in \eqref{L*}. A  straightforward calculation shows  that   $x=a$ is an ordinary  turning point of rank $3$  for $\eta^{-3}\LL^*$. By   \cite[Th. 5.2]{AKKT04}, we have that for $x\in U_a$, 
$$\eta^{-3}\LL^*=QR,$$
where $Q=\sum_{j\geq 0}\eta^{-j}Q_j(x,\eta^{-1}\frac{d}{dx})$ and $R=\sum_{j\geq 0}\eta^{-j}R_j(x,\eta^{-1}\frac{d}{dx})$ are the unique  WKB-type differential operators in $\frac{d}{dx}$ of order $0$ and order $3$  respectively. Finally, by considering  $(z-a)^{-1/2}\LL=\LL^*$, we obtain the required result.   \lqqd

\section{Some differential equations with convergent WKB-solutions (in the usual sense).}\label{SpOp}

The goal of this section is the proof of Theorems \ref{Th1} and \Ref{EC} which show the existence of  differential equations admitting  convergent WKB-solutions in usual sense.  

Let $P_n^{(\lambda)}(z)$ denote the monic Gegenbauer polynomial of degree $n$, see \cite[\S 4.7]{Szg75}. It is known that $P_n^{(\lambda)}(z)$ is a solution of the differential equation 
\begin{eqnarray}\label{0DE}
(z^2-1)y^{\prime\prime}+(2\lambda+1)zy^{\prime}-n(n+2\lambda)y=0,
\end{eqnarray}
see \cite[(4.7.5)]{Szg75}.   In the special cases $\lambda=0$ and $\lambda=1$ one obtains  the Tchebycheff polynomials of the first kind and second kind respectively.

 The variable change  $u=(z^2-1)^{\frac{\lambda}{2}+\frac{1}{4}}y$ transforms \eqref{0DE} into 
\begin{eqnarray}\label{DE}
\frac{d^2 u}{dz^2}-(n+\lambda)^2\left\{\frac{1}{z^2-1}- \frac{1}{(n+\lambda)^2}\frac{(\frac{1}{2}+\lambda-\lambda^2+\frac{z^2}{4})}{(z^2-1)^2}\right\}u=0,
\end{eqnarray}
cf. \cite[(4.7.10)]{Szg75}.   

\medskip
Denote by ${\bf B}_{k,l}$ the exponential  Bell polynomials \cite[(3d) p.134]{comtet74}, by $(n)_k=n(n-1)\cdots (n-k+1)$    the falling factorial, and by   $\varphi(z)=z+\sqrt{z^2-1}$  the Joukowsky function.  

\medskip
For  $k\in \NN\cup\{0\}$, set $$Y_1(z;k,\eta)= \eta^{-\frac{1}{2}}(z^2-1)^{\frac{k}{2}+\frac{1}{4}}(\varphi^{\eta}(z))^{(k)}, \quad Y_2(z;k,\eta)=\eta^{-\frac{1}{2}}(z^2-1)^{\frac{k}{2}+\frac{1}{4}}(\varphi^{-\eta}(z))^{(k)}, z\notin [-1,1],$$ 
where  $\eta$ is a positive parameter.

\begin{theorem}\label{Th1}
In the above notation, for $\eta$ large, one has the expansion
\begin{eqnarray*}
Y_1(z;k,\eta)&=&
 \eta^{-\frac{1}{2}}e^{ \eta\dsty \int_{1}^z \frac{1}{\sqrt{\zeta^2-1}} d\zeta }  \begin{cases}
        (z^2-1)^{\frac{1}{4}},   & k=0    \\
   \left( (z^2-1)^{\frac{k}{2}+\frac{1}{4}}\varphi^{-k}(z) \varphi^{\prime}(z)+ \right. \\
 \dsty  \left.\sum_{j=1}^{\infty} \left(\sum_{l=1}^{k-1} (z^2-1)^{\frac{k}{2}+\frac{1}{4}}a_{j,l,k}\varphi^{-l}(z){\bf B}_{k,l}(\varphi^{\prime}(z),\ldots,\varphi^{(k-l+1)}(z))\right)\frac{1}{\eta^j}\right), & k>0,  
\end{cases}\\
Y_2(z;k,\eta)&=& \eta^{-\frac{1}{2}}e^{  -\eta\dsty \int_{1}^z \frac{1}{\sqrt{\zeta^2-1}} d\zeta }  \begin{cases}
          (z^2-1)^{\frac{1}{4}},   & k=0    \\
 \left( (z^2-1)^{\frac{k}{2}+\frac{1}{4}}\varphi^{-k}(z) \varphi^{\prime}(z)+ \right. \\
 \dsty  \left. \sum_{j=1}^{\infty} \left(\sum_{l=1}^{k-1}(z^2-1)^{\frac{k}{2}+\frac{1}{4}} (-1)^ja_{j,l,k}\varphi^{-l}(z){\bf B}_{k,l}(\varphi^{\prime}(z),\ldots,\varphi^{(k-l+1)}(z))\right)\frac{1}{\eta^j}\right), & k>0,
\end{cases}\
\end{eqnarray*}
where   $\dsty \sum_{l=1}^0:=0$     and  $a_{j,l,k}$ are defined using the Laurent expansion 
\begin{eqnarray}\label{Lau}
\frac{(z)_l}{(z)_k} = \sum_{j=1}^{\infty} \frac{a_{j,l,k}}{z^j}
\end{eqnarray}
for $z$ varying in a neighborhood of $\infty$ and $1\leq l\leq k-1, k>1$. 

\medskip
 Moreover  $Y_1(z;k,\eta)$ and $Y_2(z;k,\eta)$ are WKB solutions   of the differential equation 
\begin{eqnarray}\label{kDE}
\frac{d^2 u}{dz^2}-\eta^2\left\{\frac{1}{z^2-1}- \frac{1}{\eta^2}\frac{(\frac{1}{2}+k-k^2+\frac{z^2}{4})}{(z^2-1)^2}\right\}u=0, \quad z\notin [-1,1] 
\end{eqnarray}
convergent (in the usual sense)  in a neighborhood of  $z=1$.

\end{theorem}

 \proof Observe that by  \cite[(4.21.7)]{Szg75} and \eqref{0DE} for $n\in \NN$ and $0\leq k<n$, one has that    
$$\dsty Y_1(z;k)=\frac{d^kP_n^{(0)}(z)}{dz^k} $$ is a solution to
\begin{eqnarray}\label{2DE}
(z^2-1)y^{\prime\prime}+(2k+1)zy^{\prime}-(n-k)(n+k)y=0.
\end{eqnarray} 
%moreover, $\dsty y^{(k)}_1(z;0)=\frac{1}{2}{_2F_1}\left(-n+\frac{1}{2}, n+\frac{1}{2}; \frac{3}{2}; \frac{1-z}{2}\right)  $.

Using \cite[(4.23.1)]{Szg75}, one has that a second solution  of \eqref{0DE}  linearly independent of $P_n^{(\lambda)}$ is given by 
\begin{eqnarray}\label{1}
\psi(z;\lambda)= (1-z)^{-\lambda+\frac{1}{2}}{_2F_1}\left(-n-\lambda+\frac{1}{2}, n+\lambda+\frac{1}{2}; \frac{3}{2}-\lambda; \frac{1-z}{2}\right). 
\end{eqnarray}
Therefore 
\begin{eqnarray}\label{2}
 \psi(z;k)=(1-z)^{-k+\frac{1}{2}}{_2F_1}\left(-n+\frac{1}{2}, n+\frac{1}{2}; \frac{3}{2}-k; \frac{1-z}{2}\right)
 \end{eqnarray}
is  a second solution  of \eqref{2DE}  linearly independent of $Y_1$ when $0\leq k<n, k\in \ZZ$.

\medskip
By \cite[(22) p.102]{Erd53},
\begin{eqnarray}\label{3}
\left(\frac{3}{2}-k\right)_k  z^{-k+\frac{1}{2}}{_2F_1}\left(-n+\frac{1}{2}, n+\frac{1}{2};\frac{3}{2}-k; z\right) =
\frac{d^k}{dz^k}\left(z^{\frac{1}{2}}{_2F_1}\left(-n+\frac{1}{2}, n+\frac{1}{2}; \frac{3}{2}; z\right) \right).
\end{eqnarray}
By considering $\lambda=0$ in \eqref{1} and substituting  $\dsty \frac{1-z}{2} \mapsto z $  in   the right-hand side of \eqref{3} we see that the  left-hand side of \eqref{3} coincides   with \eqref{2} (up to a constant factor). In other words,  
\begin{eqnarray}\label{Y2}
\dsty Y_2(z,k)=c\frac{d^k\psi(z;0)}{dz^k}, c\in \CC\setminus\{0\},
\end{eqnarray}
  is a second  solution  to \eqref{2DE}  linearly independent of $Y_1$.

\medskip

Hence, by considering \eqref{0DE} and \eqref{DE} we obtain  that  $(z^2-1)^{\frac{k}{2}+\frac{1}{4}}Y_1(z;k)   $ and  $(z^2-1)^{\frac{k}{2}+\frac{1}{4}} Y_2(z;k)  $   are solutions of  
\begin{eqnarray}\label{3DE}
\frac{d^2 u}{dz^2}-n^2\left\{\frac{1}{z^2-1}- \frac{1}{n^2}\frac{(\frac{1}{2}+k-k^2+\frac{z^2}{4})}{(z^2-1)^2}\right\}u=0. 
\end{eqnarray}

It can be easily verified that for $\lambda=0$ and $n\geq 0$, $\varphi^n(z)$ and $\varphi^{-n}(z)$ are  solutions of \eqref{0DE}  which implies that there exist constants $c_{ij}(n)$   such that 
\begin{eqnarray}\label{St1}
\begin{cases}
\varphi^{n}(z)&= c_{11}(n)Y_1(z;0)+  c_{12}(n)\psi(z;0)  \\
\varphi^{-n}(z)&=  c_{21}(n)Y_1(z;0)+ c_{22}(n)\psi(z;0).
\end{cases}
\end{eqnarray} 
 Since   $Y_1(z,k)$ and  $Y_2(z,k)$  are  linearly independent solutions   of \eqref{2DE},  relation \eqref{Y2} and \eqref{St1} imply  that $(\varphi^{n}(z))^{(k)}$ and $(\varphi^{-n}(z))^{(k)}$ are also  solutions of \eqref{2DE}. Therefore,  $(z^2-1)^{\frac{k}{2}+\frac{1}{4}}(\varphi^{n}(z))^{(k)}$ and $(z^2-1)^{\frac{k}{2}+\frac{1}{4}} (\varphi^{-n}(z))^{(k)}, 0\leq k< n$ are  solutions of \eqref{3DE}.

To prove that $(z^2-1)^{\frac{k}{2}+\frac{1}{4}} (\varphi^{n}(z))^{(k)}$ and $(z^2-1)^{\frac{k}{2}+\frac{1}{4}} (\varphi^{-n}(z))^{(k)}$ define  convergent  WKB solutions of \eqref{3DE}, 
observe that  by \cite[Th.A p.137]{comtet74}, for $k>0$, one has 
$$
\begin{cases}
(\varphi^{n}(z))^{(k)}=\varphi^{n}(z)\sum_{l=1}^{k} (n)_l\varphi^{-l}(z){\bf B}_{k,l}(\varphi^{\prime}(z),\ldots,\varphi^{(k-l+1)}(z)),\\
(\varphi^{-n}(z))^{(k)}=\varphi^{-n}(z)\sum_{l=1}^{k} (-n)_l\varphi^{-l}(z){\bf B}_{k,l}(\varphi^{\prime}(z),\ldots,\varphi^{(k-l+1)}(z)).
\end{cases}
$$
Extracting the first term in  the summation  for $k>1$, we obtain that the above expressions can be written  as
\begin{eqnarray}\label{SF}
\begin{cases}
\frac{1}{(n)_k}(\varphi^{n}(z))^{(k)}    =e^{n\ln \varphi(z)}\left( \varphi^{-k}(z) \varphi^{\prime}(z)+\sum_{l=1}^{k-1} \frac{(n)_l}{(n)_k}\varphi^{-l}(z){\bf B}_{k,l}(\varphi^{\prime}(z),\ldots,\varphi^{(k-l+1)}(z))\right), \\
\frac{1}{(-n)_k}(\varphi^{-n}(z))^{(k)}= e^{-n\ln \varphi(z)}\left( \varphi^{-k}(z) \varphi^{\prime}(z)+ \sum_{l=1}^{k-1} \frac{(-n)_l}{(-n)_k}\varphi^{-l}(z){\bf B}_{k,l}(\varphi^{\prime}(z),\ldots,\varphi^{(k-l+1)}(z))\right).   
\end{cases}
\end{eqnarray}

Substituting \eqref{Lau} into \eqref{SF} and  considering   the factor $(z^2-1)^{\frac{k}{2}+\frac{1}{4}}$,  we get   two WKB solutions of \eqref{kDE} for $\eta=n$ given by 
$$
\begin{cases}
Y_1(z;k,\eta)= \eta^{-\frac{1}{2}}\allowbreak \exp\left(\dsty \int_{1}^z S_{+}(\zeta,\eta)d\zeta \right)\\Y_2(z;k,\eta)=\eta^{-\frac{1}{2}}\exp\left(\dsty \int_{1}^z S_{-}(\zeta,\eta)d\zeta \right)
\end{cases}.
$$   On the other hand,  substitution of $Y_1$ or $Y_2$ in \eqref{kDE} defines a power series  which vanishes for $\eta=n$ and all $n\in \NN$. Thus $Y_1(z;k,\eta)$ and $Y_2(z;k,\eta)$ are  actually  WKB solutions  of \eqref{kDE} at $z=1$ for large $\eta\in \RR_+$. \lqqd

\begin{coro}

Let $Y_1$ and $Y_2$ be as in Theorem \ref{Th1} and consider the Stokes line   $\kappa$ emanating from $z=1$, 
$$\dsty \Im \left[\int_{1}^z \frac{1}{\sqrt{\zeta^2-1}} d\zeta\right]=0.$$ In other words,   $\kappa=\{z\in \CC: \Re[z]\geq 1, \Im[z]=0\}$.  If  two Stokes regions $U_i, i=1,2$   
share  a subset of $\kappa$ with a chosen orientation of their boundaries, then when  we cross $\kappa$ clockwise or counterclockwise we get 
$$
  \begin{cases}
  Y_{1,1}=Y_{1,2} &  \\
  Y_{2,1}=Y_{2,2},  
\end{cases} 
$$
where $Y_{i,j}=\left.Y_i(z;k,\eta)\right|_{z\in U_j}, i,j\in \{1,2\}.$

\end{coro}

\proof
Follows immediately from the relation $$
\begin{cases}
Y_1(z;k,\eta)= \eta^{-\frac{1}{2}}(z^2-1)^{\frac{k}{2}+\frac{1}{4}}(\varphi^{\eta}(z))^{(k)} \\ Y_2(z;k,\eta)=\eta^{-\frac{1}{2}}(z^2-1)^{\frac{k}{2}+\frac{1}{4}}(\varphi^{-\eta}(z))^{(k)}, z\notin [-1,1].
\end{cases}$$ \lqqd

\begin{theorem}\label{EC}
Let $\epsilon$ be a small complex parameter varying in a punctured neighborhood of the origin, and $K\subset \Omega$ be a compact set. Then, for the  Euler-Cauchy differential equation
\begin{equation} 
\label{ECEq1} z^Mv^{(M)}(z,\eta)+\sum_{k=1}^{M-1} a_kz^k v^{(k)}(z,\eta)- {\eta^{M}{v(z,\eta)} }=0,\\
%\label{ECEq2}  z^Mv^{(M)}(z,\epsilon)+\sum_{k=1}^{M-1} \frac{b_k}{\epsilon^{M-k}}z^k v^{(k)}(z,\epsilon)- \frac{v(z,\epsilon)}{\epsilon^{M} }&=&0, {\color{red}(\mbox{in progress, probably another paper?})}
\end{equation}
there exist $M$ linearly independent  WKB-solutions
$$\psi_j=\exp\left[{\dsty \sum_{k=0}^{\infty}h_{j,k}\epsilon^{k-1}\ln z}\right], \varepsilon\in V^*,$$
convergent  for all  $z\in K$, in a reduced neighborhood of $\varepsilon =0$, where  $h_{j,0}=\sqrt[M]{1}$.
\end{theorem}

In particular,  by the preceding theorem,  there is no Stokes phenomenon. 
  Before we prove Theorem \ref{EC} we need a preliminary lemma.

\begin{lemma}\label{existeCh}
Let $b_1, \ldots, b_{M-1}$ be  complex numbers. Then the algebraic equation
\begin{equation}\label{puisexeq}
w^M+\sum_{k=1}^{M-1}b_{k}w^{k} - {b_{0}}{\epsilon^{-M}}=0
\end{equation}
has $M$  solutions $\dsty w_j(\epsilon)=\sum_{k=0}^{\infty}h_{j,k}\epsilon^{k-1}, j=1,\ldots,M$  holomorphic in a neighborhood $V^{*}$ of
$0$, where $h_{j,0}=\sqrt[M]{b_0}$  and  $h_{j,k} \in \CC$.

\end{lemma}

\begin{proof} Multiplying (\ref{puisexeq}) by $\epsilon$ and making the variable change $y=w\epsilon, $ we obtain the equation 
\begin{equation} \label{ch2}
 F(\epsilon,y)=y^M +\sum_{k=1}^{M-1}b_{k}(\epsilon^{M-k}y^{k})-b_{0}=0. 
\end{equation}

Notice   that
$$
\begin{cases}
 F(0,h_{j,0}) = 0 \\
  \frac{\partial F}{\partial \epsilon}(0,h_{j,0})  \neq  0,
\end{cases}
$$
 where $h_{j,0}=\sqrt[M]{b_0}$ (i.e. all the  roots of $b_0$). Hence,  from the implicit function  theorem, there exists a  neighborhood $V$ of $0$ and $M$  unique  analytic functions $\dsty y_j(\epsilon)=\sum_{k=0}^{\infty}\epsilon^{k}h_{j,k}$ such that $y_j(0)=h_{j,0}$ and  $F(\epsilon,y_j(\epsilon))=0, \forall \epsilon\in V$, see  \cite[Th 3.11, Vol II]{Mark65}. Taking into account that  $y=w\epsilon,\;\epsilon \neq 0$  we complete the  proof.  \lqqd
\end{proof}

Next we settle Theorem \ref{EC}. 

 \begin{proof} Looking for a solution  of equation \eqref{ECEq1} in the form  $v=z^{w}$ we obtain for $w$  the  indicial  equation 
\begin{eqnarray} 
 \label{ind1} w^M+A_{M-1}w^{M-1}+\ldots +A_1w-\frac{1}{\epsilon^M}=0,
\end{eqnarray}
where $A_k\in \CC$.

%\\
% \label{ind2} w^M+B_{M-1}(\epsilon)w^{M-1}+\ldots+B_1(\epsilon)w - \frac{1}{\epsilon^M}=0,

By Lemma \ref{existeCh} we have that \eqref{ind1}  has  $M$ solutions $\dsty w_j(\epsilon)=\sum_{k=0}^{\infty}h_{j,k}\epsilon^{k-1}$ defined in a reduced neighborhood $V^{*}$ of
$0$. Hence, the eqution \eqref{ECEq1} has $M$ solutions of the form 
$$\dsty v=z^{\sum_{k=0}^{\infty}h_{j,k}\epsilon^{k-1}},\; \varepsilon\in V^*,$$
 where $h_{j,0}=\sqrt[M]{1}$. By writing the latter  expression as 
$$\dsty v=\exp\left[{\dsty \sum_{k=0}^{\infty}h_{j,k}\epsilon^{k-1}\ln z}\right],\; \varepsilon\in V^*,$$
we obtain  $M$ linearly independent convergent  WKB-solutions for the equation~\eqref{ECEq1}. 
\lqqd
\end{proof}

\section{Open problems} 

\noindent
{\bf 1.} The following questions are very crucial for our considerations. 

\begin{problem}
Give a formal definition of a virtual turning point for exactly solvable operators. 
\end{problem}

\begin{problem}
Extend Theorem \ref{Main} in the case $\rho_M=(z-a)^M$. Is the definition  of Stokes curves given by \eqref{STdeg} appropriate  for this case?
\end{problem}

\noindent
{\bf 2.} The next question is related to b) of Theorem \ref{Main} and c) of Theorem \ref{Glob}. 

\begin{problem}
Does the non-existence of self-intersections on the bicharacteristic curve  imply that all new  Stokes curves are inert? Consequently, are there  no "new turning points" from which "new Stokes curves" emanate? 
\end{problem}

\noindent
{\bf 3.} The following guess and question  is related to Theorem~\ref{VTPE}. 
%
%\begin{problem} Case 1 of Theorem~\ref{VTPE} has been proved considering the  Hadamard  finite part regularization \eqref{HadamarFP}  in the improper integral of the function $S$ of Definition \ref{WKBS}.  In \cite{Ko00}, the author obtained a connection formula for a family of second order differential equations  of pole type  using a different  regularization for improper integrals, see \cite[Ch.2 p.19]{AKT95} for details,  and assuming the Borel summability of the WKB-solutions. In particular, the following linear differential equation   
%\begin{eqnarray}\label{OPc}
%\frac{d^2 y}{dz^2}-\frac{\eta^2}{z(z+1)}y=0,
%\end{eqnarray}
%belongs to the class considered  both in  \cite{Ko00} and our  present manuscript.  It can be easily verified that the connection formula of  \cite{Ko00} for the Borel resummed  WKB-solutions differ at $z=0$ from that of Case 1 of Theorem~\ref{VTPE} , which gives as a consequence that the   WKB-solutions of \eqref{OPc} are  not Borel summable,  for the regularization considered in \cite{Ko00}.  Based on this fact, we propose the study of the  Borel  summability of the WKB-solutions of the differential equations in \cite{Ko00} with the mentioned regularization of  that manuscript.
% 
%\end{problem}

\begin{conjecture}\label{conj:new} For a generic  equation \eqref{OPer}, Case 2  never happens. 
\end{conjecture}

\noindent
{\bf 4.} The last question is the most important in this area of research.
 
\begin{problem}{\rm Describe the Stokes complex, i.e. the union of Stokes curves for an arbitrary (non-degenerate) exactly solvable operator \eqref{OPer}. }
% Extend Theorem~\ref{th:BM} from the case of  polynomial solutions (which exist for a sequence of  values of $\eta$) to the case of all solutions of \eqref{Eq1} at least for some subclass of non-degenerate exactly solvable operators. ??? } 
\end{problem}

\noindent
{\bf 5.} Inspired by Theorem~\ref{EC}  and based on  some calculations, we have the following guess.

\begin{conjecture}\label{AWKB}
For an arbitrary  holomorphic function $\rho_3$, the WKB-solutions of the differential equation 
\begin{equation} \label{OPf}
\rho_3(z)v^{\prime\prime\prime}+\rho^{\prime}_{3}(z)v^{\prime\prime}+\left(3\rho_3^{\prime\prime}(z)-\frac{\rho^{\prime 2}_3(z)}{\rho_3(z)}\right)v^{\prime}- \frac{v}{\epsilon^{M} }=0,
\end{equation}
 reduce to 
$$\psi_j=\exp\left[{\dsty \frac{1}{\epsilon}\int^z \frac{1}{\sqrt[3]{\rho_3(\xi)}}d\xi  }\right].$$
\end{conjecture}

\section{Appendix. Sibuya's theorem on Gevrey summability of formal power series  depending on a parameter.  } \label{append:III}  

Suppose that we have:

\begin{itemize}
\item[1)] a formal power series $\dsty \phi(z,\epsilon)=\sum_{k=0}^{\infty}\epsilon^k\phi_k(z)$ depending on $\epsilon$, where the coefficients $\phi_k$ are complex--valued functions holomorphic in $z$ in a simply connected domain $D_0$ of the $z$--plane;

\item[2)] a nontrivial polynomial
$$F(x_0,x_1,\ldots,x_l,z,\epsilon)=\sum_{m_0+\ldots +m_l=0}^Rx_0^{m_0}\cdots x_l^{m_l}F_{m_0,\ldots,m_l}(z,\epsilon),$$
 with coefficients $\dsty F_{m_0,\ldots,m_l}(z,\epsilon)=\sum_{m=0}^{\infty}\epsilon^mF_{m_0,\ldots,m_l;m}(z)$ which are formal power series in $\epsilon$ with  coefficients $F_{m_0,\ldots,m_l;m}(z)$ being complex--valued and holomorphic in $z$ in the domain $D_0$;

\item[3)]  a formal power series $\dsty F\left(\phi,\frac{d\phi}{dz},\frac{\phi^2\phi}{dz^2}, \ldots, \frac{d^{N-1}\phi}{dz^{N-1}},z,\epsilon\right)$ depending on  $\epsilon$ which is identically equal to zero  in the domain $D_0$;

\item[4)]  nonnegative numbers $s,K_1$, and $K_2$ such that 
$$|F_{m_0,\ldots,m_l;m}(z)|\leq K_1(m!)^sK_2^m,$$
for  $z\in D_0$ and $(m_0,m_1,\ldots,m_l;m)\in \NN^{l+2}$ such that  $\dsty 0\leq \sum_{k=0}^lm_k\leq R$. In other words, $F_{m_0,\ldots,m_l}(z,\epsilon)$ are of the  Gevrey  order $s$ in $\epsilon$ uniformly in $z\in D_0$.

\end{itemize}

A theorem due to Sibuya  (see \cite[Th.1.2.1]{Si03})  provides  the Gevrey summability of  a formal series in $\epsilon$ satisfying the condition $3)$ uniformly in $z$ on every compact subset of $D_0$, under some assumptions described below as Cases A, B and C.

To state  this theorem, let us assume that 
\begin{eqnarray}\label{Dxl}
\frac{\partial F_j}{\partial x_{N-1}}\left(\phi,\frac{d\phi}{dz},\frac{\phi^2\phi}{dz^2}, \ldots, \frac{d^{N-1}\phi}{dz^{N-1}},z,\epsilon\right)\neq 0,
\end{eqnarray}
for some $z\in D_0$ as a formal power series in $\epsilon$.

Define the linear differential  operator 
$$\mathcal{T}[x]=\sum_{h=0}^{N-1}\frac{\partial^h F_j}{\partial x_h}\left(\phi,\frac{d\phi}{dz},\frac{d^2\phi}{dz^2}, \ldots, \frac{d^{N-1}\phi}{dz^{N-1}}, z,\epsilon\right)D^hx,$$ 
where $\dsty D=\frac{d}{dz}$. 

For the above operator $\mathcal{T}$, construct a convex polygon   as follows. 

Set
 $$\frac{\partial F_j}{\partial x_h}\left(\phi,\frac{d\phi}{dz},\frac{d^2\phi}{dz^2}, \ldots, \frac{d^{N-1}\phi}{dz^{N-1}},z,\epsilon\right)= \sum_{m=0}^{\infty}\epsilon^ma_{h,m}(z), \quad h=0,\ldots,N-1.$$
 
 For $h=0,\ldots,N-1$, fix    nonnegative integers coefficients  $m_h, (h=0,1,\ldots,N-1)$ defined by the conditions 
\begin{eqnarray}
\begin{aligned}\label{ahm}
a_{h,m}(z)&=0, \quad  h=0,\ldots , m_h-1, \forall z\in U,\\
a_{h,m_h}(z)&\neq 0 \quad \mbox{for some} \quad  z\in U.
\end{aligned}
\end{eqnarray}

If all $a_{h,m}(z)=0$ identically in $z$ for all $m\geq 0$, we set $m_k=+\infty$.

\medskip
Let us consider $N$ points $(h,m_h), h=0,\ldots, N-1$ in the $(X,Y)$-plane. 
The convex hull  of the set  $\mathcal{P}=\bigcup_h\mathcal{P}_h$, where $\mathcal{P}_h=\{(X,Y): 0\leq X\leq h, Y\geq m_h\}$ is called the \textcolor{blue}{\emph{polygon of the operator $\mathcal{T}$}}. In other words, there exist nonnegative integers
\begin{eqnarray}\label{hk}
0\leq h_1<h_2<\ldots <h_k=N-1,
\end{eqnarray}
such that 
\begin{itemize}
\item[i)] $m_{h_1}\geq 0$,
\item[ii)] if we set 
\begin{eqnarray}\label{rhonu}
\rho_{\nu}=\frac{m_{h_{\nu}}-m_{h_{\nu-1}}}{h_{\nu}-h_{\nu-1}}, \nu=2, \ldots,k
\end{eqnarray}
we have 
$$0<\rho_2<\ldots<\rho_k,$$
\item[iii)]   $m_{h}\geq m_{h_1}$, for $0\leq h\leq h_1$, and 
$$\frac{m_{h_{\nu}}-m_{h}}{h_{\nu}-h}, \quad \mbox{for} \quad h_{\nu-1}<h\leq h_{\nu}, \quad \mbox{and}\quad  \nu=2,\ldots, k.$$

\end{itemize}

Now, under the assumption  \eqref{Dxl}, the  Cases A, B, and C  are described as follows.

\begin{itemize}
\item[Case A:] The integer $h_1=0$, i.e.
\begin{eqnarray}\label{CRelCas}
\epsilon^{-m_{h_1}}\left|\mathcal{T}[y]\right|_{\epsilon=0}=Q_0(z)y,
\end{eqnarray}
where $Q_0(z)$ is holomorphic on $D_0$ and not identically equal to zero. 

\end{itemize}

\noindent For Cases B and C, we have $h_1>0$, i.e.
$$\epsilon^{-m_{h_1}}\left|\mathcal{T}[y]\right|_{\epsilon=0}=\sum_{j=0}^{h_1}Q_j(z)D^{j}y,$$
where $Q_0(z), \ldots, Q_{h_1}(z)$ are  holomorphic in $D_0$ and $Q_{h_1}(z)$ is not identically vanishing in $ D_0$.

\begin{itemize}
\item[Case B:]   $Q_{h_1}(z)$ has no zeros in  $D_0$. 

\item[Case C:]  $Q_{h_1}(z)$  vanishes at some point $z\in D_0$. 

\end{itemize}

In this article we are only interested in Case A in which, under the assumption \eqref{Dxl},  Sibuya's theorem claims the following. 
 
 \begin{theorem}[Theorem 1.2.1 of  \cite{Si03}]\label{Sib}
 In Case A, the formal series $\phi(z,\epsilon)$ has  Gevrey order $\dsty \max\left(\frac{1}{\rho_2},s\right)$  in $\epsilon$ uniformly in the variable $z$ belonging to any compact subset of $D_0$.
 \end{theorem}

\bibliographystyle{ieeetr}

\bibliography{WKBExact3}

\end{document}